\documentclass{article}

\usepackage{ntheorem}

\usepackage[toc]{appendix}
\usepackage{bm}
\usepackage{scalerel}
\usepackage{mathalfa}
\usepackage{amsmath}
\usepackage{amscd}
\usepackage{amssymb}
\usepackage{theorem}
\usepackage{mathabx}

\usepackage{rotating}
\usepackage{stmaryrd}
\usepackage{mathdots}
\usepackage{hyperref}
\usepackage{caption}
\usepackage{xcolor}
\usepackage{mathtools}
\usepackage{color}

\usepackage{mathrsfs}
\usepackage[bbgreekl]{mathbbol}
\usepackage{bbm}

\usepackage[bbgreekl]{mathbbol}
\usepackage{tikz-cd}
\usepackage{graphicx}
\usepackage{leftindex}

\usepackage{multicol}
\usepackage[inline]{enumitem}

\makeatletter
\DeclareMathSizes{\@xpt}{\@xpt}{6}{4.5}
\makeatother

\DeclareMathAlphabet{\matheulerfrak}{U}{euf}{t}{n}

\DeclareMathSymbol{\mh}{\mathord}{operators}{`\-}

\newcommand{\bH}{\mbf{H}}

\newsavebox{\deltabox}

\usepackage{relsize}

\newcommand{\fqs}{flat quasi-smooth }
\newcommand{\qs}{quasi-smooth }
\newcommand{\bL}{\mbf{L}}
\newcommand{\cL}{\mathcal{L}}

\newcommand{\DDelta}{\mathbb{\Delta}}

\newcommand{\wh}{\widehat}

\newcommand{\mbf}{\mathbf}

\newcommand{\Id}{\operatorname{Id}}

\newcommand{\XX}{{\mathfrak X}}

\renewcommand{\SS}{{\mathfrak S}}

\newcommand{\qQ}{\mathcal{Q}}

\newcommand{\YY}{{\mathfrak Y}}

\newcommand{\wtil}{\widetilde}

\renewcommand{\L}{{\mathscr L}}

\newcommand{\VV}{{\mathfrak V}}

\newcommand{\nN}{\mathcal{N}}

\newcommand{\Ss}{{\mathcal S}}

\newcommand{\ii}{\mbf{i}}

\newcommand{\aaa}{{\mathbb A}}

\newcommand{\nn}{{\mathbb N}}

\newcommand{\bs}{\boldsymbol}

\newcommand{\sD}{\mbf{s}D}

\newcommand{\tMC}{\textnormal{MC}}

\renewcommand{\O}{{\mathscr O}}

\newcommand{\vV}{{\mathscr V}}

\newcommand{\gG}{{\mathcal G}}

\newcommand{\g}{\mathfrak{g}}

\newcommand{\uU}{{\mathscr U}}

\newcommand{\D}{\nabla}

\newcommand{\resto}{{\,|\,}}

\newcommand{\ups}{\upsilon}

\renewcommand{\Im}{\mathop{\rm Im}}

\newcommand{\rank}{\mathop{\rm rank}\nolimits}

\newcommand{\red}{{\mathop{\rm red}\nolimits}}

\newcommand{\dR}{\textnormal{dR}}
\newcommand{\ob}{\mathop{\rm ob}}

\newcommand{\MC}{\mathop{\mathcal{MC}}\nolimits}

\newcommand{\eff}{\mathrm{eff} }

\newcommand{\dVol}{\textnormal{dVol}}

\newcommand{\ul}{\underline}

\newcommand{\0}{ {\mathbf{ 0 }} }

\newcommand{\longiso}{\stackrel{\textstyle\sim}{\longrightarrow}}
\newcommand{\iso}{\stackrel{\sim}{\rightarrow}}

\newcommand{\doublearrowstack}[2]%
{{{{\scriptstyle#1}\atop{\textstyle\longrightarrow}}\atop{{\textstyle\longrightarrow}\atop{\scriptstyle#2}}}}
\newcommand{\rightleftarrowstack}[2]%
{{{{\scriptstyle#1}\atop{\textstyle\longrightarrow}}\atop{{\textstyle\longleftarrow}\atop{\scriptstyle#2}}}}
\newcommand{\leftrightarrowstack}[2]%
{{{{\scriptstyle#1}\atop{\textstyle\longleftarrow}}\atop{{\textstyle\longrightarrow}\atop{\scriptstyle#2}}}}

\usetikzlibrary{arrows, matrix}

\newtheorem{thm}{Theorem}[section]
\newtheorem{cor}[thm]{Corollary}
\newtheorem{lem}[thm]{Lemma}
\newtheorem{prop}[thm]{Proposition}

\newtheorem*{theorem*}{Theorem.}
\newtheorem*{proposition*}{Proposition.}
\newtheorem*{definition*}{Definition.}
\theorembodyfont{\upshape}
\newtheorem{defn}[thm]{Definition}

\newtheorem{cons}[thm]{Construction}

\newtheorem{rmk}[thm]{Remark}

\newtheorem{ex}[thm]{Example}
\newtheorem{notation}[thm]{Notation}

{\nolinebreak $\Box$ \end{trivlist}}

\newenvironment{proof}{\begin{trivlist}\item[]{\sc Proof.}}%
{\nolinebreak\hfill $\Box$ \end{trivlist}}

\newcommand{\calS}{\mathcal{S}}
\newcommand{\calD}{\mathcal{D}}

\newcommand{\Cpx}{\mathbb C}

\newcommand{\CC}{{\Cpx}}

\newcommand{\ZZ}{\mathbb Z}

\newcommand{\ep}{\epsilon}

\DeclareMathOperator\id{id}

\newcommand{\im}{\mathop{\rm im}\nolimits}
\newcommand{\rk}{\mathop{\rm rk}\nolimits}

\newcommand{\spec}{\mathop{\rm Spec}\nolimits}

\newcommand{\Sym}{\mathop{\rm Sym}\nolimits}

\newcommand{\contract}{{\,\lrcorner\,}}

\newcommand{\opname}{\operatorname}
\newcommand{\Crit}{\opname{Crit}}
\newcommand{\dCrit}{\opname{dCrit}}

\newcommand{\ed}{\opname{emb} } 
\newcommand{\noprint}[1]{}

\newcommand{\CE}{\textnormal{CE}}

\newcommand{\git}{\!\sslash\!}

\newcommand{\mdl}{ \mathbin{\leadsto}}

\author{Elliot Cheung}
\date{}

\title{{\bf \LARGE Quantum BV  $\cL_{\infty}$-algebras~I:}  \\
	{\bf \Large  {Derived geometric foundations}} }

\begin{document}
	\sloppy
	
	\maketitle
	\begin{abstract}
		We introduce the concept of a {\em quantum BV $\cL_{\infty}$-algebra} and study fundamental properties. In particular, we investigate homotopy Lie theoretic structures that naturally arise in the context of Chern-Simons theory. Of note, are the notions of \emph{homotopy BV data} and of a \emph{BV orientation}. The sequel of this paper will involve the direct application of these constructions to the setting of Chern-Simons theory.
	\end{abstract}
	\tableofcontents	
	
	\section*{Introduction}
	
	This is the first paper in a series of papers outlining the idea of a certain \emph{homotopical discretization} of Chern-Simons theory. \\
	
	A ubiquitous theme that often arises in moduli theoretic contexts is that local dg geometry is always governed by $L_{\infty}$-algebras. It is often the case that a description of derived deformation theory begins with quoting a certain principle: \\
	
	\begin{minipage}{30em}
		``Over a field of characteristic $0$, the formal neighbourhood of a moduli space or stack, is governed by a homotopy Lie algebra" (see: \cite{calaque2019formalmoduliproblemsformal}, \cite{Manetti2002})
	\end{minipage}
	\bigskip
	
	\noindent In brief, the idea is to somehow replace a singular space $X$ with the data of a resolution $\XX$ of $X$ by smooth spaces. The resolution, or the differential graded manifold, at a point $x \in X$ has a differential graded tangent space $T_{x} ^{\bullet} \XX$ which is a chain complex. In other words -- for a classical geometric space, the tangent space at a point is a vector space; for a differential graded space, the (dg) tangent space at a point is a chain complex. However, even more is true: the differential graded tangent space $T_{x}^{\bullet}\XX$ can be endowed with a natural homotopy Lie algebra structure, making it into an $L_{\infty}$-algebra that we will denote as $L^{\bullet}_{x}$. The base point $\0_{x} \in T_{x}X = L^{1}_{x}$  along with this $L_{\infty}$-algebra $L^{\bullet}_{x}$ describes the formal germ of the the dg analytic space $\XX$ at $x$. More specifically, every $L_{\infty}$-algebra encodes a certain formal analytic function, known as the \emph{Maurer-Cartan} function $F_{\tMC}$. In this present context, the equation defined by $F_{\tMC} = 0$ gives a derived resolution of a formal neighbourhood of $X$ at $x$, equivalent to the derived resolution provided by $\XX$. \\
	
	When an $L_{\infty}$-algebra $L^{\bullet}$ satisfies a certain boundedness condition on its structure operations (i.e. the brackets and higher brackets), its Maurer-Cartan function is guaranteed to converge in a polydisc of non-zero polyradius of convergence around the origin in $L^{1}$. In this case, the $L_{\infty}$-algebra $L^{\bullet}$ does not only describe a formal germ of a dg analytic space, but encodes a (convergent) germ of a marked basepoint in a dg analytic space. That is, such $L_{\infty}$-algebras describe (non-formal) local neighbourhoods of dg analytic spaces at marked basepoints. In this paper, we will work with such $L_{\infty}$-algebras and refer to them as \emph{analytic $L_{\infty}$-algebras}. A local analytic $L_{\infty}$-algebra is defined by fixing an open domain of convergence $U \subset L^{1}$, which is not necessarily marked, of an analytic $L_{\infty}$-algebra. Morphisms between local analytic $L_{\infty}$-algebras must be analytic on convergence domains. We refer to the category of local analytic $L_{\infty}$ algebras, as the category of $\cL_{\infty}$-algebras.  \\
	
	The author's primary interest in this structure comes out of an overall project involving the understanding of a ``homotopical discretization" of Chern-Simons theory. It is well-known that one can represent flat connections on a trivial bundle over a manifold $M$ as solutions to an appropriate Maurer-Cartan equation associated to an infinite dimensional differential graded Lie algebra $L_{\Omega} = \big( \Omega^{\bullet}(M,\g),d, [,] \big) $.  In this context, we may consider $L_{\Omega}$ to be an $\cL_{\infty}$-algebra, which is  (local) analytic in a infinte dimensional Fr\'echet sense. A discretization of $L_{\Omega}$ is defined by the data of an analytic $L_{\infty}$-homotopy retract context: 
\begin{equation} \label{homotopytransferbasic} 
L  \rightleftarrowstack{\mbf{I}}{\mbf{P} }    L_{\Omega} \circlearrowleft{ \bs{\eta}_{\Omega}  }
\end{equation}
	
	such that 
	\begin{enumerate}
		\item{$\mbf{I}$, $\mbf{P}$ are $L_{\infty}$-morphisms, such that $\mbf{P} \circ \mbf{I} = \Id_{L}$ and $\bs{\eta}_{\Omega}$ is an $L_{\infty}$-homotopy between $\Id_{L_{\Omega}}$ and $\mbf{I} \circ \mbf{P}$ (in particular, $\mbf{I}$ and $\mbf{P}$ are $L_{\infty}$ weak equivalences). } 
		
		\item{ $(L,U)$ is a finite dimensional $\cL_{\infty}$-algebra, all the maps $\mbf{I}$, $\mbf{P}$ and $\bs{\eta}_{\Omega}$ are required to converge on $U$ and $U_{\Omega}$, where $\mbf{P}(U_{\Omega}) = U \subset L^{1}$ }
	\end{enumerate}
	
In the above, we refer to $(L_{\Omega}, U_{\Omega})$ as a \emph{dgla-model} for the $\cL_{\infty}$-algebra $(L,U)$. In general, we refer to an $\cL_{\infty}$-algebra $(L,U)$ with a fixed dgla-model $(L_{\Omega},U_{\Omega})$ (i.e. a dgla $L_{\Omega}$ with a homotopy retract context as in \ref{homotopytransferbasic}), as an $\cL_{\infty}^{\Omega}$-algebra. The author has found this to be a useful definition, as the dgla-model $L_{\Omega}$ provides one with distinguished representations of homotopical structures one may define for an $\cL_{\infty}$-algebra $(L,U)$. In the forthcoming example of Chern-Simons theory, we will see that examples of this include: representations of ``large gauge equivalences", $(-1)$-shifted symplectic structure and BV formalism data defined from the Chern-Simons action functional and its companion $1$-form. One may think of the dgla-model of an $\cL_{\infty}$-algebra $(L,U)$ as a way to define a \emph{homotopy gauge action} on $(L,U)$.
 
\section{Dg-geometry and quasi-smooth spaces}

Suppose that $U$ is an open domain in some $\CC^{n}$, and $I \hookrightarrow \O_{D}$ is a coherent ideal sheaf of analytic functions on $D$. Then, we let $X = Z( I )$ be the vanishing locus of $I$. 

\begin{defn}
	Where $X \subset U$ is endowed with the subspace topology, we define $( X , \O_{X})$ to be the ringed space defined by the sheaf $\O_{U}/I \resto_{X}$. An \emph{analytic model space} is a ringed space of this form.
\end{defn}

\begin{defn}
	We say that $( X , \O_{X} )$ is an analytic space, if  $( X , \O_{X} )$ is a $\CC$-ringed space, such that for any $p \in X$ there exists an open neighbourhood $U$ containing $p$ such that $( U , \O_{X} \resto_{U} )$ is isomorphic to an analytic model space, as $\CC$-ringed spaces. 
\end{defn}

Of course, an analytic model space is akin to an affine variety in the language of algebraic geometry, and an analytic space like a scheme in that it is locally presented by ``affines". However, in the realm of analytic geometry, one considers an analytic space to be ``truly affine" by analogy, if it is a \emph{Stein space}. We can restrict our analytic model spaces to be modelled over polydiscs, and these will be Stein spaces. For a good overview on analytic spaces, see \cite{HansGrauert1984} and \cite{Greuel2006}.\\

Clearly, an analytic space can be very singular. If $X$ is smooth, then as a complex manifold, $X$ is always locally isomorphic to a polydisc in some $\CC^{n}$. In particular, near a point $p \in X$, we can identify an open neighbourhood $U_{p} \subset X$  containing $p$ with open neighbourhood inside of the tangent space $T_{p}X$. That is, deformations of $X$ near $p$ are ``unobstructed".  One perspective is that a singular space has a more complicated deformation theory, where the tangent sheaf of a singular space is no longer locally free and its fibers may vary in dimension at different points. \\

One starting point into the realm of derived geometry is understanding how one might ``resolve" singular analytic spaces -- and in particular, how can one make sense of ``resolving" a tangent sheaf that is not locally free. A more approachable model to understanding how this may look is provided by a class of dg analytic spaces given by \qs spaces. Roughly speaking, these are differential graded analytic spaces whose classical truncations are locally presented as a vanishing locus of some analytic equations in an ambient smooth complex analytic space (i.e. a complex manifold) -- such that the data of the defining equations globally assemble as a section of a vector bundle over a finite dimensional smooth analytic space $M$.

\begin{defn} \label{qsspace}
	A quasi-smooth analytic space (or simply, just a ``\qs space"), $\XX = (M,E, \lambda)$ is the data of: 
	\begin{itemize}
		\item{ A \emph{finite dimensional}, analytic manifold $M$}
		\item{A vector bundle $E$ over $M$}
		\item{An analytic section $\lambda: M \rightarrow E$ of the bundle $E$ }
	\end{itemize}
	The classical locus, or classical truncation, of a quasi-smooth analytic space $\XX$ is the analytic space $Z(\lambda) \hookrightarrow M$, the vanishing locus of the section $\lambda$. We will denote this as $\tau^{0}(\XX)$, or $X$. 
\end{defn}

A quasi-smooth analytic space as defined above has a differential graded structure sheaf. This is a sheaf of cdgas over $M$ that can be considered as a homological resolution of the analytic structure sheaf $\O_{X}$. Before we state the definition, note that the section $\lambda$ of the bundle $E$ over $M$ defines a morphism of sheaves (on $M$), $q_{\lambda} := \lambda^{*}: E^{\vee} \rightarrow \O_{M}$. Then, the structure sheaf of $\O_{\XX}$ is defined to be the associated ``sheaffy" Kozsul complex associated to the morphism of $\O_{M}$ modules given by $q_{\lambda}$:

\begin{defn}
	The structure sheaf of the analytic \qs space $\XX = (M,E, \lambda)$ is defined to be the sheaf of cdgas, over $M$:
	
	\[
	\O^{\bullet}_{\XX} := \bigg(   \cdots \rightarrow \bigwedge^{k} E^{\vee} \rightarrow \cdots \rightarrow E^{\vee} \xrightarrow{q_{\lambda}} \O_{M} \bigg)
	\]
\end{defn}
Note that $H^{0}( \O_{\XX}^{\bullet} ) = \O_{X}$, the analytic structure sheaf of $X = \tau^{0}(\XX)$. Indeed, $\tau^{0}(\XX)$ is defined to the vanishing locus $Z(\lambda)$. This analytic subscheme of $M$ is defined via the sheaf of ideals $I_{X,M} := \im( E^{\vee} \xrightarrow{q_{\lambda}} \O_{M} ) \hookrightarrow \O_{M}$

\begin{defn} 
	A morphism $\Phi = (\phi, \phi^{\#})$ between analytic \qs spaces $\XX = (M,E,\lambda)$ and $\YY = (N,F, \zeta)$ is defined by:
	
	\begin{itemize}
		\item{ an analytic map $\phi: M \rightarrow N$ }
		\item{and a morphism of bundles, $\phi^{\#}: E \rightarrow F$, such that $\phi^{*} \zeta = \phi^{\#} \circ\lambda$.}
	\end{itemize}
	If $\Phi$ is a morphism between analytic \qs spaces $\XX$ and $\YY$ as above, such that $\phi: M \iso N$ is a biholomorphism, then we say that $\Phi$ is a \emph{dg biholomorphism}. Clearly, a dg biholomorphism is an isomorphism of analytic \qs spaces. \\

	The compatibility condition for $\phi^{\#}$ with the sections $\lambda$ and $\zeta$ make it so that $\phi |_{ \tau^{0}(\XX) } : \tau^{0}(\XX) \rightarrow \tau^{0}(\YY)$ -- that is, the map $\phi$ between the ambient manifolds (or ``bodies") restricts to a map between corresponding classical loci.
	
\end{defn}

The compatibility condition for $\phi^{\#}$ means that the process of taking a classical truncation via $\tau^{0}$ is functorial. This gives us a functor from the category of analytic \qs spaces to analytic spaces.  \\

Sometimes we may use the same notation $\phi$ to refer to both a morphism $\Phi$ of analytic \qs spaces and the associated map between the respective bodies. Ideally, this will only be used when there is no room for confusion. \\

A morphism $\Phi: \XX \rightarrow \YY$ between analytic \qs spaces induces a morphism of dg structure sheaves $\Phi^{*}: \O_{\YY} \rightarrow \phi_{*} \O_{\XX}$ defined in the following way:

\begin{defn}
	Suppose that $\XX = (M,E,\lambda)$ and $\YY = (N,F, \zeta)$ are analytic \qs spaces. Then, given a morphism $\Phi: \XX \rightarrow \YY$ between analytic \qs spaces, we define the induced morphism between dg structure sheaves to be the morphism of dg sheaves $\Phi^{*}: \O_{\YY} \rightarrow \phi_{*} \O_{\XX}$ given by the following morphism of complexes (of sheaves over $N$):
	
	\[
	\begin{tikzcd}
		\cdots  & \bigwedge^{k} F^{\vee} \arrow[d] \arrow[r] & \cdots  \arrow[r]  \arrow[d] & F^{\vee}   \arrow[d, "(\phi^{\#})^{*}"] \arrow[r, "q_{\zeta}"] & \O_{N} \arrow[d, "\phi^{*}"]  \\
		\cdots  & \bigwedge^{k} \phi_{*}E^{\vee}  \arrow[r] & \cdots \arrow[r] & \phi_{*}E^{\vee} \arrow[r, "\phi_{*}q_{\lambda}" swap] & \phi_{*}\O_{M}  \\
	\end{tikzcd}
	\]
	
	Note that this is indeed a morphism of (sheaves of) complexes, as the morphism of analytic \qs spaces $\Phi = (\phi, \phi^{\#})$ is required to be compatible with the sections $\lambda$ and $\zeta$ respectively. 
\end{defn}

\begin{rmk}
	By adjunction, the morphism of dg structure sheaves $\Phi^{*}: \O_{\YY} \rightarrow \phi_{*}\O_{\XX}$ can be equivalently described as a morphism of dg sheaves over $M$ defined as $\Phi^{\natural}: \phi^{-1}\O_{\YY} \rightarrow \O_{\XX}$.
\end{rmk}

One can think of an analytic \qs space as a $2$-step resolution of a possibly singular space $X$.  We enlarge the class of ``spaces" to include diagrams of smooth objects, $M$ and $E$ connected by the data of a map (a section) $\lambda: M \rightarrow E$.  This may perhaps be reminiscient of taking a resolution by way of a chain complex in homological algebra.  If $\tau^{0}\XX = X$, then we say that $X$ is the classical truncation of the quasi-smooth analytic space $\XX$, or that $\XX$ is a \emph{derived enhancement} of $X$. \\

Below are examples of how analytic \qs spaces may be viewed as a ``resolution" of a singular analytic spaces.

\begin{ex} 
	Let $X = Z( x_{1}^{3} - x_{1}x_{2} ) \subseteq \aaa_{k}^{2}$. Define the map $F(x,y): \aaa^{2} \rightarrow \aaa^{1}$ by the equation $F(x,y) =  x_{1}^{3} - x_{1}x_{2}  $. Then, we take $E = \aaa^{2} \times \aaa^{1}$ as the trivial $\aaa^{1}$ bundle over $\aaa^{2}$. Then, the map $F$ defines a section $\lambda(x,y) = (x,y) \times (F(x,y))$ of the bundle $E$ over $M = \aaa^{2}$.
\end{ex}

\begin{ex}
	Let consider the analytic map $F: \CC^{2} \rightarrow \CC^{2}$ defined by $F(x,y) = (x^{2} - y^{3}, xy)$. The Jacobian is given by
	
	\[
	J_{F} = 
	\begin{bmatrix}
		2x & -3y^{2} \\
		y & x 
	\end{bmatrix}
	\]
	
	Thus, one can see that $F$ is singular at the origin. Regardless, one can take the idea above in example 1.4 and define an analytic \qs space with the data $M = \CC^{2}$, E = $\CC^{2} \times \CC^{2}$, $\lambda(x,y) = (x,y ; x^2 - y^{3} , xy)$. 
\end{ex}

By including the data of a bundle in our notion of spaces, tangent spaces of dg-spaces become $2$-term complexes.

\begin{defn}
	Given a quasi-smooth dg-space $\XX = (M,E,\lambda)$, we define the dg-tangent space at $\mu \in \tau^{0}(\XX)$ to be:
	
	\[
	T^{\bullet}_{\mu} \XX = \big[ T_{\mu} M \xrightarrow{D_{\mu} \lambda} E_{\mu} \big]
	\]
\end{defn}

\begin{rmk}\hfill
	\begin{enumerate}
		\item{
			Note that in the above definition, we are considering the derivative $D_{\mu} \lambda$ of a section at the point $\mu$.  As $\mu \in Z(\lambda)$, there is a canonically defined derivative $D_{\mu} \lambda: T_{\mu}M \rightarrow  E_{\mu}$ that does not depend on a choice of a connection. }
		
		\item{With this, we can view $M$ as providing a space of ``ambient deformations" of $X$, and $E$ as providing an obstruction bundle. A tangent vector $v$ at $\mu \in M$ is a tangent vector of $\mu \in X$ if $v \in \ker(D_{\mu} \lambda)$.}
		
	\end{enumerate}
\end{rmk}

With the notion of dg-tangent spaces for quasi-smooth analytic spaces, we can define a notion of weak equivalence, or quasi-isomorphism between quasi-smooth analytic spaces. These are morphisms $\Phi: \XX \rightarrow \YY$ that are \'etale or local biholomorphisms in the dg sense. Roughly speaking, for every point $\mu \in \tau^{0}(\XX)$, the induced tangent map $T^{\bullet}_{\mu}\XX \rightarrow T^{\bullet}_{\phi(\mu)} \YY$ should be a quasi-isomorphism -- which is the relevant notion of (weak) equivalence between chain complexes.

\begin{defn}
	A morphism $\Phi = (\phi, \phi^{\#})$ between quasi-smooth dg-spaces $\XX = (M,E,\lambda)$ and $\YY = (N,F, \zeta)$ is considered a \emph{quasi-isomorphism} if:
	\begin{itemize}
		\item{ $\phi |_{\tau^{0}(\XX)} : \tau^{0}(\XX) \rightarrow \tau^{0}(\YY)$ is an isomorphism between classical truncations (as analytic spaces) }
		
		\item{ for any $\mu \in \tau^{0}(\XX)$, we have that the induced commutative square is a quasi-isomorphism of the rows:
			
			\begin{equation} \label{tangentmap}
				\begin{tikzcd}
					T_{\mu} M \arrow[r, "D_{\mu} \lambda"  ] \arrow[d, "D_{\mu} \phi"] & E |_{\mu} \arrow[d, "\phi^{\#}|_{\mu}"]  \\
					T_{\phi(\mu)} N \arrow[r, "D_{\phi(\mu)} \zeta" swap]  & F |_{\phi(\mu)} \\
				\end{tikzcd}
			\end{equation}
		}
		
	\end{itemize}
	Of course, the two rows in the above square are the dg tangent spaces $T_{\mu}^{\bullet}\XX$ and $T_{\phi(\mu)} \YY$ respectively. The induced square can be taken as the definition of the induced tangent map $D_{\mu} \phi : T_{\mu}^{\bullet} \XX \rightarrow T_{\phi(\mu)} \YY $.
\end{defn}

\paragraph{Restricting a \qs analytic space} 

\begin{defn}
	\item{ Suppose that $\XX = (M,E,\lambda)$ is a quasi-smooth analytic space.  If $U_{M} \hookrightarrow M$ is an open subset of $M$, then we can define the restriction $\XX |_{U_{M}} $ of $\XX$ to the open subset $U_{M}$ by defining:
		\[
		\XX |_{U_{M}} := ( U_{M}, E |_{U_{M}} , \lambda |_{U_{M}}  )
		\] }
\end{defn}
More generally, suppose that $\phi: N \rightarrow M$ is a morphism of analytic manifolds, and $\XX = (M,E,\lambda)$ is an analytic \qs space. Then the morphism $\phi$ induces via pullback, an analytic \qs structure we denote as $\phi^{*} \XX$, defined to be:
\[
\phi^{*}\XX = ( N , \phi^{*} E, \phi^{*} \lambda)
\]

\paragraph{ Immersions of analytic \qs spaces }

\begin{defn}
	
	Suppose that we have two analytic \qs spaces $\XX = (M,E, \lambda)$ and $\YY = (N,F,\zeta)$, and  a morphism of analytic \qs spaces $\iota: \YY \rightarrow \XX$. $\iota: \YY \hookrightarrow \XX$ is an immersion if:
	\begin{enumerate}
		\item{ $\iota: N \hookrightarrow M$ is an embedding of smooth analytic spaces }
		
		\item{ We have an isomorphism $\iota^{*}E \cong F$ }
		
	\end{enumerate}
	
	\indent We say that $\YY$ is \emph{open} in $\XX$ if $\iota(N) \subseteq M$ is an open subset, so that $\XX |_{\iota(N)} \cong \YY$
\end{defn}

For an open subset $U_{M} \subseteq M$, we clearly have an immersion of analytic \qs spaces $\XX|_{U_{M}} \hookrightarrow \XX$. \\

We may say that an analytic \qs space $\XX$ is \emph{pointed}, if we also specify a base point $\mu \in X  = \tau^{0}\XX$. 
In general, we can use (pointed) analytic \qs spaces to model analytic mapping germs by following the two examples above for any analytic mapping germ $F: (\CC^{n}, \0) \rightarrow (\CC^{m},\0)$.

\begin{defn}(Local analytic \qs spaces) \label{localdgspace}\\
	Let $V$ be a finite dimensional complex vector space. \\
	A local analytic \qs space is a \emph{pointed} analytic \qs space of the form $\uU = ( U, E_{U}, \lambda)$ with base point being the origin $\0 \in V$, with:
	\begin{itemize}
		\item{
			$U \subseteq V$ is a polydisc of the origin. }
		\item{$E_{U} $ is a trivial bundle over $U$ }
		\item{$\lambda: U \rightarrow E$ is an analytic map between $U$ and the vector space $E$ that vanishes at $\0$.}
	\end{itemize}
	
	Note that, where $E$ is a vector space, we use the notation $E_{U}$ to denote the trivial bundle over $U$ with fiber $E$. As $E_{U}$ is a trivial bundle over $U$, a section is equivalent to specifying an analytic mapping $\lambda: U \rightarrow E$. As $\uU = (U, E_{U}, \lambda)$ is completely determined by  $U \subset V$ and the  analytic mapping $\lambda: U \rightarrow E$, we will use the notation $\uU = ( U, \lambda: U \rightarrow E)$ to specify a local analytic \qs space.\\

	To emphasize the base point of a local analytic \qs space $\uU$, we may use the notation $(\uU,\0)$ when appropriate. For local analytic \qs spaces $(\uU,\0)$ and $(\vV,\0)$, a \emph{pointed} morphism of local analytic \qs spaces $\Phi: (\uU,\0) \rightarrow (\vV,\0)$ is a morphism of analytic \qs spaces that is compatible with the base points (i.e., $\phi(\0) = \0$).\\
	
\end{defn}

One can view a local analytic \qs space as a local model for analytic \qs spaces:  for any analytic \qs space, one can choose a point $\mu \in \tau^{0}\XX$ and restrict the bundle $E$ to a trivializing chart of $M$ around $\mu$ and produce a local analytic \qs space. We will expand on this more later, when we define the notion of a dg-chart.\\

A local analytic \qs space well has a defined notion of taking the derivative of the analytic map $\lambda: U \rightarrow E$, producing another analytic map $\cal{D} \lambda: U \rightarrow E$. Of course, for a general analytic \qs space $\XX$, we can globalize this if we equipped the bundle $E$ with a \emph{holomorphic} connection $\D$.

\begin{defn}(\fqs spaces)
	\begin{enumerate}
		\item{
			We say that an analytic \qs space $\XX = (M,E,\lambda)$ is an analytic \qs space \emph{with connection}, if $E$  is equipped with a \emph{holomorphic} connection $\D$. An analytic quasi-smooth space is \emph{flat} if the holomorphic connection on $E$ is a flat connection. 
		}
		
		\item{  If $\XX = ( M,E, \lambda, \D^{E})$ and $\YY = (N,F, \zeta, \D^{F})$ are analytic \qs spaces with connection, then a flat morphism $\Phi = (\phi, \phi^{\#})$ of analytic \qs spaces with connection is a morphism of analytic \qs spaces $\Phi: \XX \rightarrow \YY$
			such that $\D \phi^{\#} = 0$. That is, for any section $s : M \rightarrow E$, we have that $\D \phi^{\#} s =  \phi^{\#} \D s$.
		}
		
	\end{enumerate}
\end{defn}

\begin{ex}
	Continuing with example 1.4 above, we can take the differential $dF = (3x_{1}^{2} - x_{2})dx_{1} - x_{1}dx_{2}  $ which is a section of the cotangent bundle $\Omega_{M} \cong \aaa^{2} \times \CC^{2}$. This defines an analytic \qs space, whose classical truncation is given by the vanishing of $dF$ -- that is, the classical truncation is given by $\Crit(F) \subset \aaa^{2}$. 
\end{ex}

This is an instance of one of the most important examples of analytic \qs spaces that will have a central role in this article. They are known as \emph{derived critical loci}. 

\begin{defn}
	Suppose that $M$ is a smooth analytic space, and $S: M \rightarrow \CC$ is an analytic function. Then, let $ X = \Crit(S) = Z(dS)$, where $dS$ is a section of the cotangent bundle $\Omega_{M}$. Then, the \emph{derived critical locus} of  $S: M \rightarrow \CC$ is written as $\dCrit(S)$, and defined as:
	\[
	\dCrit(S) = (M,\Omega_{M}, dS)
	\]
	We see that by definition $\tau^{0} \dCrit(S) = \Crit(S)$, so that the quasi-smooth analytic space $\dCrit(S)$ is a derived enhancement of the critical locus of $S$ in $M$. 
\end{defn}

\begin{defn}
	We say that an analytic \qs space $\YY = (N,F, \zeta)$ is a \emph{contractible} analytic \qs space, if the map $\YY \rightarrow \star$ is a quasi-isomorphism of dg-spaces, where $\star$ is a singleton point space. 
	
	\begin{rmk}
		Recall that an analytic manifold $M$ can be considered an analytic \qs space with the data $(M,0,0)$. Thus, $\star = (\star, 0 ,0 )$ as an analytic \qs space. Then, a quasi-isomorphism from $\YY \rightarrow \star$ means that $\tau^{0}( \YY ) = \nu = Z(\zeta)$: the section $\zeta$ vanishes at a single point $\nu  \in N$. Furthermore, the dg-tangent space at the unique point $\nu$ 
		
		\[
		T_{\nu } \YY = \big[ T_{\nu } N \xrightarrow{ D_{\nu } \zeta } F_{\nu } \big]
		\]
		\noindent has zero cohomology. In otherwords, the differential $D_{\nu } \zeta$ is an isomorphism of vector spaces. 
	\end{rmk}

\end{defn}

\begin{lem}
	If $\YY = (N, F,\zeta)$ is a contractible analytic \qs space, then $\YY \cong (U_{\nu}, (F_{\nu})_{U_{\nu}}, D_{\nu}\zeta)$, where $\tau^{0}\YY = \{ \nu \}$, $U_{\nu} \subset T_{\nu}N$ is an open subset of the origin, and $(F_{\nu})_{U_{\nu}}$ is the trivial $F_{\nu}$ bundle over $U_{\nu}$
\end{lem}

\begin{proof}
	
	As $N$ is a manifold, we can find an open set $U_{\nu}$ of the origin in $T_{\nu}N$ small enough so that $U_{\nu}$ is contained in a chart of $N$ containing $\nu$. We can also suppose that the bundle $F$ is trivialized over the open set $U_{\nu}$. Therefore, $F|_{U_{\nu}} \cong U_{\nu} \times F_{\nu}$. Therefore, by forming the quasi-smooth analytic space $(U_{\nu}, (F_{\nu})_{U_{\nu}}, D_{\nu} \zeta)$, we see that we have a quasi-isomorphism  $\Phi: (U_{\nu}, (F_{\nu})_{U_{\nu}}, D_{\nu} \zeta) \rightarrow \YY$. 
\end{proof}

If $\YY = (N, F, \zeta)$ is a contractible analytic \qs space of the form given in lemma 1.21, then we say that $\YY$ is a \emph{linear contractible} analytic \qs space. Note that while the property of being contractible is invariant under quasi-isomorphisms, the propery of being linear contractible is \emph{not}.\\

Suppose that $\XX = (M,E,\lambda)$ is a dg-space. As $M$ is a complex manifold, $M$ is equipped with a holomorphic atlas. Therefore, for each $\mu \in M$ there is an open neighbourhood $U_{M}^{\mu} \subset M$ of $\mu$, such that $\phi_{\mu}: U_{\mu} \hookrightarrow M$ is a chart for $M$ with $\phi_{\mu}(U_{\mu}) = U_{M}^{\mu}$. We take the convention that $U_{\mu}$ is an open neighbourhood of the origin in a vector space $V_{\mu}$, with $\phi_{\mu}(\0_{\mu}) = \mu \in M$. \\

Therefore, $\uU_{\mu} := \phi_{\mu}^{*}\XX $ is a local analytic \qs space with an open immersion $\Phi_{\mu}: \uU_{\mu} \hookrightarrow \XX$. As $M$ is a manifold, there is an atlas given by a collection of charts $(U_{M}^{\mu}, \phi_{\mu})$ such that $\bigcup\limits_{\mu} U_{M}^{\mu} = M$. Thus, there is a corresponding collection of local analytic \qs spaces $\uU_{\mu}$ with $\Phi_{\mu}: \uU_{\mu} \xrightarrow{\sim} \XX|_{U_{M}^{\mu}}$ where $\{ U_{M}^{\mu} \}$ covers $M$.\\

\begin{defn}
	Let $\XX = (M,E,\lambda)$ be an analytic \qs space, with a given atlas $\{ (U_{\mu}, \phi_{\mu}) \}$ for $M$.
	\begin{itemize}
		
		\item{ A dg chart of $\XX$ is a local analytic \qs space $\uU_{\mu} = (U_{\mu}, \lambda_{U_{\mu}})$, with an isomorphism of analytic \qs spaces $\Phi_{\mu}: \uU_{\mu} \xrightarrow{\sim} \XX|_{\phi_{\mu}(U_{\mu})}$.}
		\item{A dg atlas is a collection of dg charts $\Phi_{\mu}: \uU_{\mu} \xrightarrow{\sim} \XX|_{\phi(U_{\mu})}$ such that $\bigcup\limits_{\mu} \phi_{\mu}(U_{\mu}) = M$.}
		
	\end{itemize}  
\end{defn}

Suppose that $\Phi_{\mu}: \uU_{\mu} \xrightarrow{\sim} \XX|_{U_{M}^{\mu}}$ is a dg chart around $\mu$. Then, as $\uU_{\mu} = ( \DDelta^{n}(\0, \rho_{\mu} ), \lambda_{\mu})$ is a local analytic \qs space, it corresponds to an analytic mapping germ $\lambda_{\mu}: (\DDelta^{n}(\0, \rho_{\mu} ), \0_{\CC^{n}} ) \rightarrow (\CC^{r_{\mu} }, \0_{\CC^{r_{\mu}}})$ which is embedded in the smooth manifold $M$. Of course, up to isomorphism of germs, the choice of embedding manifold $M$ or target $E$ are by no means unique. The notion of quasi-isomorphism between analytic \qs spaces allows us to account for the possible ambiguity of the choices of $M$ and $E$. 

\begin{defn}(quasi dg charts) \\
	Suppose that $\XX$ is an analytic \qs space with a point $\mu \in \tau^{0}\XX$. We say that a local analytic \qs space $\uU'_{\mu} = (U_{\mu}', \lambda'_{\mu})$ with an injective morphism $\Phi'_{\mu}: \uU'_{\mu} \hookrightarrow \XX$ is a quasi dg chart around $\mu$ if there exists an open neighbourhood $U_{M}^{\mu} \subseteq M$ such that $\Phi'_{\mu}: \uU'_{\mu} \hookrightarrow \XX|_{U_{M}^{\mu}}$ is a quasi-isomorphism of analytic \qs spaces. 
\end{defn}

\noindent Note that if $\Phi'_{\mu}: \uU_{\mu}' \hookrightarrow \XX$ is a quasi dg chart, the dimension $m$ of the polydisc $U_{\mu}'$ is less than or equal or $n$.  On the other hand, if $\Phi'_{\mu}: \uU'_{\mu} \hookrightarrow \XX|_{U_{M}^{\mu}}$ is a quasi dg chart, then
the condition that $\Phi'_{\mu} : \uU'_{\mu} \hookrightarrow \XX|_{U_{M}^{\mu}}$ is a \emph{quasi-isomorphism} in particular means that the dimension of $U_{\mu}$ must be \emph{at least} equal to $d_{\mu} := \dim T_{\mu} X$. \\

The reader that is familiar with singularity theory will recall that the germ of an analytic space $(X, \0)$ given by closed embedding of analytic spaces $X \hookrightarrow M$ has an \emph{embedding dimension}, denoted by $\ed_{\0}X$. The embedding dimension of the germ $(X,\mu)$ is by definition the minimum dimension $n$ of affine space $(\aaa^{n}, \0)$ such that there exists an embedding of analytic germs $(X, \mu) \hookrightarrow (\aaa^{n},0)$. It is a standard result that the embedding dimension for the analytic germ $(X,\mu)$ is given by $d_{\mu} := \dim T_{\mu} X$.

\begin{defn}
	We say that a quasi-smooth analytic space $\XX = (M,E,\lambda)$ is \emph{minimal} at a point $\mu \in \tau^{0} \XX $ if $D_{\mu} \lambda = 0$.  
\end{defn}

\begin{lem} \label{mindim}
	Suppose that $\XX = (M,E,\lambda)$ is an analytic \qs space, and  $\mu \in X$. Then, a quasi dg chart of the form $\uU_{\mu}^{\min} = (  U_{\mu}^{\min}, \lambda^{\min}_{\mu}: U_{\mu}^{\min} \rightarrow \CC^{d})$ where $U_{\mu}^{\min} \subset \CC^{d_{\mu}}$ is an open polydisc of the origin and $d_{\mu} = \dim T_{\mu} X$, is a local analytic \qs space which is \emph{minimal} at $\0_{\CC^{d_{\mu}}}$.

\end{lem}

\begin{proof}
	As $\Phi_{\mu}$ is a quasi-isomorphism, we have that the following commutative square is a quasi-isomorphism of the rows (where $\Phi_{\mu} = (\phi_{\mu} , \phi_{\mu}^{\#})$):
	
	\[
	\begin{tikzcd}
		T_{\0} \big( \DDelta^{d_{\mu}}(\ep_{\mu}) \big)  \arrow[r, "D_{\0} \lambda^{\min}_{\mu} " ] \arrow[d, "D_{\0} \phi_{\mu}",swap] & \CC^{\ob(\mu) }  \arrow[d, "\phi_{\mu}^{\#}"] \\
		T_{\mu} M \arrow[r, "D_{\mu} \lambda" ] & E_{\mu} \\
	\end{tikzcd}
	\]
	
	For the bottom row, we have that $H^{0} ( T_{\mu} M \xrightarrow{ D_{\mu} \lambda } E_{\mu} ) = T_{\mu}X$. Therefore, if $\dim T_{\0} \big( \DDelta^{d_{\mu}  }(\ep_{\mu} ) \big) = \dim T_{\mu} X$ and the map $ \ker ( D_{\0} \lambda^{\min}_{\mu} ) \rightarrow T_{\mu} X$ is to be an isomorphism, we must have that $ \ker ( D_{\0} \lambda^{\min}_{\mu} ) = T_{\0} \big( \DDelta^{d_{\mu}  }(\ep_{\mu} ) \big)$,  or that $D_{\0} \lambda^{\min}_{\mu} = 0$. In other words, the analytic \qs space $\uU^{\min}_{\mu}$ is \emph{minimal} at $\0$.
\end{proof}

\begin{lem}
	Suppose that $\XX = (M,E,\lambda)$ is an analytic \qs space that is minimal at point $\mu \in \tau^{0} \XX$. Suppose that $\Phi: \XX \rightarrow \YY = (N, F, \zeta)$ is a quasi-isomorphism of analytic \qs spaces, with $\YY$ being minimal at $\phi (\mu) \in \tau^{0} \YY $. Then, there exist open subsets $U_{\mu} \subseteq M$ containing $\mu$, and $U_{\phi (\mu)} \subseteq N$ containing $\phi(\mu)$ such that $\XX|_{U_{\mu}} \cong \YY |_{U_{\phi(\mu)}}$ (that is, $\XX$ and $\YY$ are strictly isomorphic in a neighbourhood of $\mu$ and $\phi (\mu)$ respectively). 
\end{lem}

\begin{proof}
	
	As $\XX \rightarrow \YY$ is a quasi-isomorphism, we have the following exact square:
	
	\[
	\begin{tikzcd}
		T_{\mu} M \arrow[r, "D_{\mu} \lambda" ] \arrow[d, "D_{\mu} \phi"] & E |_{\mu} \arrow[d, "\phi^{\#}|_{\mu}"]  \\
		T_{\phi (\mu)} N \arrow[r, "D_{\phi (\mu)} \zeta" swap]  & F |_{\phi (\mu)} \\
	\end{tikzcd}
	\]

	As $\XX$ and $\YY$ are minimal at $\mu$ and $\phi (\mu)$ respectively, we have that $D_{\mu} \lambda = 0 = D_{\phi(\mu)} \zeta $. As the square is quasi-isomorphism of the rows, this means that the vertical arrows given by $( D_{\mu} \phi , \phi^{\#}|_{\mu} )$ are isomorphisms. Therefore, there are open sets $U_{\mu}$ and $U_{\phi(\mu)}$ in $M$ and $N$ containing $\mu$ and $\phi (\mu)$ respectively, that are isomorphic as analytic spaces. The above exact square also shows us that $E|_{U_{\mu}}$ and $(\phi^{*} F)|_{ U_{\mu}}$  are isomorphic as bundles over $U_{\mu}$.
\end{proof}

Lemma \ref{mindim} says that around a point $\mu \in X$, a quasi dg chart $\uU_{\mu}^{\min} = ( \DDelta^{d_{\mu}  }(\0, \rho^{\min}_{\mu} ), \ul{\CC}^{r^{\min}_{\mu}}, \lambda^{\min}_{\mu})$ such that $d_{\mu}$ coincides with the embedding dimension of the germ $(X,\mu)$ is a minimal analytic \qs space. Next, we would like to show that we can always find a quasi dg chart whose dimension attains this embedding  dimension of $(X,\mu)$. That is, we would like to show that for any $\mu \in X = \tau^{0}(\XX)$, there exists a \emph{minimal dg chart} $\uU^{\min}_{\mu} \hookrightarrow \XX$ around $\mu$. 

\subsubsection*{Homotopy retracts and minimal models of local analytic \qs spaces} 

Suppose that $\uU_{\mu} = (U_{\mu}, \ul{\CC}^{r_{\mu} }, \lambda_{\mu})$ is a local analytic \qs space. Previously, we had seen that if $D_{\0_{\mu}} \lambda_{\mu} = 0$, then $\uU_{\mu}$ is \emph{minimal} at $\0_{\mu}$. In this section, we  show that one can always find a minimal dg chart, essentially by splitting off a subspace $U^{\min}_{\mu}$ of $U_{\mu}$ where $D_{\0_{\mu}} \lambda_{\mu} |_{U_{\mu}^{\min} } = 0$. Furthermore, we can find a minimal $\uU^{\min}_{\mu} \hookrightarrow \uU_{\mu}$ such that $\uU_{\mu}^{\min}$ is a homotopy deformation retract of $\uU_{\mu}$.

\begin{defn}
	Suppose that $\uU_{\mu} = (U_{\mu}, \ul{\CC}^{r_{\mu}}, \lambda_{\mu})$ is a local analytic \qs space. We define a minimal homotopy retract of $\uU_{\mu}$ to be the data of:  
	
	\begin{enumerate}
		\item{a \emph{minimal} local analytic \qs space $(\uU^{\min}_{\mu}, \0_{\mu})$, with a morphism $I_{\mu}: \uU_{\mu}^{\min} \hookrightarrow \uU_{\mu}$   }
		\item{a morphism  $P_{\mu}: \uU_{\mu} \rightarrow \uU_{\mu}^{\min}$, such that $P_{\mu} I_{\mu} = \Id_{\uU^{\min}_{\mu}}$}
		
		\item{a $1$-parameter family of morphisms $H^{t}_{\mu} = (h^{t}_{\mu}, h^{t}_{\mu}{}^{\sharp} ) : \uU_{\mu} \rightarrow \uU_{\mu}$ such that:

			\begin{itemize}
				\item{$h^{t}_{\mu}(x)$ is smooth in $t$ for all $x \in U_{\mu}$; and furthermore we have $H^{t}_{\mu} |_{t = 0 } = \Id_{\uU_{\mu} } $ and $H^{t}_{\mu} |_{t = 1} = I_{\mu} P_{\mu}$ }
				\item{ $H^{t}_{\mu} \circ I_{\mu} = I_{\mu}$ for all $t \in [0,1]$}  
			\end{itemize}
			\textbf{We refer to this $1$-parameter family $H_{\mu}^{t}$ of morphisms as a homotopy. } 
		}
	\end{enumerate}
	
	\noindent We say that $\uU_{\mu}$ has a \emph{minimal model decomposition} if we have further a decomposition of underlying bodies: $U_{\mu} \cong U_{\mu}^{\min} \times j_{\0}(N_{\mu}) = j( N_{U_{\mu}^{\min} })$, where $U_{\mu}^{\min}$ denotes the underlying body of $\uU_{\mu}^{\min}$, and $j_{\0}(N_{\mu})$ is a sufficiently small open polydisc around the origin of $N_{\mu}$.
	\\
	
	If $\uU_{\mu}$ has a minimal model decomposition, we write $\uU_{\mu} = \uU_{\mu}^{\min} \times j_{\0}(N_{\mu})$, or  where $U_{\mu}^{\min}$ is a minimal model which is a homotopy retract of $\uU_{\mu}$.  As a notational convenience, we may also write $\uU_{\mu}^{\min} \times \wh{N_{\mu}} := \uU_{\mu}^{\min} \times j_{\0}(N_{\mu}) $, and $\wh{N_{\mu}}$ in place of $j_{\0}(N_{\mu})$ as the specific choice of polydisc $j_{\0}( N_{\mu}) $ around the origin in $N_{\mu}$ is irrelevant.
	
\end{defn}

\begin{prop}
	Suppose that $\uU_{\mu} = ( U_{\mu}, \ul{\CC}^{r_{\mu}}, \lambda_{\mu})$ is a local analytic \qs space. Then, there exists a minimal model decomposition $\uU_{\mu} \cong \uU_{\mu}^{\min} \times j_{\0}(N_{\mu})$. In particular, there exists: 
	
	\begin{itemize}
		\item{A minimal chart $I_{\mu}: \uU_{\mu}^{\min} \hookrightarrow \uU_{\mu}$, and a projection $P_{\mu}: \uU_{\mu} \rightarrow \uU_{\mu}^{\min}$ such that $P_{\mu} I_{\mu} = \Id_{\uU_{\mu}^{\min} }$} 
		\item{A $1$-parameter family of morphisms $H_{\mu}^{t}: \uU_{\mu} \rightarrow \uU_{\mu}$ such that $H^{t}_{\mu} |_{t = 1 }  = \Id_{\uU_{\mu}}$, $H^{t}_{\mu}|_{t = 0 } = I_{\mu} P_{\mu}$, and $H^{t}_{\mu} \circ I_{\mu} = I_{\mu}$ for all $t \in [0,1]$.  } 
	\end{itemize} 
\end{prop}

\begin{proof}
	First, we define coordinates for $U_{\mu}$ so that $T_{\mu}: U_{\mu}^{\min} \times j_{\0}(N_{\mu}) \iso  U_{\mu}$ and $T^{\#}_{\mu}: \CC^{r_{\mu}} \cong \CC^{r_{\mu}'} \times \CC^{r_{\mu}''}$ so that $T_{\mu}^{\#} \circ \lambda_{\mu} \circ T_{\mu}^{-1} =(\wtil{\lambda}_{\mu}^{\min}(\bf{z}_{\mu}, \bf{n}_{\mu} ),  \bf{n}_{\mu})$. \\
	
	First, consider kernel $V_{\mu}^{\min}$ of the map:
	
	\[
	D_{\0_{\mu}} \lambda_{\mu}: V_{\mu} = T_{\0_{\mu}}U_{\mu} \rightarrow \ul{\CC}^{r_{\mu} } 
	\]
	and define $U_{\mu}^{\min} := V_{\mu}^{\min} \cap U_{\mu}$. Now, we choose a complementary subspace $N_{\mu}$ of $V_{\mu}^{\min}$ so that $V_{\mu} \cong V_{\mu}^{\min} \times N_{\mu}$. Then, have that the restriction $D_{\0_{\mu}} \lambda_{\mu} |_{N_{\mu} }$  is an isomorphism onto its image, so that we may find decompose $\CC^{r_{\mu}}$ into $\CC^{r_{\mu}' } \times \CC^{r''_{\mu}}$ where $D_{\0_{\mu}} \lambda_{\mu} |_{N_{\mu} } \cong \CC^{r_{\mu}''}$. \\ 
	Therefore, we can identify $N_{\mu}$ with its image under $D_{\0_{\mu}} \lambda_{\mu}|_{N_{\mu}}$ so that we have a commutative diagram:
	
	\begin{equation} \label{changecoord}
		\begin{tikzcd}[column sep = 5em]
			U_{\mu} \arrow[r, "\lambda_{\mu}" ] \arrow[d, "T_{\mu}"] & \CC^{r_{\mu}} \arrow[d, "T_{\mu}^{\#}"] \\
			U_{\mu}^{\min} \times N_{\mu}   \arrow[r, "T_{\mu}^{\#} \circ \lambda_{\mu} \circ T_{\mu}^{-1} " swap] &   \CC^{r_{\mu}'} \times \CC^{r_{\mu}''} 
		\end{tikzcd}
	\end{equation}
	
	where $T^{\#}_{\mu} \circ \lambda_{\mu} \circ T_{\mu}^{-1} ( \mbf{z}_{\mu}, \mbf{n}_{\mu} ) = (\wtil{\lambda}_{\mu}^{\min}( \mbf{z}_{\mu}, \mbf{n}_{\mu} ), \mbf{n}_{\mu} )$. Let us write 
	\[
	\lambda^{+}_{\mu} := T^{\#}_{\mu} \circ \lambda_{\mu} \circ T_{\mu}^{-1}  
	\]
	
	Now, we define $\lambda_{\mu}^{\min}(\mbf{z}_{\mu}) = \wtil{\lambda}_{\mu}^{\min}( \mbf{z}_{\mu}, \0 )$, and
	
	\[
	\uU_{\mu}^{\min} = (U_{\mu}^{\min}, \ul{\CC}^{r_{\mu}'}, \lambda^{\min}_{\mu})
	\]
	
	and along with $I_{\mu}: \uU_{\mu}^{\min} \hookrightarrow \uU_{\mu}$, given by $I_{\mu} = (\mbf{i}_{\mu}, \mbf{i}^{\sharp}_{\mu}  )$ where $\mbf{i}_{\mu}: U_{\mu}^{\min} \hookrightarrow U_{\mu} = U_{\mu}^{\min} \times j_{\0}(N_{\mu})$ and $\mbf{i}_{\mu}^{\sharp}: \CC^{r_{\mu}'} \hookrightarrow \CC^{r_{\mu}' } \times \CC^{r_{\mu}''}$ are given by the respective zero sections (i.e. $\mbf{i}_{\mu}(\mbf{z}_{\mu}) = (\mbf{z}_{\mu}, 0)$ and $\mbf{i}_{\mu}^{\sharp}( \mbf{n}_{\mu} ) = (0 , \mbf{n}_{\mu} )$). By \ref{changecoord}, $I_{\mu}$ is an immersion of analytic \qs spaces. \\
	
	Now, we define the morphism of analytic \qs spaces $P_{\mu}: U_{\mu} \rightarrow U_{\mu}^{\min}$. We need to define $P_{\mu} = (\pi_{\mu}, \pi_{\mu}^{\#})$ so that the diagram
	
	\begin{equation} \label{projectiondiagram}
		\begin{tikzcd}
			U_{\mu} \arrow[r, "\lambda^{+}_{\mu}"]  \arrow[d, "\pi_{\mu}"] & \ul{\CC}^{r_{\mu}} \arrow[d, "\pi_{\mu}^{\#}"] \\
			U_{\mu}^{\min} \arrow[r, "\lambda_{\mu}^{\min}"] & \ul{\CC}^{r_{\mu}'} 
		\end{tikzcd}
	\end{equation}

	To this end, we define $\pi_{\mu}(\mbf{z}_{\mu}, \mbf{n}_{\mu}) = \mbf{z}_{\mu} \in U_{\mu}^{\min}$, and 
	\[
	\pi_{\mu}^{\#}(v', v'') = v' - v'' \cdot  \big( \int_{0}^{1} \wtil{\lambda}_{\mu}^{\min}(\mbf{z}_{\mu}, t \mbf{n}_{\mu} dt) \big) 
	\]
	
	Note that with this definition, we have
	
	\begin{align*}
		& \pi_{\mu}^{\#} \circ \lambda_{\mu}^{+} ( \mbf{z}_{\mu}, \mbf{n}_{\mu} ) = \pi_{\mu}^{\#}( \wtil{\lambda}_{\mu}^{\min}( \mbf{z}_{\mu}, \mbf{n}_{\mu}), \mbf{n}_{\mu} ) \\ 
		& = \wtil{\lambda}_{\mu}^{\min}( \mbf{z}_{\mu}, \mbf{n}_{\mu} ) - \mbf{n}_{\mu} \cdot \big( \int_{0}^{1} \wtil{\lambda}_{\mu}^{\min}( \mbf{z}_{\mu}, t  \mbf{n}_{\mu} ) dt \big)  \\
		& = \wtil{\lambda}_{\mu}^{\min}( \mbf{z}_{\mu}, \mbf{n}_{\mu} ) - \big( \int_{0}^{1} \frac{d}{dt} \wtil{\lambda}_{\mu}^{\min}( \mbf{z}_{\mu}, t  \mbf{n}_{\mu} ) dt \big) \\
		& = \wtil{\lambda}_{\mu}^{\min}( \mbf{z}_{\mu}, \mbf{n}_{\mu} ) - \big( \wtil{\lambda}_{\mu}^{\min}( \mbf{z}_{\mu}, \mbf{n}_{\mu} )  - \wtil{\lambda}_{\mu}^{\min}( \mbf{z}_{\mu}, 0 ) \big) \\
		& = \wtil{\lambda}_{\mu}^{\min}(\mbf{z}_{\mu}, \0 ) = \lambda_{\mu}^{\min} \circ \pi_{\mu}( \mbf{z}_{\mu}, \mbf{n}_{\mu} )
	\end{align*}
	
	Therefore, $P_{\mu}: \uU_{\mu} \rightarrow \uU_{\mu}^{\min}$ is a morphism of analytic \qs spaces. \\

	Finally, let us show now that there is a homotopy retract context 
	
	\[
	\uU_{\mu}^{\min}  \rightleftarrowstack{I_{\mu} }{P_{\mu} }   \uU_{\mu}  \circlearrowleft{ H_{\mu}^{t} } 
	\] 
	
	We define $H_{\mu}^{t} = (h_{\mu}^{t}, h_{\mu}^{t}{}^{\#})$ where 
	\begin{align*}
		&  h_{\mu}^{t}( \mbf{z}_{\mu}, \mbf{n}_{\mu} ) = ( \mbf{z}_{\mu}, t \mbf{n}_{\mu} ) \\	
		& h_{\mu}^{t}{}^{\#}(v',v'') =  \left( v' - v'' \cdot  \big( \int_{t}^{1} \wtil{\lambda}_{\mu}^{\min}(\mbf{z}_{\mu}, t \mbf{n}_{\mu} dt) \big), t \cdot v''  \right) \\
	\end{align*}
	
	Then, we can clearly see that:
	
	\begin{itemize}
		\item{$H_{\mu}^{t} |_{t = 1} = \Id_{\uU_{\mu}}$, $H_{\mu}^{t}|_{t = 0 } = I_{\mu} P_{\mu}$}
		\item{$H_{\mu}^{t} \circ I_{\mu} = I_{\mu}$ for all $t \in [0,1]$. } 
	\end{itemize}
	Therefore, $\uU_{\mu}^{\min} \times j_{\0}(N_{\mu})$, $I_{\mu}$, $P_{\mu}$, $H_{\mu}^{t}$ defines a minimal model decomposition of $\uU_{\mu}$.
\end{proof}

	\begin{defn} \label{minimalatlas} \hfill \\
		Suppose that $\XX = (M,E,\lambda)$ is an analytic \qs space. We say that a dg atlas $\Phi$ for $\XX$ is a \emph{homotopy minimal atlas} if: 
		
		\begin{itemize}
			\item{There exists a dg chart $\Phi_{\mu}: \uU_{\mu} \hookrightarrow \XX$, for each $\mu \in X = \tau^{0}\XX$. That is, for each $\mu \in \tau^{0}\XX$, there exists an open neighbourhood $U_{M}^{\mu} \subset M$ containing $\mu$ with a dg chart $\Phi_{\mu}: \uU_{\mu} \iso \XX|_{U_{M}^{\mu}}$. }
			
			\item{
				For each dg-chart $\Phi_{\mu}: \uU_{\mu} \hookrightarrow \XX$, there is a chosen minimal model decomposition 
				
				\[
				\uU_{\mu} \cong \uU_{\mu}^{\min} \times \wh{N_{\mu}}
				\]
			}

		\end{itemize}
.
		
		That is, for each $\mu \in \tau^{0}\XX$, and dg chart $\Phi_{\mu}: \uU_{\mu} \hookrightarrow \XX$, we choose a minimal chart $I_{\mu}$ and projection $P_{\mu}$, and a homotopy $H^{t}_{\mu}$ between $\Id_{\uU_{\mu} }$ and $I_{\mu} P_{\mu}$, so that the triple $(I_{\mu}, P_{\mu}, H^{t}_{\mu})$  defines a minimal homotopy retract context.  \\
		
		With this definition, the atlas $\Phi$, along with its minimal model decompositions at each $\mu$, is so that the following diagram commutes:
		
		
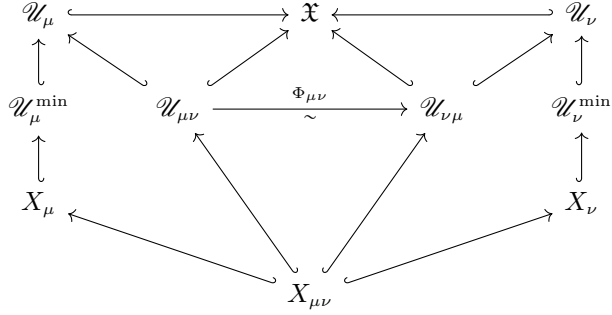
\[\begin{tikzcd}
			{\uU_{\mu}} && \XX & {} & {\uU_{\nu}} \\
			{\uU_{\mu}^{\min}} & {\uU_{\mu \nu}} && {\uU_{\nu \mu}} & {\uU_{\nu}^{\min}} \\
			{X_{\mu}} & {} & {} & {} & {X_{\nu}} \\
			&& {X_{\mu \nu}}
			\arrow[hook, from=1-1, to=1-3]
			\arrow[hook', from=1-5, to=1-3]
			\arrow[hook', from=2-1, to=1-1]
			\arrow[hook', from=2-2, to=1-1]
			\arrow[hook, from=2-2, to=1-3]
			\arrow[from=2-2, to=2-4, "\Phi_{\mu \nu}",  "\sim" swap]
			\arrow[hook', from=2-4, to=1-3]
			\arrow[hook, from=2-4, to=1-5]
			\arrow[hook, from=2-5, to=1-5]
			\arrow[hook', from=3-1, to=2-1]
			\arrow[hook, from=3-5, to=2-5]
			\arrow[hook', from=4-3, to=2-2]
			\arrow[hook, from=4-3, to=2-4]
			\arrow[hook', from=4-3, to=3-1]
			\arrow[hook, from=4-3, to=3-5]
		\end{tikzcd}\]
		
	\end{defn}

	Now, for any morphism of local analytic \qs spaces $\Phi_{\ups_{01}}: (\uU_{\ups_{01}}, \0_{\ups_{01} } ) \rightarrow (\uU_{\ups_{10}}, \0_{\ups_{1}})$, we have an induced commutative diagram:

	\[
	\begin{tikzcd} \label{eqn:mindiagram}
		\uU_{\ups_{01} }^{\min} \arrow[r, "I_{\ups_{0} }", hookrightarrow] \arrow[d, dotted, "\Phi_{\ups_{01}}^{\min}"] & \uU_{\ups_{01} } \arrow[d, "\Phi_{\ups_{01}}"] \\
		\uU_{\ups_{10}}^{\min}  \arrow[r, "P_{\ups_{1} }", leftarrow] & \uU_{\ups_{10}} 
	\end{tikzcd}
	\]

	\subsubsection*{Induced bundle constructions}
	A minimal model decomposition $\uU_{\mu} = \uU_{\mu}^{\min} \times \wh{N_{\mu}}$ as above gives a trivial linear polydisc bundle over the body $U_{\mu}^{\min}$ of the minimal model $\uU_{\mu}^{\min}$, as the decomposition 
	gives $U_{\mu} = U_{\mu}^{\min} \times \wh{N_{\mu}}$, for some polydisc $j_{\0}N_{\mu}$ centered at the origin of $N_{\mu}$.  We may refer to this linear polydisc bundle over $U_{\mu}^{\min}$ with the notation
	
	\[
	\wh{N}_{U_{\mu}^{\min}} := U_{\mu}^{\min} \times \wh{N_{\mu}}  \xrightarrow{P_{\mu}} U_{\mu}^{\min}
	\]

	From now on, we will always assume that polydisc fiber bundles are all \textbf{linear} polydisc bundles, so that we may drop the adjective ``linear" in our terminology.\\

	\begin{ex}
		Suppose that $M = \CC^{n}$,  and $S: M \rightarrow \CC$ is an analytic function of the form:
		
		\[
		S = z_{1}^{2} + \cdots + z_{d}^{2} + p(z_{d+1}, \cdots, z_{n}) 
		\]
		\noindent where $p$ is an analytic function in degrees at least $3$. \\
		
		Then, we let $\XX = \dCrit(S) = (M, \Omega_{M}, dS)$, so that $\tau^{0}(\XX) = X = \Crit(S)$.  \\

		We have a decomposition $M = \CC^{d} \times \CC^{n-d}$, which also gives us a decomposition $\Omega_{M} \cong \Omega_{\CC^{d}} \boxplus \Omega_{\CC^{n-d}}$. Note that the function $S$ splits into $q_{d}(z_{1}, \cdots, z_{d}) = z_{1}^{2} + \cdots z_{d}^{2}$ and $p(z_{d+1}, \cdots z_{n})$, which is an analytic function that starts in cubic degrees. This gives us a section $dQ \boxplus dp$ of the split bundle $\Omega_{\CC^{d}} \boxplus \Omega_{\CC^{n-d}}$, giving a decomposition:
		
		\[
		\dCrit(S) \cong \dCrit(Q) \times \dCrit(p)
		\]
		
		This is a global minimal model decomposition. One can check that $\dCrit(p)$ is a minimal local analytic \qs space with base point $\0_{\CC^{n}}$, and $\dCrit(Q)$ is a contractible analytic \qs space. \\
		
		We can interpret the quasi-isomorphism between the minimal model $\dCrit(p) \hookrightarrow \dCrit(S)$ as describing a \emph{\qs equivalence} between $p(z_{d+1}, \cdots, z_{n})$ and $S = p(z_{d+1}, \cdots, z_{n} ) + z_{1}^{2} + \cdots z_{d}^{2}$
	\end{ex}

	For general analytic functions $S: M \rightarrow \CC$, we have:
	
	\begin{lem} \emph{(Morse-Thom splitting lemma \cite{Greuel2006})}\\ \label{morselemma}
		Given an analytic manifold $M$ and an analytic function $S: M \rightarrow \CC$, suppose that $x \in M$ is a critical point for $S$. Then, for $d = \rank(H(S)(x))$, there exists an analytic chart $\phi_{x}: U_{x} \xrightarrow{\sim} U_{M}^{x} = \phi(U_{x})$ around $x$ such that 
		
		\[
		S \circ \phi_{x}( z_{1}, \cdots, z_{n}) = S^{(2)}(z_{1},\cdots,z_{d}) + S^{(\geq 3)}(z_{d+1}, \cdots, z_{n})
		\]
		with $S^{(2)}$ being a non degenerate quadratic form vanishing at the origin, and $S^{(\geq 3)}$ is an analytic function vanishing at the origin in degrees $3$ or higher. Such a decomposition is unique up to equivalence.
	\end{lem} 
	
	Thus, for the analytic \qs space $\dCrit(S)$, around a point $x \in \tau^{0} \dCrit(S)$, we can find an open neighbourhood $U_{x}$ and a minimal model decomposition:
	
	\[
	\dCrit(S) \cong  \dCrit(  S^{(\geq 3)} ) \times \dCrit(S^{(2)}) 
	\]

\subsection*{Dg analytic spaces}
We will find the need to expand our model for derived spaces to more general differential graded spaces. In studying derived geometric spaces arising from $L_{\infty}$-algebras, we will run into differential graded analytic spaces that may not be quasi-smooth.\\

There is essentially an equivalence between \emph{positively graded} convergent $L_{\infty}$-algebras and local differential graded analytic spaces. We will recall that a local \qs space is essentially equivalent to the data of the analytic function on a polydisc. This perspective underpins the general idea behind the equivalence between local dg analytic spaces and positively graded $L_{\infty}$-algebras. When an analytic space $X$ has a more complicated obstruction theory (for example, not given by any vector bundle), a differential graded analytic space which provides a derived enhancement is no longer locally determined by the data of an analytic function, but instead by a homological generalization known as an $L_{\infty}$-algebra.

Briefly speaking, a differential graded analytic space is given by a smooth analytic space with a structure sheaf of commutative algebras enriched in chain complexes. While local analytic spaces (or ``affine analytic spaces") are defined via commutative algebras, affine differential graded analytic spaces are modelled on commutative differential graded algebras (which are commutative algebras in chain complexes).

\begin{defn}
	A differential graded analytic space is a $\CC$-ringed space $(M, \O^{\bullet}_{M})$ where:
	
	\begin{itemize}
		
		\item{$M$ is a smooth analytic space }
		\item{ $(\O^{\bullet}_{M}, q_{M})$ is a non-positively graded coherent sheaf of analytic cdgas over $M$}
	\end{itemize}
\end{defn}
Given a dg analytic space, the analytic space $X$ defined by the coherent sheaf of ideals $\Im( q^{-1}_{M}: \O_{M}^{-1} \rightarrow \O_{M}^{0} )$ is called  the \emph{classical truncation} of $(M , \O_{M}^{\bullet})$. Furthermore, we have that the structure sheaf of $X$ as an analytic space is given by $H^{0}(\O_{M}^{\bullet})$ 

\begin{defn}
	A  local model dg analytic space is given by a model complex space $U$ (for example, a polydisc in some $\CC^{n}$), and a cdga $[ \cdots \rightarrow A^{-1} \xrightarrow{q^{-1} } A^{0} ]$ with $A^{0} = \Gamma(\O_{U}) = \CC[U]$.
\end{defn}

Thus, we can easily see that $H^{0}(A) = \CC[ U ] / \im(q|_{A^{-1}} )$ describes the coordinate ring for a local model analytic space embedded a smooth complex space $U$. This is the coordinate ring of the associated classical truncation, which is the analytic space given by the vanishing locus of the ideal of analytic functions $\im(q|_{A^{-1}} )$ in $U$.

	\section{The geometry of analytic $L_{\infty}$-algebras}

	In this section, we introduce the concept of an $L_{\infty}$-algebra. The purpose of section is primarily expositional, and to provide the reader with some important context. While the subject of $L_{\infty}$-algebras is already well-established, the interested reader may sometimes find it difficult to find resources that illuminate some of the more geometric ideas underlying $L_{\infty}$-algebras.\\
	
	Our main source of differential graded geometry will come out of $L_{\infty}$-algebras. Where germs of analytic spaces are essentially described by the Taylor expansions of the defining analytic equations, $L_{\infty}$-algebras encode germs or local models of dg analytic spaces. In particular, $L_{\infty}$-algebras that encode local models of quasi-smooth analytic spaces are indeed equivalent to describing the Taylor expansions of the corresponding defining equations. We will introduce the notion of a \emph{local analytic} $L_{\infty}$-algebra, or an $\cL_{\infty}$-algebra: these are $L_{\infty}$-algebras that encode \emph{convergent} germs of dg analytic spaces. \\ 
	
	In this section, we also introduce the notion of \emph{gauge-fixing data} for local analytic $L_{\infty}$-algebras. Dating back to the work of Kuranishi in deformations of complex manifolds, the terminology ``Kuranishi family" has been widely adopted to describe local moduli spaces as encoded by the cohomology of an $L_{\infty}$-algebra which governs the deformation theory of some objects. In this paper, we restrict our attention to a class of $\cL_{\infty}$-algebras see definition \ref{quasismoothLinftybundles}).  We refer to as \emph{quasi-smooth} $\cL_{\infty}$-algebras. 
	These can be viewed as a ``homotopical gluing" of quasi-smooth dg analytic spaces obtained from minimal $L_{\infty}$-algebras (i.e. cohomology $L_{\infty}$-algebras, as in Kuranishi theory).

	\subsection{$L_{\infty}$-algebras and dg formal neighbourhoods}
	
	All the content in this subsection is standard, and included here for exposition. For a wonderful reference on $L_{\infty}$-algebras, including the description of a pointed formal neighbourhood cdga, see \cite{Manetti2022}.
	
	Given a dg analytic space $\XX = (M, \O_{M}^{\bullet})$, we define the formal neighbourhood at a point $\mu$ in the classical truncation by taking the formal completion of $\O_{M}^{\bullet}$ at $\mu$. That is, 
	
	\[
	\O_{M, \mu}^{\bullet} := (\wh{\Sym}_{\CC}(V_{\mu}[1]^{\vee}), q_{\mu})
	\]
	where $T_{\mu} \XX  = V_{\mu}$ is a positively graded vector space (so that $V_{\mu}[1]^{\vee}$ is in non-positive degrees) and $q_{\mu}$ is a degree $1$ differential. The differential $q_{\mu}$ is defined by taking the formal completion of $q$ at $\mu$ -- in other words, it is given by the Taylor expansion of $q$ at $\mu$.  In particular, we have that $q_{\mu}$ is a degree $1$ derivation on $\wh{\Sym}_{\CC}(V_{\mu}[1]^{\vee})$ satisfying $q_{\mu}^{2} = 0$.\\ 
	
	Recall that a degree $1$ derivation on $\wh{\Sym}_{\CC}(V_{\mu}[1]^{\vee})$ is uniquely determined by a collection of multilinear maps of degree $1$:
	
	\begin{align*}
	& q_{\mu}^{(1)}: V_{\mu}[1]^{\vee} \rightarrow V_{\mu}[1]^{\vee}  \\
	& q_{\mu}^{(2)}: V_{\mu}[1]^{\vee}  \rightarrow  V_{\mu}[1]^{\vee} \odot V_{\mu}[1]^{\vee} \\
	& \cdots \\
	& q_{\mu}^{(k)}: V_{\mu}[1]^{\vee} \rightarrow  \odot^{k}  V_{\mu}[1]^{\vee}   \\
	& \cdots
	\end{align*}
	
	satisfying the equation $q_{\mu}^{2} = 0$, where $q_{\mu} = \sum\limits_{k = 1}^{\infty} q_{\mu}^{(k)}$. If we take the dual and de-shift appropriately, we can equivalently write this as a sequence of graded skew-symmetric multilinear operations
	
	\begin{align*}
	&  l_{1}^{\mu} := (q_{\mu}^{(1)})^{\vee}:  V_{\mu}^{\bullet} \rightarrow V_{\mu}^{\bullet }[1]  \\
	& l_{2}^{\mu} := (q_{\mu}^{(2)} )^{\vee} : V_{\mu}^{\bullet} \wedge V_{\mu}^{\bullet}  \rightarrow V_{\mu}^{ \bullet }\\
	& \cdots \\
	& l_{k}^{\mu} := (q_{\mu}^{(k)} )^{\vee} : \bigwedge^{k} V_{\mu}^{\bullet}  \rightarrow V_{\mu}^{\bullet }[2-k] \\
	& \cdots
	\end{align*}
	which satisfy some equations amounting to the dual of the equations set by $q_{\mu}^{2} = 0$. In the dual form, the equations may look more familiar -- for example, if we write $d = l_{1}^{\mu}$, $[-,-]_{2} := l_{2}^{\mu}$, and $[-,-,-]_{3} := l_{3}^{\mu}$, then we get the equation
	
	\[
	[x,[y,z]_{2}]_{2} + [[x,y]_{2},z]_{2} + [z, [x,y]_{2}]_{2} = d( [x,y,z]_{3} ) = 0
	\]
	
	That is, $[-,-]_{2}$ defines a skew-symmetric ``bracket" on $V_{\mu}$, which satisfies the Jacobi identity up to a coboundary term specified by the triple bracket $[-,-,-]_{3}$.  From this point of view, we see why an $L_{\infty}$-algebra is referred to as a \emph{(strong) homotopy 
	Lie algebra}. Now, note that the dg formal neighbourhood construction above does not depend on the fact that $V_{\mu}^{\bullet}$ is \emph{positively} graded, and makes sense if $V_{\mu}^{\bullet}$ is $\ZZ$-graded. Therefore, we define:

	\begin{defn} ($L_{\infty}$\textbf{-algebras})\\
	Suppose that $L^{\bullet}$ is a $\ZZ$-graded vector space. Then, an $L_{\infty}$ structure over $L^{\bullet}$ is defined to be a degree $1$ derivation $q$ on $\wh{S}(L[1]^{\vee}) = \wh{\Sym}(L[1]^{\vee})$  such that $q^{2} = 0$. This cdga is often referred to as the (cohomological) Chevalley-Eilenberg cdga of $L$, and is often denoted in the literature as $\CE^{\bullet}(L) := \Big( \wh{S}(L[1]^{\vee}), q \Big)$
	\end{defn}

	Given our discussion on formal neighbourhood cdgas of dg analytic spaces, this is a fairly natural way to define an $L_{\infty}$-algebra. We remark that one can make a number of equivalent definitions. Firstly, an $L_{\infty}$-algebra is more typically defined directly through describing the operations $l_{k}^{\mu}: \bigwedge^{k} V_{\mu} \rightarrow V_{\mu}$, instead of describing them through their duals $q_{\mu}^{(k)}$.

	\begin{defn}\textbf{(Alternative definition  1) }\\
	Suppose that $L^{\bullet}$ is a $\ZZ$-graded vector space. Then, an $L_{\infty}$ structure over $L^{\bullet}$ is given by a sequence of skew-symmetric multilinear maps of degree $2-n$,
	
	\[
	l_{n} : \bigwedge^{n}  L^{\bullet}  \rightarrow L^{\bullet}
	\]
	satisfying the (higher) homotopy Jacobi identities:
	
	\[
	\sum \sum \pm l_{j}(l_{k}(v_{\sigma(1)}, \cdots, v_{\sigma(k)} ), v_{\sigma(k+l)}, \cdots, v_{\sigma(n) } ) = 0
	\] 
	\end{defn}
	
	By considering graded coalgebras instead of completed symmetric algebras, we can do away with the need for a completion and define an $L_{\infty}$-algebra as below.
	
	\begin{defn}\textbf{(Alternative definition 2)}	\\
	Suppose that $L^{\bullet}$ is a $\ZZ$-graded vector space. Then, an $L_{\infty}$ structure over $L^{\bullet}$ is defined to be a degree $1$ co-derivation $Q$ on the free symmetric coalgebra $S^{c}(L[1])$, such that $Q^{2} = 0$.	
	\end{defn}

	With any of the above equivalent definitions, we say that the data $(L^{\bullet}, l_{1}, l_{2}, \cdots, l_{k} , \cdots )$ defines an $L_{\infty}$-algebra. \\

	One usually thinks of $\wh{\Sym}(V_{\mu}[1]^{\vee})$ as the cdga of formal functions (i.e. given by arbitrary power series) on a pointed formal neighbourhood with base point given by the origin $\0_{\mu}$. Then, as $q_{\mu}$ is a derivation acting on $\wh{\Sym}(V_{\mu}[1]^{\vee})$, one can view $q_{\mu} = \sum\limits_{ k \geq 1} q^{(k)}_{\mu}$ as a formal vector field acting on these functions, that is of homological degree $1$.  In light of this interpretation, the dual $L_{\infty}$ operations $\{ l_{k}^{\mu} \}_{k=1}^{\infty}$ can be viewed as the Taylor series coefficients of this formal vector field $q_{\mu}$.  Furthermore, we can easily see that $H^{0}( \wh{\VV}_{\mu} ) = \wh{\Sym}(V_{\mu}^{1}[1]^{\vee}) / \im( q_{\mu}: \wh{\VV}_{\mu}^{-1} \rightarrow \wh{\VV}_{\mu}^{0} )$ can be identified with the formal germ of the vanishing locus of the vector field $q_{\mu}|_{\wh{\VV}_{\mu}^{-1} }$ at the base point $\0_{\mu}$.\\
	
	With the Koszul dual, formal pointed dg space description of $L_{\infty}$-algebras, there is a natural way one can define morphisms between $L_{\infty}$-algebras. 
	
	\begin{defn}
	Suppose that $L_{\mu}$ and $L_{\nu}$ are $L_{\infty}$-algebras. We define a morphism $\phi: L_{\mu} \rightarrow L_{\nu}$ of $L_{\infty}$-algebras to be given by a morphism of cdgas 
	
	\[
	\phi_{\sharp}: \CE^{\bullet}(L_{\nu} ) \rightarrow \CE^{\bullet}(L_{\mu})
	\]
	
	such that $\phi_{\sharp} \circ q_{\nu} = q_{\mu} \circ \phi_{\sharp}$.
	\end{defn}
	
	By the universal property of the symmetric algebra, such a morphism $\phi^{\sharp}:  \calS(L_{\nu} ) \rightarrow \calS(L_{\mu})$ is uniquely determined by a sequence of maps
	
	\begin{align*} 
	& \phi_{\sharp}^{(1)}: L_{\nu}[1]^{\vee} \rightarrow L_{\mu}[-1]^{\vee} \\
	& \phi_{\sharp}^{(2)}: L_{\nu}[1]^{\vee} \rightarrow L_{\mu}[1]^{\vee} \odot  L_{\mu}[1]^{\vee} \\
	& \cdots \\
	& \phi_{\sharp}^{(k)}: L_{\nu}[1]^{\vee} \rightarrow \odot^{k}  L_{\mu}[1]^{\vee}  \\
	& \cdots
	\end{align*}
	
	or dually, by a sequence of maps $\phi^{k}: \bigwedge^{k} L_{\mu}  \rightarrow L_{\nu}$.  Dualizing the condition that $ \phi_{\sharp} \circ q_{\nu} = q_{\mu} \circ \phi_{\sharp}$, we obtain a more familiar form of the definition of the morphism between $L_{\infty}$-algebras.
	
	\begin{defn}
	We say that a morphism $\phi: L_{\mu} \rightarrow L_{\nu}$ of $L_{\infty}$-algebras is a quasi-isomorphism, if the underlying map $\phi_{\sharp}: \CE^{\bullet}(L_{\nu}) \rightarrow \CE^{\bullet}(L_{\mu})$ of cdgas is a quasi-isomorphism.
	\end{defn}
	
	If $V_{\mu}$ is a $\CC$-vector space, then we can speak of the convergence of the formal vector field $q_{\mu}|_{\wh{\VV}_{\mu}^{-1} }$. If $q_{\mu}|_{\wh{\VV}_{\mu}^{-1} }$ converges in some polydisc around the origin, then the vector field $q_{\mu}|_{\wh{\VV}_{\mu}^{-1} }$ defines a genuine local model dg analytic space over a polydisc of convergence. Of course, the convergence of $q_{\mu}|_{\wh{\VV}_{\mu}^{-1} }$ on some polydisc can be guaranteed by a boundedness condition on the growth of the Taylor coefficients given by the $l_{k}^{\mu}$. These $L_{\infty}$-algebras will be referred to as \emph{analytic $L_{\infty}$-algebras}. We will be primarily interested in such $L_{\infty}$-algebras, as we are interested in constructing differential graded analytic spaces that are not only formal. If $q_{\mu}|_{\wh{\VV}_{\mu}^{-1} }$ converges, then the $L_{\infty}$-algebra in question encodes the (non-formal) germ of the analytic space defined by the equation:
	
	\[
	F_{\tMC}(\mbf{z}) = \sum\limits_{k=1}^{\infty} \frac{1}{k!} l_{k}^{\mu}(\mbf{z}) = 0
	\]
	
	This known as the \emph{Maurer-Cartan equation} for the (convergent) $L_{\infty}$-algebra $(V_{\mu}, \{ l_{k}^{\mu} \}_{k=1}^{\infty})$, and $F_{MC}$ is known as the Maurer-Cartan function.\\
	
	We will see in the next subsection that germs of analytic differential graded spaces can \emph{always} be described by an appropriate convergent $L_{\infty}$-algebra. To recover the classical germ of an analytic space, we can consider the Maurer-Cartan equation associated to the $L_{\infty}$-algebra. This is essentially a feature obtained by using \emph{differential graded} objects to model derived geometry.

	\subsection{Analytic \texorpdfstring{$L_{\infty}$}{L-infinity}-algebras and $\cL_{\infty}$-algebras }

	\begin{defn} \label{analytic}
	An $L_{\infty}$-algebra $L$, whose underlying vector space $(L^{\natural}, \| \cdot \|)$ is normed, is \emph{analytic} if there exists a constant $C > 0$, independent of $k$ such that
	
	\[
	\| l_{k} \| < k! \cdot C^{k}
	\]
	
	for all $k$, where $l_{k}$ are the structure operations for $L$. Alternatively, we may say that $L$ and stucture maps $l_{k}$ are \emph{bounded}. 
	\end{defn}

	Note that we do not explicitly state that $L$ is finite dimensional, although finite dimensional analytic $L_{\infty}$-algebras will be of primary concern to us. However, we will need to consider infinite dimensional $L_{\infty}$-algebras in the forthcoming sequel to this paper -- when we apply the general ttheory to the setting of Chern-Simons theory. In the infinite dimensional setting, our $L_{\infty}$-algebras will be modelled on Fr\'echet spaces. 
	
	\begin{defn}
	An $L_{\infty}$-algebra $L$, such that the underlying graded vector space is a Fr\'echet space $(L^{\natural}, \{ | \cdot |_{n} \}_{n \in \nn_{0}} )$, is (Fr\'echet) analytic if there exist finitely many semi-norms $| \cdot |_{i_{1}} , \cdots, | \cdot |_{i_{n}}$  along with constants $C_{i_{j}} \geq 0$ for $j = 1, \cdots, n$, such that for each $i_{j}$ we have that 
	
	\[
	| l_{k} |_{i_{j}} < k! \cdot  C_{i_{j}}^{k}
	\]
	for all $k$.  
	\end{defn}

	\begin{rmk}
	As the operator norm and the dual operator norm satisfy $\| T \| = \| T^{\vee} \|$ for $T : V \rightarrow V$, we can formulate the analytic condition of an $L_{\infty}$-algebra dually with the same conditions. That is, $L^{\bullet},  \{ l_{k} \}$ is an analytic $L_{\infty}$-algebra, if there exists a constant $C$ such that for each $k$ we have that 
	
	\[
	\| q^{k} \| < k! \cdot C^{K}
	\]
	
	where $q^{k} = l_{k}^{\vee}$. 
	Thus, $q^{k}$ is bounded if and only if $l_{k}$ is bounded.
	\end{rmk}
	
	We will simply use the terminology ``analytic $L_{\infty}$-algebra" to refer to either scenario. In the infinite dimensional setting, an analytic $L_{\infty}$-algebra will always be required to be a Fr\'echet analytic $L_{\infty}$-algebra. Of course, the finite dimensional definition is a special case of the more relaxed Fr\'echet definition, where the underlying vector space only has a single norm generating its topology. \\

	We would like to work with analytic $L_{\infty}$-algebras, to guarantee that the Maurer-Cartan function converges in some open domain in $L^{1}$. 
	
	\begin{defn} 	
	The Maurer-Cartan function for an analytic $L_{\infty}$-algebra is:
	
	\[
	F_{\tMC}(x) = \sum \limits^{\infty}_{k=1} \frac{1}{k!} l_{k}(x, \cdots, x) 
	\]
	\end{defn}

	\begin{defn} (Convergence domains and local analytic $L_{\infty}$-algebras)\\
	Suppose that $L$ is an analytic $L_{\infty}$-algebra. 
	\begin{itemize}
		\item{ 
			If $U \subset L^{1}$ is an open (convex) subset on which the Maurer-Cartan function $F_{MC}$ converges, then we refer to $U$ as a \emph{convergence domain}. }
		\item{ We say that the pair $(L,U)$ is a local analytic $L_{\infty}$-algebra, if $U$ is a convergence domain for $L$.  }
		\item{ We write $\tMC(L,U) $ to be the solutions of the local Maurer-Cartan equation $F_{MC}(x) = 0$ for $x \in U \subset L^{1}$. Elements $\mu \in \tMC(L,U)$ are known as \emph{local Maurer-Cartan elements}. }
	\end{itemize}
	
	\textbf{As a matter of convienence, we also refer to a local analytic $L_{\infty}$-algebra as an $\cL_{\infty}$-algebra.}
	\end{defn}

	\begin{rmk}(The Maurer-Cartan equation in the ``dual picture")\\
	Recall that the $L_{\infty}$ operations $\{ l_{k} \}_{k = 1}^{\infty}$ can be interpreted as being dual to the Taylor coefficients of the defining homological vector field $q_{\mu}$. That is, $l_{k}^{\vee} = q^{(k)}: L[1]^{\vee} \rightarrow \bigodot^{k} L[1]^{\vee}$, and each $q^{(k)}$ extends uniquely (via the Leibiniz rule) to a derivation $q^{(k)}: \wh{S}(L[1]^{\vee}) \rightarrow \wh{S}(L[1]^{\vee})$. Recall that we write $q = \sum q^{(k)}$, which is a derivation on $\wh{S}(L[1]^{\vee})$. \\ 
	
	For $\alpha \in L[1]^{\vee}$ of degree $0$ (i.e. $\alpha \in (L^{1})^{\vee}$) one has $e^{\alpha} = \sum\limits_{m = 0}^{\infty} \frac{1}{m!} \alpha^{\odot m} \in \wh{S}(L[1]^{\vee})$. Then, one can dually describe the Maurer-Cartan locus as 
	
	\[
	\tMC^{*}\Big( \Ss(L) \Big) = \{ \alpha \in (L[1]^{\vee})^{0}  \mid q( e^{\alpha} ) = 0 \}
	\]	
	
	One can show that there is an analytic bijection between $\tMC^{*}\Big( \Ss(L)  \Big)$ and $\tMC(L)$. In finite dimensions, we can simply take the duals of $\alpha \in \tMC^{*}\Big( \Ss(L) \Big)$ to produce $x \in \tMC(L)$ and vice versa.

	\end{rmk}

	If $L$ is analytic, then we can consider the Taylor expansion of $q$ centered at a nearby $b \in L^{1}$. The corresponding Taylor coefficients will be given by the following:
	
	\begin{defn}
	For $b \in U \subset L^{1}$, we can perturb the structure operations by $b$ to define twisted $L_{\infty}$ operations:
	
	\[
	l_{k}^{b}(x_{1}, \cdots x_{k}) = \sum\limits_{j = 0}^{\infty} \frac{1}{j!} l_{j + k} (b^{ \odot j}, x_{1}, \cdots x_{k})
	\]
	\end{defn}
	
	\begin{lem}
	
	If $b \in U$ as above, then these operations pointwise converge and are well defined. If $b = \mu \in \tMC(L,U)$ then the operations $\{ l_{k}^{\mu} \}$ form an $L_{\infty}$ structure on the graded vector space $L^{\natural}$.
	\end{lem}
	
	This can be seen by a direct verification of the Jacobi axioms, see \cite{kraft2022introductionlinftyalgebrashomotopytheory} for details. \\
	
	We write $L_{\mu} = (L^{\natural}, \{ l_{k}^{\mu} \}_{k=1}^{\infty} )$ to denote the twisted $L_{\infty}$-algebra.  We will often use the letter $d$ to denote the differential $l_{1}: L^{\bullet} \rightarrow L^{\bullet + 1}$, and write $d_{\mu} = l_{1}^{\mu}$ for the twisted differential.

	\begin{rmk}
	Recall that an analytic $L_{\infty}$-algebra $L$ describes the germ of a convergent Maurer-Cartan function at the basepoint $\0$. The twisting procedure above, when twisting by a Maurer-Cartan element $\mu \in \tMC(L,U)$, produces an $L_{\infty}$-algebra $L_{\mu}$ which describes the germ of the \emph{same} analytic space, but centered at a nearby point $\mu$.  If $\mu$ is not a Maurer-Cartan element, then the twisted operations do not describe the germ of this analytic space at $\mu$, as the defining equations do not even vanish at $\mu$. The resulting object one obtains via these operations is known as a \emph{curved} $L_{\infty}$-algebra.
	\end{rmk}

	Therefore, we see that a local analytic $L_{\infty}$-algebra $(L,U)$, describes a family of \emph{curved} $L_{\infty}$-algebras over the base $U$. Note that the curved $L_{\infty}$-algebras are all defined over the same graded vector space $L^{\natural}$, so that we may describe this family in terms of a graded vector bundle over $U$.

	\begin{cons} (The bundle construction)  \label{koszulconstruction} \\
	Suppose that $(L,U)$ is a local analytic $L_{\infty}$-algebra, with structure operations $\{ l_{k} \}_{k=1}^{\infty}$. For each $k$, we define $\bL^{k}_{U} = U \times L^{k}$  to be the trivial $L^{k}$ bundle over $U$. Then, we define the following morphisms:
	
	\begin{enumerate}
		\item{For each $n \geq 1$, we define $\lambda_{n}: \bigodot^{n} \bL^{\bullet}_{U}[1] \rightarrow  \bigodot \bL^{\bullet}_{U}[1]$ by setting:
			
			\[
			\lambda_{n}(b ; v_{1}, \cdots v_{n}) = (b ; l^{b}_{n}(v_{1}, \cdots, v_{n}) )
			\]
			
		}
		
		\item{We define $\lambda_{0}: U \rightarrow \bL^{2}_{U} $ to be the section of the bundle $\bL^{2}_{U}$ defined by the Maurer-Cartan function:
			
			\[
			\lambda_{0}(b) = F_{MC}(b) = \sum\limits_{k=1}^{\infty} \frac{1}{k!}l_{k}(b^{k})    
			\]
		}
	\end{enumerate}	
	By restricting considering the restriction to strictly positive degrees, given by $L^{\bullet \geq 1}$, we obtain a dg analytic space $\Ss(L^{\bullet \geq 1},U)$ with the sheaf cdga of functions defined by
	
	\[
	\O^{\bullet}_{\Ss(L^{\geq 1},U)} = \wh{\Sym}_{\O_{U}}( \bL^{\bullet \geq 2}[-1]^{\vee} )   
	\]
	together with the differential $q = \sum q_{k}$, where $q_{k} = \lambda_{k}^{\vee}$.\\
	
	We say that $\Ss(L^{\geq 1},U)$ is the $L_{\infty}$ space associated to the $L_{\infty}$-algebra $L^{\geq 1}$ (the positive spectrum of the $L_{\infty}$-algebra).
	\end{cons}

	We see that the classical truncation recovers the Maurer-Cartan locus in $U$, as $\tau^{0}(\Ss(L^{\geq 1},U)) = \spec \big( H^{0}(\O^{\bullet}_{\Ss(L^{\geq 1},U)} ) \big) = \tMC(L^{\geq 1},U)$. \\

	Using the notation of an local dg analytic space associated to an analytic $L_{\infty}$-algebra $L$, we can define \emph{analytic morphisms} between local analytic $L_{\infty}$-algebras.
	
	\begin{defn} \label{Linftymorphisms}
	Suppose that $(L_{\alpha}, U_{\alpha})$ and $(L_{\beta}, U_{\beta})$ are local analytic $L_{\infty}$-algebras, with respective convergence domains $U_{\alpha}$ and $U_{\beta}$. We define a local analytic morphism $\phi: (L_{\alpha}, U_{\alpha}) \rightarrow (L_{\beta}, U_{\beta})$ to be given by a morphism of cdgas 
	
	\[
	\phi_{\sharp}: \O^{\bullet}( \calS(L_{\beta}^{\bullet \geq 1}, U_{\beta}) ) \rightarrow \phi_{*}\O^{\bullet}( \calS(L_{\alpha}^{\bullet \geq 1}, U_{\alpha} ) )
	\]
	such that $ \phi_{\sharp} \circ q_{\beta} = \phi_{*}q_{\alpha} \circ \phi_{\sharp}$ 
	\end{defn}
	
	Again, such a morphism $\phi_{\sharp}:  \wh{\Sym}_{\O(U_{\beta}) }( L_{\beta}^{\bullet \geq 2}[-1]^{\vee} )  \rightarrow \wh{\Sym}_{\O(U_{\alpha}) }( L_{\alpha}^{\bullet \geq 2}[-1]^{\vee} )$ is uniquely determined by a sequence of maps $ \phi_{\sharp}^{(k)}: L_{\beta}[-1]^{\vee} \rightarrow \odot^{k}  L_{\alpha}[1]^{\vee}$ or dually, by a sequence of maps $\phi^{k}: \bigwedge^{k} L_{\alpha}  \rightarrow L_{\beta}$. \\
	
	We say that a local analytic morphism $\phi: (L_{\alpha}^{\bullet \geq 1}, U_{\alpha}) \rightarrow (L_{\beta}^{\bullet \geq 1}, U_{\beta})$ is a quasi-isomorphism, if  $\phi^{\sharp}:  \wh{\Sym}_{\O(U_{\beta}) }( L_{\beta}^{\bullet \geq 2}[-1]^{\vee} )  \rightarrow \wh{\Sym}_{\O(U_{\alpha}) }( L_{\alpha}^{\bullet \geq 2}[-1]^{\vee} )$ is a quasi-isomorphism of cdgas.  
	
	Note that it is essentially by definition, that if a morphism $\phi^{\geq 1}: (L^{\geq 1}_{\alpha}, U_{\alpha}) \rightarrow (L^{\geq 1}_{\beta}, U_{\beta})$ is a quasi-isomorphism, then we have an isomorphism of (classical) analytic spaces 
	
	\begin{equation}\label{equalmc}
	\tau^{0} \phi: \tMC(L^{\geq 1}_{\alpha}, U_{\alpha} ) \xrightarrow{\sim} \tMC(L^{\geq 1}_{\beta}, U_{\beta} )
	\end{equation}
	
	So, we see that a quasi-isomorphism between analytic $L_{\infty}$-algebras in positive degrees induces an equivalence between their corresponding Maurer-Cartan germs. Therefore, if one had in question an analytic space $X$ such that $X \cong \tMC(L,U)$ for some analytic $L_{\infty}$-algebra $L$ with convergence domain $U$, then replacing $L$ by a quasi-isomorphic $L'$ results in an equivalent dg enhancement of $X$. 
	
	We will see that this principle will allow us to define a ``discretization of Chern-Simons theory": while Chern-Simons theory conventionally works with an infinite dimensional $L_{\infty}$-algebra $L_{\Omega}$ (which happens to actually be a differential Lie algebra), one can consider instead a finite dimensional homotopy retract of $L_{\Omega}$.\\
	
	As we have explained, local dg analytic spaces always correspond to analytic $L_{\infty}$-algebras. Therefore, it is natural to ask in what way do the local quasi-smooth dg spaces introduced in Chapter $1$ correspond to analytic $L_{\infty}$-algebras. 
	
	\begin{prop} \label{localquasismooth}
	Suppose that $\uU_{\mu} = (U_{\mu}, \lambda: U_{\mu} \rightarrow \CC^{d})$ is a local quasi-smooth dg space as in \ref{localdgspace}, with basepoint $\0_{\mu}$. Then, there exists an analytic $L_{\infty}$-algebra $L_{\mu}^{\bullet} = L_{\mu}^{\bullet \geq 1}$ and a sufficiently small open $U'_{\mu} \subset U_{\mu}$ containing $\0_{\mu}$, such that $U'_{\mu}$ is a convergence domain for $L_{\mu}^{\bullet}$, and 
	\[
	\calS( L_{\mu}, U'_{\mu} ) \cong \uU_{\mu} \resto_{ U'_{\mu}} 
	\]
	as quasi-smooth dg spaces.
	\end{prop}
	
	\begin{proof}
	As usual, $\0_{\mu}$ denotes the base point of $U_{\mu}$ (being the origin of the unit polydisc in $V_{\mu} := \CC^{\dim U_{\mu} }$). Furthermore, we have that $\lambda(\0_{\mu}) = 0$, and so the Taylor series expansion of $\lambda$ centered at $\0_{\mu}$ is in degrees $\geq 1$. Each Taylor coefficient can be described as a map:
	
	\[
	\lambda_{k}:  V_{\mu}^{\vee} \rightarrow \odot^{k} ( \CC^{d})^{\vee}
	\]
	Indeed, $\lambda: U_{\mu} \rightarrow \CC^{d}$ is an analytic mapping, whose order $k$ Taylor coefficient is simply an polynomial mapping from $U_{\mu} \rightarrow \CC^{d}$ (hence the appearance of the symmetric tensor). \\
	
	Now, we define the graded vector space $L_{\mu}^{\bullet}$ by setting $L_{\mu}^{1} = V_{\mu}$ and $L_{\mu}^{2} = \CC^{d}$. The idea is that the Taylor coefficients $\lambda_{k}:  V_{\mu}^{\vee} \rightarrow \odot^{k} ( \CC^{d})^{\vee}$ 
	assemble to define an analytic $L_{\infty}$ structure on $L_{\mu}^{\bullet}$.  Now, note that the maps $\lambda_{k}$ already have the correct degrees by definition, and the condition that $\lambda^{2} = 0$ for  $\lambda = \sum\limits_{k= 1 }^{\infty} \lambda_{k}$ is vacuously true (for degree reasons). Therefore, $(L_{\mu}^{\bullet}, \{ \lambda_{k} \}_{k = 1}^{\infty} )$ defines an $L_{\infty}$-algebra (with the linear part $\lambda^{(1)}$ of the analytic mapping being the differential). \\
	Next, note that $\lambda$ is analytic on $U_{\mu}$, so there exists a sufficiently small polydisc $U_{\mu}'$ containing $\0_{\mu}$ so that the terms $\lambda_{k}$ for all $k$ satisfy a bound
	
	\[
	\| \lambda_{k} \| < k! C^{k}
	\]
	for some constant $C$, independent from $k$. Therefore, the open polydisc $U_{\mu}'$ defines a convergence domain for the $L_{\infty}$-algebra $L_{\mu}$. It is trivial to see that by definition, the Maurer-Cartan function for $L_{\mu}$ restricted to the convergence domain $U_{\mu}'$ is given by $\lambda|_{U_{\mu}'}$. Thus, we have that 
	
	\[
	\calS( L_{\mu}, U'_{\mu} ) \cong \uU_{\mu}  \resto_{ U'_{\mu}} 
	\]
	\end{proof}

	\subsection{$\cL_{\infty}$-algebras and $L_{\infty}[1]$-bundles}

	There is an elegant construction that generalizes \ref{koszulconstruction}, given by $L_{\infty}[1]$-bundles. They allow for one to ``glue together" germs of derived spaces over a base, and it also allows for one to consider derived geometric objects that are not only concentrated in strictly positive degrees. On a level of generality that will not be required for this paper, one could form the homotopy category of $L_{\infty}[1]$-bundles and glue them homotopically to describe more general derived spaces. For more details, see: \cite{behrend2021deriveddifferentiablemanifolds}, \cite{Behrend:2023mnj} and \cite{behrend2023structureetalefibrationslinftybundles}. Similar constructions can also be found in the work of Costello, and the work of Tu (\cite{costello}, \cite{tu2014homotopylinfinityspaces}), under the name of $L_{\infty}$-spaces. \\

	\begin{defn} ($L_{\infty}[1]$-bundles) \\
	We say that $(M, \L^{\bullet}, \lambda_{\bullet})$ defines an $L_{\infty}[1]$-bundle if:
	\begin{enumerate}
		\item{$M$ is an analytic manifold, }
		\item{ $\L^{\bullet}$ is a $\ZZ$-graded vector bundle over $M$}
		\item{ $ (\lambda_{k} )_{k = 0}^{\infty}$ is a sequence of analytic bundle operations
			
			\[
			\lambda_{n}: \bigodot \L^{\bullet} \rightarrow \L^{\bullet} 
			\]
			each of degree $1$, that together assemble to satisfy the axioms of a curved $L_{\infty}[1]$ algebra (in the category of graded bundles over $M$). 
		}
	\end{enumerate}
	Note that with this definition, the curvature of the curved $L_{\infty}[1]$ algebra $(M, \L^{\bullet}, \lambda_{\bullet} )$ is given by an analytic section $\lambda_{0}: M \rightarrow \L^{1}$ of the bundle $\L^{1}$.
	\end{defn}
	
	\begin{rmk}
	Note the shift in the grading: these are referred to as $L_{\infty}[1]$-bundles, as they parametrize shifted $L_{\infty}$-algebras (i.e. $L_{\infty}[1]$-algebras). The shift in the grading is more convenient for describing the bundle construction, as we can describe the structure maps $\lambda_{n}$ as \emph{graded symmetric} maps of \emph{degree $1$} as opposed to graded skew-symmetric maps of degree $2-n$. Furthermore, from the point of view of derived geometry, the shifted grading is more natural, as $L^{1}$ should be thought of as the ``base manifold", so that the grading of an $L_{\infty}$-algebra should be shifted before being thought of as a derived space. Of course, the category of $L_{\infty}$-algebras and $L_{\infty}[1]$-algebras are entirely equivalent. The shift functor is known as the \emph{d\'eclage isomorphism}. For more details on $L_{\infty}[1]$-algebras, see \cite{Manetti2022}.
	\end{rmk}

	Applying the Koszul duality correspondence to a $L_{\infty}[1]$-bundle gives a sheaf of $\ZZ$-graded cdgas over $M$
	
	\[
	\Ss(\L,M) := \Big( \wh{S}_{\O_{M}}( \L^{\vee}) , q = \sum q^{k} \Big)   
	\]
	where as before, $q^{k} = (\lambda_{k})^{\vee}$. \\
	
	In the language of derived differential geometry, $\Ss(\L,M)$ would be referred to as a ``$Q$-manifold" \cite{Voronov_2019}.
	We could refer to these as ``stacky" dg analytic spaces, but instead we will just consider $L_{\infty}[1]$-bundles and the corresponding Koszul dual ``stacky" dg spaces as interchangeable objects.
	
	\begin{notation}
	As a matter of convenience, we may abbreviate the notation $(M, \L, \lambda)$ to simply just $\L(M)$ -- where it is understood that for $\L(M)$ to be an $L_{\infty}[1]$-bundle, there is a base manifold $M$ and a curved $L_{\infty}[1]$-structure on $\L$ given by $\lambda = (\lambda)_{k = 0}^{\infty}$.  
	\end{notation}
	
	The definitions of morphisms and quasi-isomorphisms of $L_{\infty}[1]$-bundles are exactly the same as how they are defined in \ref{Linftymorphisms}. That is,
	
	\begin{defn}
	A morphism between $L_{\infty}[1]$-bundles $(M_{\alpha}, \L_{\alpha}, \lambda_{\alpha})$ and $(M_{\beta}, \L_{\beta}, \lambda_{\beta})$ is defined as an analytic map $\phi: M_{\alpha} \rightarrow M_{\beta}$, along with a morphism of sheaf-cdgas 
	
	\[
	\phi^{\#} : \Ss(\L_{\beta}, M_{\beta}) \rightarrow \phi_{*} \Ss(\L_{\alpha}, M_{\alpha})
	\]
	such that $\phi^{\#} \circ q_{\beta} =   \phi_{*}( q_{\alpha} )$
	\end{defn}

	Note that \ref{localquasismooth} now appropriately generalizes: an $L_{\infty}[1]$-bundle for where there is only 1 bundle $\L^{\bullet} = \L^{1}$ in degree $1$ is exactly the data of a quasi-smooth space as defined in chapter $1$. In the notation of chapter $1$, we can write $E = \L^{1}$ and $\lambda = \lambda_{0}$
	
	\begin{lem}
	There is a one-to-one correspondence between analyic \qs spaces, as defined in \ref{qsspace}, and $L_{\infty}[1]$-bundles $(M, \L, \lambda)$ such that $\L^{\bullet} = \L^{1}$ is concentrated entirely in degree $1$.
	\end{lem}
	
	\begin{notation} Given a local analytic $L_{\infty}$-algebra $(L,U)$ we will use the notation $\bL(U)$ to denote that $L_{\infty}[1]$-bundle associated to $(L,U)$. Hence, both $\bL(U)$ and $(L,U)$ refer to the same data. Recall that we refer to local analytic $L_{\infty}$-algebras as $\cL_{\infty}$-algebras, as shorthand terminology. Note that the convergence domain $U \subset L^{1}$ is not necessarily \emph{pointed} (i.e. there may not be a canonical choice of Maurer-Cartan element associated to an $\cL_{\infty}$-algebra). 
	\end{notation}

	\begin{rmk}
	We will primarily be interested in discussing $\cL_{\infty}$-algebras (i.e. local analytic $L_{\infty}$-algebras $(L,U)$, for some convergence domain $U$). We briefly introduced the notion of an $L_{\infty}[1]$-bundle only to contextualize the type of derived geometry carried by an $\cL_{\infty}$-algebra. Furthermore, as analytic quasi-smooth spaces are $L_{\infty}[1]$-bundles, it makes sense to have a morphism between an analytic quasi-smooth space and an $\cL_{\infty}$-algebra.
	\end{rmk}

	\subsubsection*{Differential graded tangent spaces}
	
	Suppose that $\L(U)$ is an $L_{\infty}[1]$-bundle, and that $\mu \in \tMC( \L(U) )$. Then, recall that the fiber of $(\L, \lambda)$ at $\mu$ defines an $L_{\infty}$-algebra which we denote as $(L_{\mu}, d_{\mu}, l_{2}^{\mu}, \cdots )$. This $L_{\infty}$-algebra is linearized by the underlying chain complex $(L_{\mu}, d_{\mu})$.
	
	\begin{defn}
	We define the dg tangent space of $\L(U)$ at $\mu \in \tMC(\L(U))$ to be the underlying chain complex $(L_{\mu}^{\bullet}, d_{\mu})$ of the $L_{\infty}$-algebra $(L_{\mu}, d_{\mu}, l_{2}^{\mu}, \cdots )$. That is,

	\[
	T_{\mu}^{\bullet} \L(U)  = L_{\mu}^{\bullet} = [ \cdots \xrightarrow{d_{\mu}} L^{1}_{\mu} \xrightarrow{d_{\mu}} L_{\mu}^{2} \xrightarrow{d_{\mu}} \cdots ]
	\]
	\end{defn}

	\begin{defn} \label{quasismoothLinftybundles}
	We say that an $L_{\infty}[1]$-bundle $\L(U)$ is \emph{\qs} if for each $\mu \in \tMC(\L(U))$, 
	the cohomology of $T_{\mu}\L(U)$ is concentrated in degrees $[0,1]$. That is, $H^{k}( L_{\mu} ) = 0$ for all $k \neq 0,1$. Thus, this terminology applies as well to $\cL_{\infty}$-algebras: if $(L,U)$ is a local analytic $L_{\infty}$-algebra, and $U$ is chosen is that $H^{k}(L_{\mu}) = 0$ for all $k \neq 0,1$, for $\mu \in \MC(L,U)$.  
	\end{defn}
	
	{\bfseries  ** Throughout this paper, we will assume that all of our  $\cL_{\infty}$-algebras  are quasi-smooth, unless stated otherwise. ** }

	\section{Non-negatively graded $\cL_{\infty}$-algebras}
	
	In the previous section, we saw how the positively graded part of an analytic $L_{\infty}$-algebra $L$ may be naturally identified with a dg analytic space $\Ss(L^{\bullet \geq 1},U)$ -- such that if $L_{\alpha}^{\bullet \geq 1}$ and $L_{\beta}^{\bullet \geq 1}$ are quasi-isomorphic then the two dg spaces $\calS(L_{\alpha}^{\bullet \geq 1}, U_{\alpha})$ and $\calS(L_{\beta}^{\bullet \geq 1}, U_{\beta})$ have isomorphic Maurer-Cartan germs.\\
	
	In this section, we would like to understand how the non-positive graded pieces of an analytic $L_{\infty}$-algebra $L$ fit into the geometric picture. Roughly speaking, the $L^{\bullet \leq 0}$ part of an analytic $L_{\infty}$-algebra encodes a ``Lie algebroid up to homotopy" structure over the dg space $\calS( L^{\bullet \geq 1}, U)$. Where we have seen how positively graded analytic $L_{\infty}$-algebras encode local models of dg analytic spaces, analytic $L_{\infty}$-algebras with no grading restrictions encode infinitesimal or local descriptions of homotopy groupoids over dg analytic spaces.

	In this paper, we will restrict our attention to non-negatively graded analytic $L_{\infty}$-algebras. That is, we will only consider analytic $L_{\infty}$-algebras that are $\geq 0$ in homological degrees. The corresponding $L_{\infty}[1]$-bundles produced out of non-negatively graded analytic $L_{\infty}$-algebras will be concentrated in degrees $[-1,n]$.  We will see that the degree $L^{0}$ part of $L^{\bullet}$ encodes a ``gauge action" on $L^{\bullet \geq 1}$, and that the statement analogous to \ref{equalmc} will be that a quasi-isomorphism between $L_{\alpha}$ and $L_{\beta}$ induces an isomorphism between local Maurer-Cartan loci, up to local gauge equivalence. \\
	
	Suppose that $(L,U)$ is a non-negatively graded local analytic $L_{\infty}$-algebra with a non-trivial $L^{0}$ term. Then, in particular, for any $\mu \in \tMC(L,U)$, the dg tangent space $T_{\mu}^{\bullet} \Ss(L,U) = [L_{\mu}^{1} \xrightarrow{d_{\mu}} L_{\mu}^{2} \xrightarrow{d_{\mu}} \cdots ]$ can be extended into homological degree $-1$
	\[
	[ L_{\mu}^{0} \xrightarrow{d_{\mu}} T_{\mu}^{\bullet} \Ss(L,U) ] =  [L_{\mu}^{0} \xrightarrow{d_{\mu}} L_{\mu}^{1} \xrightarrow{d_{\mu}} L_{\mu}^{2} \xrightarrow{d_{\mu}} \cdots ]
	\]
	so that the equation $d_{\mu}^{2} = 0$ remains true. In other words, $d_{\mu}^{1} d_{\mu}^{0} = 0$, so that $\im(d_{\mu}^{0}: L^{0}_{\mu}  \rightarrow L^{1}_{\mu}) \subset \ker(d_{\mu}^{1}: L_{\mu}^{1} \rightarrow L_{\mu}^{2})$.  Therefore, as we vary $\mu \in \tMC(L,U)$ we see that the various twisted differentials $d_{\mu}^{0}$ define a distribution on $U$, that is tangent to $\tMC(L,U) \subset U$ at every $\mu \in \tMC(L,U)$. We refer to this distribution, as the \emph{gauge distribution} of the analytic $L_{\infty}$-algebra $L$. \\

	\begin{notation}
	We refer to the map $\delta^{0}: L^{0} \times L^{1} \rightarrow L^{1}$ defined by $\delta^{0}(a, x) := d_{x}(a) = \lambda_{1}^{0}(a,x)$, as the \emph{anchor map}. By abuse of terminology, we may also refer to $d_{x}$ at a fixed $x$ as the anchor map. Note that the anchor map is well-defined whether or not $x$ is a Maurer-Cartan element. 
	\end{notation}

	Furthermore, purely for degree reasons, we see that the higher operations $l^{\mu}_{k}$ vanish when restricted to $L_{\mu}^{0}$ except for $l_{2}^{\mu}$. Therefore, $(L_{\mu}^{0}, l_{2}^{\mu})$ is a (ungraded) Lie algebra, with an anchor map $d_{\mu}: L^{0}_{\mu} \rightarrow L_{\mu}^{1}$. For each $\mu$, we refer to the Lie algebra $(L_{\mu}^{0}, l_{2}^{\mu})$ as the gauge algebra at $\mu$. Sometimes, we may write $\g_{\mu} = (L_{\mu}^{0}, l_{2}^{\mu})$.

	\begin{notation} \label{gaugenotation}
	In this paper, we may use the notation $\g$ to denote the gauge algebra $L^{0}$ of an $\cL_{\infty}$-algebra $\bL(U)$ concentrated in non-negative degrees.
	\end{notation}

	\subsection{Local gauge equivalences and homotopies}

	Formulating the correct notion of an $L_{\infty}$ homotopy is a subtle matter. It is generally accepted that a correct notion of $L_{\infty}$ homotopies is given by the \emph{Sullivan model} of an $L_{\infty}$ homotopy. For more details, see \cite{guan2020gaugeequivalencecompletelinftyalgebras}, \cite{Manetti2022}, \cite{vangarderen2024cyclicainfinityalgebrascalabiyau}, \cite{Getzler_2009}, \cite{getzler2018maurercartanelementshomotopicalperturbation}. 
	
	\subsubsection*{The $L_{\infty}$ path object and homotopies}
	
	\begin{defn}
	We write $\Omega_{\Delta^{1}}$ to be the cdga of piecewise analytic $\CC$-valued differential forms on the unit interval $[0,1]$. Then, for any analytic $L_{\infty}$-algebra $L$, we define $L \hat{\otimes} \Omega_{\Delta^{1}}$ to be the analytic $L_{\infty}$-algebra with $L_{\infty}$ operations defined by: \\
	
	\[
	l_{k}^{\Omega_{\Delta^{1}}} = l_{k} \otimes m_{k}
	\]
	where $m_{k}$ defined the $k$-ary multiplication operator in $\Omega_{\Delta^{1}}$.
	\end{defn}
	Note that $ L \hat{\otimes} \Omega_{\Delta^{1} } =$
	\[
	\{ \gamma(t) + \eta(t) dt \mid \gamma(t) \text{ is a path in } L^{1}, \text{ } \eta(t) \text{ is a path in } L^{0}, \text{ both piecewise analytic in } t \}
	\]
	
	\begin{defn}
	Suppose that $(L,U)$ is a local anaytic $L_{\infty}$-algebra, and that $\mu, \nu \in \tMC(L,U)$. Then, we say that $\mu \sim \nu$ are (local) gauge equivalent Maurer-Cartan elements in $U$, if there exists a Maurer-Cartan element $\gamma + \eta dt \in L \hat{\otimes} \Omega_{\Delta^{1}}$ such that $\gamma(t) \in U \subset L^{1}$ for all $t$, and $\gamma(0) = \mu$, $\gamma(1) = \nu$. \\
	
	We refer to this relation between Maurer-Cartan elements in $U$ as local gauge equivalence.
	\end{defn}
	
	\begin{lem} \label{equivalencerel}
	Suppose that $(L,U)$ is a local analytic $L_{\infty}$-algebra, with convergence domain $U$. Then, local gauge equivalence between Maurer-Cartan elements in $U$ indeed defines an equivalence relation.
	\end{lem}
	\begin{proof}
	In what follows, $\sim$ denotes local gauge equivalence. 
	\begin{enumerate}
		\item{If $\mu \in \tMC(L,U)$ then we have that $\mu \sim \mu$. Indeed, we define the constant path $\gamma_{\mu}(t) = \mu$ for all $t \in [0,1]$. Then $\gamma_{\mu}(t) \in \tMC( L \hat{\otimes} \Omega_{\Delta^{1} } )$, as $\frac{d}{dt} \gamma_{\mu}(t) = 0$. }
		\item{If $\mu, \nu \in \tMC(L,U)$ then we have that $\mu \sim \nu$ iff $\nu \sim \mu$. If $\mu \sim \nu$ via a local gauge equivalence $\gamma_{\mu \nu}(t) + \eta_{\mu \nu}(t) dt \in \tMC( L \hat{\otimes} \Omega_{\Delta^{1}} )$, where $\gamma_{\mu \nu}(0 ) = \mu$ and $\gamma_{\mu \nu}(1) = \nu$, the reverse time path $\gamma_{\mu \nu}(1-t) + \eta_{\mu \nu}(1-t) dt$ defines a local gauge equivalence from $\nu$ to $\mu$.  }
		\item{Suppose that $\mu \sim \nu$ and $\nu \sim \zeta$ , for $\mu, \nu, \zeta \in \tMC(L,U)$. Then, $\mu \sim \zeta$ via the concatenation of piecewise analytic paths. }
	\end{enumerate}
	\hfill
	\end{proof}

	\begin{ex}
	Suppose that $L$ is a non-negatively graded differential graded Lie algebra, with a non-trivial $L^{0}$ term. Then,  a gauge equivalence between Maurer-Cartan elements $\mu$ and $\nu$ is a Maurer-Cartan element $\gamma(t) + \eta(t)dt \in L \hat{\otimes} \Omega_{\Delta^{1}}$, such that $\gamma(0) = \mu$ and $\gamma(1) = \nu$. Explicitly, we can write down the Maurer-Cartan equation for $L \hat{\otimes} \Omega_{\Delta^{1}}$:
	
	\[
	d( \gamma(t) + \eta(t)dt) + \frac{1}{2} [ \gamma(t) + \eta(t) dt , \gamma(t) + \eta(t) dt ] = 0 
	\]
	
	which splits into two equations, for all $t \in [0,1]$:
	
	\begin{enumerate}
		\item{$d \gamma(t) + \frac{1}{2}[\gamma(t), \gamma(t)] = 0$}
		\item{$\frac{d}{dt} \gamma(t) + d \eta(t) + [\gamma(t), \eta(t) ] = 0$} 
	\end{enumerate}
	
	Equation $1.$ says that $\gamma(t)$ is a path through the Maurer-Cartan locus $\tMC(L,U)$, and equation $2.$ says that the direction of the path $\gamma(t)$ is along the image of $\eta(t)$ via the anchor map $d_{\gamma(t)}: L_{\gamma(t)}^{0} \rightarrow L_{\gamma(t)}^{1}$ 
	\end{ex}
	
	\begin{rmk} \label{gaugegroup}
	Suppose that $L$ is a non-negatively graded differential graded Lie algebra, with a non-trivial $L^{0}$ term. Suppose further, that the gauge algebra $\g = (L^{0}, l_{2} = [,])$ can be identified in coordinates, with a matrix Lie algebra acting on the vector space $L^{\geq 1}$. Or alternatively, in a more coordinate free way, suppose that we can identify $\g$ with an endomorphism Lie algebra of the vector space $L^{\geq 1}$. Now, suppose that $G = \exp(L^{0})$ is a well defined analytic Lie group. Then, $G$ is referred to as the \emph{gauge group}. It acts on $L^{\bullet \geq 2}$ by conjugation $x \mapsto g^{-1}xg$,  and on $L^{1}$ via the gauge action:
	
	\[
	g \cdot x = g^{-1} x g + g^{-1}dg
	\]
	
	Then, by direct calculation, one can verify that: for a fixed $x \in L^{1}$, the derivative of the action map $g \mapsto g \cdot x$ at the identity, in the direction of $a \in \g$ is given by the twisted differential $d_{x}(a)$. Therefore, we see in such a situation, the anchor map agrees with the anchor map defined by the associated action algebroid of a gauge group action. 
	\end{rmk}
	
	\begin{defn}
	We say that morphisms of local analytic $L_{\infty}$-algebras $F, G: (L_{\alpha}, U_{\alpha}) \rightarrow (L_{\beta}, U_{\beta})$ 
	are homotopic, if there exists a local analytic $L_{\infty}$ morphism $H: (L_{\alpha}, U_{\alpha} ) \rightarrow (L_{\beta} \hat{\otimes} \Omega_{\Delta^{1}}, U_{\beta} ^{\Delta^{1}} )$ such that $H |_{t = 0 } = F$ and $H |_{t = 1}  = G$.
	\end{defn}
	
	\begin{defn}
	We say that analytic $L_{\infty}$-algebras $L_{\alpha}$ and $L_{\beta}$ are homotopy equivalent, if there exist analytic $L_{\infty}$ morphisms $F: L_{\alpha} \rightarrow L_{\beta}$ and $G: L_{\beta} \rightarrow L_{\alpha}$ such that $F \circ G \sim \Id_{L_{\beta}}$ and $G \circ F \sim \Id_{L_{\alpha}}$ 
	\end{defn}
	
	\begin{defn}
	We say that an $\cL_{\infty}$-algebra $(L_{\beta}, U_{\beta})$ is a homotopy retract of an $\cL_{\infty}$-algebra $(L_{\alpha}, U_{\alpha})$ if there exists
	
	\begin{itemize}
		\item{an analytic $L_{\infty}$ morphism $\mbf{I}_{\beta \alpha}: L_{\beta} \rightarrow L_{\alpha}$ and an analytic $L_{\infty}$ morphism $\mbf{P}_{\alpha \beta}: L_{\alpha} \rightarrow L_{\beta}$ satisfying $\mbf{P}_{\alpha \beta} \circ \mbf{I}_{\beta \alpha}   = \Id_{L_{\beta} }$ }
		\item{ an analytic homotopy $\mbf{H}_{\alpha}: L_{\alpha} \rightarrow  L_{\alpha} \hat{\otimes} \Omega_{\Delta^{1} }$ such that $\mbf{H}_{\alpha} |_{t = 0 } = \Id_{L_{\alpha} }$ and $\mbf{H}_{\alpha} |_{t = 1} = \mbf{I}_{\beta \alpha} \circ \mbf{P}_{\alpha \beta}$ }
	\end{itemize}

	\end{defn}
	
	If $(L_{\beta}, U_{\beta})$ is a homotopy retract of $(L_{\alpha}, U_{\alpha})$ via the data $(\mbf{I}_{\beta \alpha}, \mbf{P}_{\alpha \beta}, \mbf{H}_{\alpha})$ then we say that the triple $(\mbf{I}_{\beta \alpha}, \mbf{P}_{\alpha \beta}, \mbf{H}_{\alpha})$ is an $L_{\infty}$ homotopy retract context. It is diagrammatically depicted as 
	
	\[
	L_{\beta}  \rightleftarrowstack{\mbf{I}_{\beta \alpha }}{\mbf{P}_{\alpha \beta} }    L_{\alpha} \circlearrowleft{ \mbf{H}_{\alpha}  }
	\]
	
	If we furthermore have that $\mbf{H}_{\alpha} \circ \mbf{I}_{\beta \alpha} = 0$ and $\mbf{H}_{\alpha}^{2} = 0 = \mbf{P}_{\alpha \beta} \circ \mbf{H}_{\alpha}$, then we say that the homotopy retract context  $(\mbf{I}_{\beta \alpha}, \mbf{P}_{\alpha \beta}, \mbf{H}_{\alpha})$ defines a \emph{strong} homotopy retract.

	\begin{defn}
	If $F: (L_{\alpha}, U_{\alpha})  \rightarrow  (L_{\beta}, U_{\beta})$ is a morphism of $\cL_{\infty}$-algebras, we define

	\[ 
	F_{*}( \gamma(t) + \eta(t)dt) := (F_{*}\gamma)(t) + (F_{*} \eta)(t)  \in L_{\beta} \wh{\otimes} \Omega_{\Delta^{1} }
	\]
	where for any fixed $t$, $(F_{*}\gamma)(t) := F_{0}(\gamma(t))$ and $(F_{*}\eta)(t) = F^{(1)}_{\gamma(t)}(\eta(t) )$. Recall that we write
	\begin{align*}
		& F_{0}(\mu) := \sum\limits_{k=0}^{\infty} \frac{1}{k!} F^{(k)}( \mu^{\odot k} )   \\
		& F^{(1)}_{\mu}(x) := \sum \limits_{k = 0}^{ \infty} F^{(1 + k)}( x, \mu^{\odot k} ) 
	\end{align*}
	\end{defn}
	
	\begin{lem}
	If $\gamma(t) + \eta(t)dt \in \tMC( L_{\alpha} \wh{\otimes}  \Omega_{\Delta^{1} } )$ is a homotopy between the Maurer-Cartan elements $\gamma(0)$ and $\gamma(1)$ in $L_{\alpha}$, then $F_{*}( \gamma(t) + \eta(t) dt ) \in \tMC( L_{\beta} \wh{\otimes} \Omega_{\Delta^{1} } )$, so that $F_{*}( \gamma(t) + \eta(t) dt )$ is a homotopy between the Maurer-Cartan elements $F_{0}(\gamma(0))$ and $F_{0}(\gamma(1))$ in $L_{\beta}$. 
	\end{lem}

	\begin{proof}
	If $\gamma(t) \in \tMC(L_{\alpha}, U_{\alpha})$ so that $d_{\gamma(t)} \eta(t) \in T_{\gamma(t)} ( U_{\alpha} )$, then note that $F_{\gamma(t)}^{(1)}(\eta(t)) = (DF_{0})_{\gamma(t)} \big( \eta(t) \big)$ so that 
	\[
	\frac{d}{dt} \big( F_{0}( \gamma(t) ) \big) = (DF_{0})_{\gamma(t) } \big( d_{\gamma(t)}( \eta(t) ) \big) = F_{\gamma(t)}^{(1)}(\eta(t))
	\]  
	therefore, we have that the differential equation
	\[
	\frac{d}{dt} (F_{*}\gamma)(t) = d_{ F_{*}\gamma(t) } ( F_{*} \eta(t) ) 
	\]
	is satisfied for all $t$. So, $F_{*}( \gamma(t) + \eta(t)dt )$ is indeed a Maurer-Cartan element in $\tMC(L_{\beta}, U_{\beta})$ (i.e. a gauge equivalence). Clearly, one has by definition, the initial conditions: $F_{*}( \gamma(t) + \eta(t)dt ) |_{t = 0} =  F_{0}(\gamma(0))$ and $F_{*}( \gamma(t) + \eta(t)dt ) = F_{0}( \gamma(1) )$.
	\end{proof}
	
	\subsubsection*{The $1$-truncated Maurer-Cartan nerve $\MC_{\leq 1}$}
	
	When an $L_{\infty}$-algebra has a non-trivial degree $0$ graded piece, then is no longer generally true that quasi-isomorphic, or homotopic analytic $L_{\infty}$-algebras $(L,U)$ and $(L',U')$ have isomorphic Maurer-Cartan loci in that $\tMC(L,U) \cong \tMC(L',U')$. Instead, $(L,U)$ and $(L',U')$ have equivalent Martan-Cartan loci, up to gauge equivalence. Thus, in generality, the Maurer-Cartan locus of an analytic $L_{\infty}$-algebra $(L,U)$ should be constructed as a (higher) groupoid which parametrizes Maurer-Cartan locus up to gauge (and possibly higher gauge) equivalences. This is known as the Maurer-Cartan, see nerve \cite{Getzler_2009}, \cite{getzler2018maurercartanelementshomotopicalperturbation} for more details. In this paper, we will only work with non-negatively graded ``\qs" $L_{\infty}$-algebras (\ref{quasismoothLinftybundles}), so that we do not need higher groupoids to describe gauge equivalence. \\
	
	We would like to make a statement of the form, where $\sim$ denotes local gauge equivalence:\\
	
	`` if $(L_{\alpha}, U_{\alpha})$ and $(L_{\beta}, U_{\beta})$ are homotopic local analytic $L_{\infty}$-algebras, then $[\tMC(L_{\alpha}, U_{\alpha}) /\sim ]  \cong [\tMC(L_{\beta}, U_{\beta})/\sim]$ " \\

	In general, making sense of the quotient as an analytic space can be complicated and wrought with technicalities, so we take the approach of understanding gauge equivalence by way of simplicial sets. To simplify matters, we will restrict our attention to the $1$-truncated Maurer-Cartan nerve.  
	
	\begin{cons}
	Given a local analytic $L_{\infty}$-algebra $(L,U)$, concentrated  in non-negative degrees.
	We define the $1$-truncated simplicial Maurer-Cartan nerve $\MC_{\leq 1}(L,U)$  to be the following \textbf{simplicial set}:
	
	\begin{itemize}
		\item{The space of vertices $\MC_{0}(L,U)$ is given by the set $\tMC(L,U)$}
		\item{The space of edges $\MC_{1}(L,U)$ is given by the set $\tMC(L \otimes \Omega_{\Delta^{1} }, U)$}
	\end{itemize}
	
	Furthermore, we must define face maps $d_{0}, d_{1}: \MC_{1}(L,U) \rightarrow \MC_{0}(L,U)$ and a single degeneracy map $s_{0}: \MC_{0}(L,U) \rightarrow \MC_{1}(L,U)$. 
	\begin{itemize}
		\item{ For each edge $\gamma(t) + \eta(t)dt \in \MC_{1}(L,U)$, we define $d_{0}( \gamma(t) + \eta(t)dt) = \gamma(0)$  and $d_{1}( \gamma(t) + \eta(t)dt) = \gamma(1)$. That is, two Maurer-Cartan elements $\mu$ and $\mu'$ in $\MC_{0}(L,U)$ are ends points of a non-degenerate $1$-simplex in $\MC_{1}(L,U)$ if there is a gauge equivalence $\gamma(t) + \eta(t)dt \in \tMC( L \otimes \Omega_{\Delta^{1} } )$ between them. }
		
		\item{ For each vertex $\mu \in \MC_{0}(L,U)$, we define $s_{0}(\mu) = \gamma_{\mu}(t) \in \MC_{1}(L,U)$, where $\gamma_{\mu}(t) = \mu$ for all $t \in [0,1]$ is the constant path at $\mu$. Clearly, this is a Maurer-Cartan element in $L \otimes \Omega_{\Delta^{1} }$. }

	\end{itemize}
	By \ref{equivalencerel}, the gauge equivalence is an equivalence relation on $\tMC(L,U)$. Therefore, $\MC_{\leq 1}(L,U)$ is a category. In fact, $\MC_{\leq 1}(L,U)$ is a groupoid.
	\end{cons}

	\begin{prop} \label{mcgaugeiso}
	
	Suppose that $(L_{\alpha}, U_{\alpha})  \rightleftarrowstack{F}{ G}    (L_{\beta}, U_{\beta})$ is a pair of homotopy equivalences that are homotopy inverses of each other, for local analytic $L_{\infty}$-algebras $(L_{\alpha}, U_{\alpha}) $ and $(L_{\beta}, U_{\beta})$ that are concentrated in non-negative degrees. Then, there is an induced pair of weak equivalences between their $1$-groupoids:
	
	\[
	\MC_{\leq 1}(L_{\alpha},U_{\alpha})  \rightleftarrowstack{ F_{\bullet}  }{  G_{\bullet} }   \MC_{\leq 1}(L_{\beta},U_{\beta}) 
	\]
	\end{prop}
	
	\begin{proof}
	Recall that, given morphisms $F$ and $G$, we have induced simplicial morphisms:
	
	\[
	\begin{tikzcd}[column sep = 0.7em]
		\tMC( L_{\alpha} \otimes \Omega_{\Delta^{1} } ) \arrow[r, shift left] \arrow[r, shift right] \arrow[d, "F_{*}", shift left] & \tMC(L_{\alpha}, U_{\alpha} ) \arrow[d, "F", shift left]\\
		\tMC( L_{\alpha} \otimes \Omega_{\Delta^{1} } )  \arrow[r, shift left] \arrow[r, shift right] \arrow[u, "G_{*}", shift left] & \tMC(L_{\beta}, U_{\beta}) \arrow[u, "G", shift left]
	\end{tikzcd}
	\]
	We will write  $F_{\bullet} = (F_{0}, F_{1})$ and $G_{\bullet} = (G_{0}, G_{1})$, where $F_{0} := F$, $F_{1} := F_{*}$ and $G_{0} := G$, $G_{1} := G_{*}$. \\
	
	We would like to show that $F_{\bullet}$ and $G_{\bullet}$ are weak equivalences between the simplicial sets $\MC_{\leq 1}( L_{\alpha} , U_{\alpha} )$ and $\MC_{\leq 1} (L_{\beta} , U_{\beta})$, by showing that there are simplicial homotopies 
	\begin{enumerate}
		\item{$H_{L_{\alpha}}: \MC_{\leq 1} ( L_{\alpha}, U_{\alpha}) \times \Delta^{1}  \rightarrow \MC_{\leq 1} ( L_{\alpha}, U_{\alpha})$ between $G_{\bullet} \circ F_{\bullet}$ and $\Id_{ \MC_{\leq 1}( L_{\alpha}, U_{\alpha} ) }$  }
		\item{$H_{L_{\beta}}: \MC_{\leq 1} ( L_{\beta}, U_{\beta}) \times \Delta^{1}  \rightarrow \MC_{\leq 1} ( L_{\beta}, U_{\beta})$ between $F_{\bullet} \circ G_{\bullet}$ and $\Id_{ \MC_{\leq 1}( L_{\beta}, U_{\beta}) }$. }
	\end{enumerate}
	
	As usual, this amounts to defining a family of maps $(H_{L_{\alpha}})_{n}: \MC_{\leq 1}(L_{\alpha}, U_{\alpha})_{ n }  \rightarrow  \MC_{\leq 1}(L_{\alpha}, U_{\alpha})_{ n + 1}$,  $(H_{L_{\beta}})_{n}: \MC_{\leq 1}(L_{\beta}, U_{\beta})_{ n }  \rightarrow  \MC_{\leq 1}(L_{\beta}, U_{\beta})_{ n + 1}$  satisfying compatibility with face and degeneracy maps. As our simplicial sets are $1$-truncated, we only need to define $(H_{L_{\alpha}})_{0}$ and $(H_{L_{\beta}})_{0}$. \\
	
	We define:
	\begin{enumerate}
		\item{$(H_{L_{\alpha}})_{0}: \MC_{0}( L_{\alpha}, U_{\alpha} ) \rightarrow \MC_{1}( L_{\alpha}, U_{\alpha} )$ is defined by:
			
			\[
			\mu \mapsto h_{L_{\alpha}}(\mu) = \gamma_{\alpha}(t) + \eta_{\alpha}(t)dt
			\]
			where $h_{L_{\alpha}} : L_{\alpha} \rightarrow L_{\alpha} \otimes \Omega_{\Delta^{1}}$ is the  $L_{\infty}$ local homotopy between $G \circ F$ and $\Id_{L_{\alpha}}$. 
		}
		\item{$(H_{L_{\beta}})_{0}: \MC_{0}( L_{\beta}, U_{\beta} ) \rightarrow \MC_{1}( L_{\beta}, U_{\beta} )$ is defined by:
			
			\[
			\mu \mapsto h_{L_{\beta}}(\mu) = \gamma_{\beta}(t) + \eta_{\beta}(t)dt
			\]
			where $h_{L_{\beta}} : L_{\beta} \rightarrow L_{\beta} \otimes \Omega_{\Delta^{1}}$ is the  $L_{\infty}$ local homotopy between $F \circ G$ and $\Id_{L_{\beta}}$. 
		}
	\end{enumerate}
	As our simplicial sets are $1$-truncated, the only non-trivial requirement to define a simplicial homotopy is given by $d_{0} (H_{L_{\alpha } } )_{0} = G \circ F = (G_{\bullet} \circ F_{\bullet})_{0}$ , $d_{1} (H_{L_{\alpha } } )_{0} = \Id_{L_{\alpha} } = (\Id_{ \MC_{\leq 1}( L_{\alpha}, U_{\alpha} ) } )_{0}$ and 
	$d_{0} (H_{L_{\beta } } )_{0} = F \circ G$ , $d_{1} (H_{L_{\beta} } )_{0} = \Id_{L_{\beta} }$ which hold by definition of our maps $(H_{L_{\alpha } } )_{0}$ and $(H_{L_{\beta} } )_{0}$.
	Therefore, $\MC_{\leq 1}(L_{\alpha}, U_{\alpha})$ and $\MC_{\leq 1}( L_{\beta}, U_{\beta})$ are weakly equivalent as simplicial sets as we have simplicial morphisms 
	
	\[
	\MC_{\leq 1}(L_{\alpha},U_{\alpha})  \rightleftarrowstack{ F_{\bullet}  }{  G_{\bullet} }   \MC_{\leq 1}(L_{\beta},U_{\beta})
	\]
	
	such that $F_{\bullet} \circ G_{\bullet} \sim (\Id_{ \MC_{\leq 1}( L_{\alpha}, U_{\alpha} ) } )_{\bullet}$  and $G_{\bullet} \circ F_{\bullet} \sim (\Id_{ \MC_{\leq 1}( L_{\beta}, U_{\beta} ) } )_{\bullet}$ where $\sim$ denotes equivalence under simplicial homotopy.
	\end{proof}
	
	Therefore, we see that if local analytic $L_{\infty}$-algebras $(L_{\alpha}, U_{\alpha})$ and $(L_{\beta} , U_{\beta})$ are homotopy equivalent and are concentrated in non-negative degrees, then they have homotopy equivalent $1$-truncated Maurer-Cartan nerves (i.e. Maurer-Cartan local groupoids). This is the sense in which $(L_{\alpha}, U_{\alpha})$ and $(L_{\beta}, U_{\beta})$ have equivalent quotients by local gauge equivalence. \\

	\subsection{Minimal model decompositions of $L_{\infty}$-algebras}

	Recall that in chapter $1$, we discussed the matter of choosing a minimal model at a point of a quasi-smooth analytic space. Recall that a local \qs analytic space $\big( U_{\mu} , \lambda_{\mu} \big)$ is minimal if $D_{\0} \lambda_{\mu} = 0$. In this section, we define the notion of a minimal models of $\cL_{\infty}$-algebras. 
	
	\begin{defn}
	Suppose that $(L^{\bullet}, d , l_{2}, l_{3}, \cdots, )$ is  an  $L_{\infty}$-algebra. 
	\begin{enumerate}
		\item{$L^{\bullet}$ is \emph{minimal}, if $d = 0$.}   
		\item{If $H^{\bullet} \rightarrow L^{\bullet}$ is a quasi-isomorphism of $L_{\infty}$-algebras, such that $H^{\bullet}$ is minimal, then we say that $H^{\bullet}$ is a minimal model of $L^{\bullet}$. } 
	\end{enumerate}
	\end{defn}

	To produce minimal models of $L_{\infty}$-algebras, the main tools at our disposal are the homological perturbation lemma and homotopy transfer principle. This is a well-known classic tool in homological algebra. We recall it here, and then explain how applying it in ``families" leads to minimal model decompositions for  $\cL_{\infty}$-algebras. 
	
	\subsubsection*{Homological perturbation theory}
	
	For more details on homological perturbation theory, see \cite{crainic2004perturbationlemmadeformations}, \cite{getzler2018maurercartanelementshomotopicalperturbation}
	
	\begin{defn} \label{sideconditions}
	A (strong) homotopy retract context between chain complexes $(V^{\bullet}, d_{V})$ and $ (W^{\bullet}, d_{W})$ consists of a pair of maps $i$, $p$, and a chain homotopy $h: V \rightarrow V[-1]$ such that:
	
	\begin{enumerate}
		\item{$pi = \id_{W}$}
		\item{$\id_{V} - ip = [d, h]$}
		\item{$h^{2} = 0 = ph = hi$ (known as the \emph{side conditions})} 
	\end{enumerate}
	
	and it is diagrammically depicted as:
	
	\[
	(W,d_{W}) \rightleftarrowstack{i  }{p } (V, d_{V})  \circlearrowleft{h} 
	\]

	Futhermore, one says that a degree $1$ map $\mu: V \rightarrow V[1]$ is a perturbation of $d_{V}$ if $(d_{V} + \mu)^{2} = 0$. One says that $\mu$ is a small perturbation, or a convergent perturbation, if 
	
	\[
	\sum\limits_{i = 0}^{\infty} (-1)^{i} ( h \mu )^{i}  =( 1 + h \mu )^{-1} 
	\]
	
	is pointwise convergent.

	\end{defn}

	Then, the homological perturbation lemma states:
	
	\begin{lem} \label{HPL}
	Suppose that $(V, W, i,p, h)$ is a homotopy retract context:
	
	\[
	(W,d_{W}) \rightleftarrowstack{i  }{p } (V, d_{V})  \circlearrowleft{h} 
	\] 
	
	Then, if $\mu$ is a small perturbation of $d$, then the following formulas:
	
	\begin{enumerate}
		\item{ $h_{\mu} =  ( 1 + h \mu )^{-1} h$ }
		\item{$i_{\mu} = ( 1 + h \mu )^{-1} i $ }
		\item{$p_{\mu} = p ( 1 + h \mu )^{-1} = \sum\limits_{n = 0}^{\infty} g (-\mu h)^{n}$}
		\item{$d^{\mu}_{W} =   d + \sum\limits_{n = 0}^{\infty} p (-\mu h)^{n} \mu i$}
	\end{enumerate}
	
	are so that 
	
	\[
	(W,d^{\mu}_{W}) \rightleftarrowstack{i_{\mu}  }{p_{\mu} } (V, d_{V} + \mu )  \circlearrowleft{h_{\mu}} 
	\] 
	
	is another (strong) homotopy retract context. 
	\end{lem}

	\paragraph*{The (symmetric) tensor trick} \hfill \\
	Now, we will explain how the homological perturbation lemma, and the ``tensor trick" may be used to prove the homotopy transfer theorem for $L_{\infty}$-algebras. 
	
	\begin{prop}\textnormal{(Homotopy transfer theorem for $L_{\infty}$-algebras)}\\

	Suppose that $(V, W, i,p, h)$ is a homotopy retract context:
	
	\[
	(W,d_{W}) \rightleftarrowstack{i  }{p } (V, d_{V})  \circlearrowleft{h} 
	\] 
	
	Furthermore, suppose that there exists an $L_{\infty}$ structure $(V, d_{V}, l_{2}, \cdots)$ lifting $(V, d_{V})$. Then, there eixsts an $L_{\infty}$ structure $(W, d_{W}, l_{2}^{W}, \cdots)$ lifting $(W, d_{W})$, along with $L_{\infty}$ quasi-isomorphisms $I: (W, d_{W}, l_{2}^{W}, \cdots) \hookrightarrow (V, d_{V}, l_{2}, \cdots)$ and $P: (V, d_{V}, l_{2}^{V}, \cdots) \rightarrow (W, d_{W}, l_{2}^{W}, \cdots)$ lifting $i$ and $p$ respectively (i.e. their linear components are given by $i$ and $p$).
	\end{prop}
	
	For a proof, see \cite{kraft2022introductionlinftyalgebrashomotopytheory}. We will explain the idea. \\
	
	The main idea is that we can define a $L_{\infty}$-algebra by considering the ``Koszul dual"
	picture: an $L_{\infty}$-algebra is defined to be a square $0$ degree $-1$ derivation $q$ on the completed graded symmetric algebra $\wh{S}(V[1]^{\vee})$, such that $q^{2} = 0$. Furthermore the linear component of the derivation $q$, given by $q^{(1)}$, is the dual of a differential on the underlying graded vector space $V^{\bullet}$. We can view the equation $q^{2} = 0$ as defining a small perturbation $q^{(\geq 2)}$ of $q^{(1)}$ so that $ q^{2} = (q^{(1)} + q^{(\geq 2)})^{2} = 0$. \\
	
	Therefore, it appears that one may able to apply the homological perturbation lemma to induce a corresponding deformation $q_{W} = q_{W}^{(1)} + q^{(\geq 2)}_{W}$ of $q_{W}^{(1)} = d_{W}^{\vee}$, which would define an $L_{\infty}$ structure on $W$, lifting $d_{W}$. However, in order to do this, we will need enhance the homotopy retract context 
	\[
	(W,d_{W}) \rightleftarrowstack{i  }{p } (V, d_{V})  \circlearrowleft{h} 
	\] 
	
	To a homotopy retract context 
	
	\begin{equation} \label{tensortrickperturb}
	(\wh{S}(W[1]^{\vee}), q_{W}^{(1)} ) \rightleftarrowstack{ S(p^{\vee}) }{ S(i^{\vee})  } (\wh{S}(V[1]^{\vee}), q^{(1)}_{V})    \circlearrowleft{S(h^{\vee})} 
	\end{equation}
	
	and this is precisely where the ``tensor trick" comes into play.

	Given $i$ and $p$, we can induce morphisms of completed cdgas $S(i^{\vee}): \wh{S}(W[1]^{\vee}) \rightarrow \wh{S}(V[1]^{\vee})$, $S(p^{\vee}) =  \wh{S}(V[1]^{\vee}) \rightarrow \wh{S}(W[1]^{\vee})$ by defining:
	\[
	\begin{array}{lc}
	S(i^{\vee}) = \bigoplus_{k = 0}^{\infty} (i^{\vee})^{\odot k} ,  & 
	S(p^{\vee}) = \bigoplus_{k = 0}^{\infty} (p^{\vee})^{\odot k}
	\end{array}
	\]
	
	Then, to define $S(h^{\vee})$, we need to rely on taking the symmetrization of $T(h^{\vee}) =  \bigoplus_{k=0}^{\infty} \sum\limits_{k=1}^{n}  (p^{\vee} \circ i^{\vee})^{\otimes k - 1} \otimes h^{\vee} \otimes  \id^{\otimes n - k}$. That is, we define 
	
	\[
	S(h^{\vee}) = \SS(T(h^{\vee}))
	\]
	where $\SS(T(h^{\vee}))$ the symmetrization of $T(h^{\vee})$ defined above. Then, one can verify that this gives the homotopy retract context \ref{tensortrickperturb}.\\
	
	Therefore, by applying the homological perturbation lemma \ref{HPL}, we obtain:
	
	\begin{enumerate}
	\item{A perturbation $q_{W} = q_{W}^{(1)} + q_{W}^{(2)} + \cdots$ defining an $L_{\infty}$ structure on $W$, lifting $d_{W}$. Explicitly, one has that
		\begin{equation}\label{perturbeddifferential}
			q_{W} = q^{(1)}_{W} + \sum\limits_{k=0}^{\infty}  S(i^{\vee})  \circ q_{V}^{(\geq 2)} \circ \Big( - S(h^{\vee}) \circ q_{V}^{(\geq 2)} \Big)^{k} \circ  S(p^{\vee})
		\end{equation}
		
	}
	\item{Perturbed morphisms of cdgas $S(i^{\vee})_{q^{\geq 2}}: \wh{S}(V[1]^{\vee}) \rightarrow \wh{S}(W[1]^{\vee})$, $S(p^{\vee})_{q^{\geq 2}}: \wh{S}(W[1]^{\vee}) \rightarrow \wh{S}(V[1]^{\vee}) $, compatible with $q_{W}$ $q_{V}$ respectively, so that they define $L_{\infty}$ morphisms $I: (W,d_{W}, l_{2}^{W}, \cdots) \rightarrow (V, d_{V}, l_{2}^{V}, \cdots)$ and $P: (V, d_{V}, l_{2}^{V}, \cdots) \rightarrow (W,d_{W}, l_{2}^{W}, \cdots)$. Explicitly, by the homological perturbation lemma, we have that 
		
		\begin{align} \label{perturbedmaps}
			& S(i^{\vee})_{q^{\geq 2}} =  S(i^{\vee} ) \circ  ( 1 + q_{V}^{(\geq 2)} \circ S(h^{\vee}) )^{-1}\\
			\nonumber &  S(p^{\vee})_{q^{\geq 2}} = (1 + S(h^{\vee} ) \circ  q_{V}^{(\geq 2)}  )^{-1} \circ S(p^{\vee})
		\end{align}	
	}
	\end{enumerate} 
	
	After dualizing \ref{perturbedmaps}, we see that the $L_{\infty}$ morphism  $I: (W,d_{W}, l_{2}^{W}, \cdots) \rightarrow (V, d_{V}, l_{2}^{V}, \cdots)$ and $P: (V, d_{V}, l_{2}^{V}, \cdots) \rightarrow (W,d_{W}, l_{2}^{W}, \cdots)$  is determined by the ``Taylor coefficients":
	
	\begin{itemize}
	\item{$I^{(1)} = \iota$}
	\item{$I^{(2)} = h(l_{2}^{V}( \iota, \iota )  ) $ }
	\item{In general, we have the recursive identity:
		
		\begin{equation} \label{perturbinclusion} 
			I^{(n)} = \sum\limits_{k=1}^{n-1}  \pm h( l_{2}^{V}( I^{(k) }  , I^{(n-k)}   )  )
		\end{equation}
	}
	\item{For example, $I^{(3)} = h( l_{2}^{V}( \iota,  h(l_{2}^{V}( \iota, \iota ))   )  )  \pm   h( l_{2}^{V}(  h(l_{2}^{V}( \iota, \iota )), \iota )  )$
		
	} 
	
	\end{itemize}
	
	\begin{rmk} \hfill
	\begin{enumerate}
		\item{In the case that $h$, $d$ and $\{ l_{k} \}_{k \geq 2}$ are bounded operators on $V$, then we can use the formulas \ref{perturbeddifferential}, \ref{perturbedmaps} to show that the resulting $L_{\infty}$ structure on $W$ defined via $q_{W}$, and that the $L_{\infty}$ morphisms $I$, $P$ are \emph{bounded} (in the sense of \ref{analytic}).
		} 
		\item{The above is usually formulated with the dual language of \emph{symmetric coalgebras}. In many ways, it is much cleaner to use coalgebras -- for example, one does not have to deal with completed symmetric algebras, and one does not have to dualize the formulae provided by the basic homological perturbation lemma. However, as we have not defined coalgebras, we will stick with the above notions.}
		
		\item{
			Although we can obtain $L_{\infty}$ morphisms $I$ and $P$ by applying the homological perturbation lemma in this way, the resulting perturbation $S(h^{\vee})_{q^{\geq 2}}$ of $h$ does \textbf{not} in general define a homotopy of $L_{\infty}$-algebras. To obtain a homotopy of $L_{\infty}$-algebras, we will need to cite the minimal model decomposition theorem for $L_{\infty}$-algebras \ref{strongminimalmodel}.}
		
	\end{enumerate}
	\end{rmk}
	
	\subsubsection*{Minimal model decompositions}

	\begin{defn} \label{splitting}
	Suppose that $L$ is an $L_{\infty}$-algebra.  Then, a \emph{splitting} is defined to be a degree $-1$ map $h: L^{\bullet} \rightarrow L^{\bullet}[-1]$ such that:
	
	\begin{enumerate}
		\item{$h^{2} = 0$}
		\item{$h d h = h$}
		\item{$d h d = d$} 
	\end{enumerate} 
	\end{defn}
	
	From the definition of a splitting we see that we obtain idempotent operators $dh$ and $hd$, as $(dh)^{2} = dhdh = dh$ and $(hd)^{2} = hdhd = hd$. We also have the ``Laplacian" defined by $\Delta = [d,h] = dh + hd$ which is also idempotent. From this, we obtain a splitting for each $k$:
	
	\[
	L^{k} = \im( d )^{k} \oplus H^{k} \oplus \im(h)^{k} = \im( \Delta ) \oplus \ker( \Delta )
	\]
	
	where $H^{k} = \ker( [d,h])$. 
	
	\begin{lem}
	One has that $H^{k} = H^{k}(L^{\bullet})$. Then, we have a homotopy retract context
	\[
	(H^{\bullet}, 0) \rightleftarrowstack{ i_{H}  }{ \id - \Delta } (L, d)  \circlearrowleft{h} 
	\] 
	\end{lem}

	\begin{cor} \label{basicminimalmodel1}
	There exists an $L_{\infty}$ structure $\big( H^{\bullet}, 0, l_{2}^{H}, \cdots, \big)$ on $H^{\bullet}$, and an $L_{\infty}$ retract context

	\[
	(H^{\bullet}, l_{\bullet \geq 2}^{H} ) \rightleftarrowstack{ I_{H}  }{ P_{H}} (L, l_{\bullet \geq 1} )  \circlearrowleft{ \mbf{\eta}_{H} } 
	\] 
	In particular, $I_{H}: (H^{\bullet}, l_{\bullet \geq 2}^{H} ) \hookrightarrow (L, l_{\bullet \geq 1} ) $ is a minimal model of $(L, l_{\bullet \geq 1} ) $.
	
	\end{cor}

	The next lemmas concern lifting properties of $L_{\infty}$-morphisms, and are key for establishing a ``strong form" of the $L_{\infty}$ minimal model theorem \ref{strongminimalmodel}.

	\begin{lem} \label{Linftylift}
	Suppose that $(N, d, 0 , 0 , \cdots )$ is a linear contractible $L_{\infty}$-algebra, and that $\phi: (V, d_{V}, \cdots ) \rightarrow (W, d_{W}, \cdots)$ is a morphisms of $L_{\infty}$-algebras. Then, for any morphism $j: (N,d) \rightarrow (W, d_{W})$ of chain complexes, there exists an $L_{\infty}$-morphism
	
	\[
	\phi' : (V, d_{V}, \cdots) \times (N, d, 0, \cdots ) \rightarrow (W, d_{W}, \cdots)
	\]
	
	such that $\phi' \resto_{V} = \phi$ and $\phi'(n) = j(n)$ for all $n \in N$.
	\end{lem}

	The next proposition may remind the reader of a ``homotopy minimal chart" \ref{minimalatlas} from chapter $1$ -- and indeed, applying the strong form of the minimal model theorem, in the situation where a quasi-smooth space is a local $L_{\infty}[1]$-bundle exactly produces a homotopy minimal chart. 
	For details on this version of the $L_{\infty}$ minimal model theorem, see \cite{Manetti2022}.
	
	\begin{prop} \emph{(Strong form of the $L_{\infty}$ minimal model theorem)} \label{strongminimalmodel}\\
	Any $L_{\infty}$-algebra is the product of a minimal model and a linear contractible $L_{\infty}$-algebra (i.e. a chain complex). More precisely, for each choice of minimal model  $I_{H}: (H^{\bullet}, l_{\bullet \geq 2}^{H} ) \hookrightarrow (L, l_{\bullet \geq 1} )$, there exists a linear contractible $L_{\infty}$-algebra $(N, d, 0, 0 ,\cdots )$ and an $L_{\infty}$ isomorphism
	
	\[
	\phi:  (H^{\bullet}, l_{\bullet \geq 2}^{H} ) \times (N, d, 0, 0 ,\cdots ) \longiso (L, l_{\bullet \geq 1} )
	\]
	\end{prop}
	
	From this, one can prove the following ``$L_{\infty}$ enrichment" of homotopy retract contexts:

	\begin{prop} \label{transferLinfty} (The ``strong" $L_{\infty}$ transfer theorem)\\
	Suppose that $(V,W,i,p,h)$ is a homotopy retract context. Then, if there is an $L_{\infty}$-algebra structure $(V,d_{V}, l_{2}, \cdots )$  (i.e. an $L_{\infty}$ structure ``lifting" $d_{V}$), then there exist:
	
	\begin{enumerate}
		\item{An $L_{\infty}$ structure $(W, d_{W}, l_{2}^{W}, \cdots )$ for $W$, lifting $d_{W}$}
		\item{An $L_{\infty}$ quasi-isomorphism $I_{W}: W \rightarrow V$, with $i$ being the linear term of $I_{W}$.}
		\item{An $L_{\infty}$ quasi-isomorphism $P_{W}: V \rightarrow W$ with $p$ being the linear term of $P_{W}$.} 
		\item{An $L_{\infty}$ homotopy $\bs{\eta}: (V, d_{V}, l_{2}, \cdots) \rightarrow (V, d_{V}, l_{2}, \cdots) \otimes \Omega_{\Delta^{1}}$, such that $\bs{\eta} \resto_{t = 0} = \Id$, $\bs{\eta} \resto_{t = 1} = I_{W} P_{W}$  }
	\end{enumerate}
	\end{prop}

	\begin{cor}
	Suppose that $\bL(U)$ is an $\cL_{\infty}$-algebra, then an $L_{\infty}$ minimal model decomposition $L' = H^{\bullet} \times N^{\bullet} \longiso L$ defines a local slice chart $\Phi: \bL'(U') \rightarrow \bL(U)$ near the origin. 
	\end{cor}
	
	If $\bL(U)$ is an $\cL_{\infty}$-algebra, then we say that $\Phi$ is a homotopy minimal atlas for $\bL(U)$ if for each $\mu \in \tMC(L,U)$ there is a choice of minimal model decomposition $L_{\mu}' = H^{\bullet}_{\mu} \times N_{\mu} \longiso L_{\mu}$ 
	so that $\Phi_{\mu}: \bL'_{\mu}(U) \iso \bL(U)$ defines a local slice chart for $\bL(U)$ centered at $\mu$.

	\section{Homotopy minimal atlases for $\cL_{\infty}$-algebras}

	\noindent Now, we consider the above in families. We say that $\eta$ defines a family of gauges for an $\cL_{\infty}$-algebra $\bL(U)$, if for each $\mu \in \tMC(L,U)$, $\eta_{\mu}$ is a gauge (or a contraction), for the $L_{\infty}$-algebra $L_{\mu}$. 
	
	\begin{defn} \label{gauge}
	Suppose that $\bL(U)$ is an $\cL_{\infty}$-algebra. A family of gauges (or simply just a gauge), is the data of a degree $-1$ bundle morphism $ \eta: \bL^{\bullet} \rightarrow \bL^{\bullet}[-1]$  such that: $\eta^{2} = 0$, $\eta \delta  \eta = \eta$ and $\delta  \eta \delta  = \delta $.   
	\end{defn}

	Similar to a splitting, gauge $\eta$ defines a ``Laplacian", given by:
	
	\begin{equation}\label{laplacian}
	\Delta_{U} = \eta \delta + \delta \eta = [\eta, \delta]
	\end{equation}
	
	So that $\Delta_{U} : \bL^{\bullet}_{U} \rightarrow \bL^{\bullet}_{U}$ is a bundle map.
	
		\begin{defn}
		The harmonic distribution is defined as 
		
		\[
		\mbf{H}_{U} = \ker( \Delta_{U}  ) \hookrightarrow T_{U}
		\] 
	\end{defn}

	While definition \ref{gauge} looks a lot like definition \ref{splitting} of a splitting, the main difference is that  $\eta_{\mu}$ does not in general define a deformation retract context onto the cohomology of $L_{\mu}$. The reason is because we may not have $\id = d_{\mu} \eta_{\mu} +  \eta_{\mu} d_{\mu}$, although $[d_{\mu} ,\eta_{\mu}]$ will be an automorphism of $\im ( [ d_{\mu}, \eta_{\mu}])$. This is why one needs a ``Green's operator":
	
	\begin{defn} (Green's operator)\\
	For each $\mu$, we define the corresponding splitting homotopy for $L_{\mu}$ to be $\bs{\eta}_{\mu} =  \eta_{\mu} G_{\mu}$, where $G_{\mu}$ is the inverse to $[d_{\mu}, \eta_{\mu}]$ on $\im([d_{\mu},\eta_{\mu}])$ and $0$ on $\ker([d_{\mu}, \eta_{\mu}])$. We say that $G_{\mu}$ is the Green's operator associated to the gauge $\eta_{\mu}$.
	\end{defn}

	Therefore, for each $\mu \in X$, the fiber of the gauge and harmonic distribution $\ker( \Delta_{U})$  at $\mu$ defines a minimal model homotopy retract context:
	
	\[
	H_{\mu}^{\bullet} \rightleftarrowstack{ I_{\mu}  }{ P_{\mu} } L_{\mu} \circlearrowleft{ \bs{\eta}_{\mu} } 
	\]

	Suppose that $X^{\red} = \bigcup  X_{(d)}$ denotes the stratification of $X^{\red}$ by embedding dimension. Then, on each smooth stratum $X_{(d)}$, $\mbf{H}_{(d)} := \mbf{H} \resto_{X_{(d)}}$ along with $\im(\Delta_{U}) \resto_{X_{(d)}}$ define integrable foliations. For each $\mu \in X_{(d)}$, a trivialized foliation chart $U_{\mu}$ defines a homotopy minimal chart for $\mu$
	
	\[
	\bL_{\mu}'(U_{\mu}') = \mbf{H}_{\mu}(U_{\mu/\g}^{\min} ) \times \wh{N}_{\mu} 
	\]
	
	such that for $\mu$ and $\nu$ in $X_{(d)}$, there is a commutative diagram

	\[
	\begin{tikzcd}
	\mbf{H}_{\mu}(U_{\mu/\g}^{\min} ) \times \wh{N}_{\mu}  \arrow[r, "\Phi_{\mu,\nu}^{(d)}"]  \arrow[d, "P_{\mu}" swap , shift right  = 1em ] & 
	\mbf{H}_{\nu}(U_{\nu/\g}^{\min} ) \times \wh{N}_{\nu}  \arrow[d, "P_{\nu}" , shift left = 1em ] \\
	\mbf{H}_{\mu}(U_{\mu/\g}^{\min} )  \arrow[u, "I_{\mu}" swap] \arrow[r, "\Phi^{\min}_{\mu \nu}" ]  &  \mbf{H}_{\nu}(U_{\nu/\g}^{\min} )   \arrow[u, "I_{\nu}"] 
	\end{tikzcd}
	\]
	
	\begin{defn}
	We define the \emph{stratified} homotopy minimal atlas $\Phi$ associated to a gauge $\eta$ to be the union of foliation atlases $\Phi = \bigcup_{d} \Phi^{(d)}$ 
	\end{defn}

	\subsection{$\cL_{\infty}^{\Omega}$-algebras and homotopy gauge actions}

	In this section, we define the notion of an $\cL_{\infty}^{\Omega}$-algebra. Briefly speaking, we can think of an $\cL_{\infty}^{\Omega}$-algebra as an $\cL_{\infty}$-algebra $\bL(U)$ with a \emph{homotopy} gauge group action by an analytic group $\gG_{\Omega}$. While this notion may be much more general than what will be defined, we restrict ourselves to a simpler notion of a homotopy gauge group action as it will be sufficient for the setting of Chern-Simons theory. In other words, we will define a homotopy gauge group action by considering a fixed dgla model, i.e. a strictification of $\bL(U)$, with a \emph{strict} gauge group action via $\gG_{\Omega}$. Throughout the paper, whenever it is convenient, we may use the notation $\gG$ in place of $\gG_{\Omega}$.

	\begin{defn} 
	An $\cL^{\Omega}_{\infty}$-algebra is an $\cL_{\infty}$-algebra $(L,U)$, along with the data of a dgla $L_{\Omega}$ and a fixed local analytic $L_{\infty}$-homotopy retract context:
	
	\begin{equation}\label{omegamodel}
		(L,U) \substack{ \xrightarrow{\mbf{I}} \\ \xleftarrow{\mbf{P}}} (L_{\Omega}, U_{\Omega}  ) \circlearrowleft_{\bs{\eta}_{\Omega} }
	\end{equation}
	for some domain $U_{\Omega}  \subset L_{\Omega}^{1}$. We will use the notation $(L \mdl L_{\Omega}, U)$ to denote an $\cL_{\infty}$-algebra $(L,U)$ together with a dgla model $(L_{\Omega}, U_{\Omega} )$ and a retract context \ref{omegamodel}. 
	\end{defn}
	
	\begin{rmk}
	It is a fact that every $L_{\infty}$-algebra can be functorially strictified into a dgla (for example, see \cite{krizMay}). So, a $\cL^{\Omega}_{\infty}$-algebra fixes a choice of a strictification $(L_{\Omega},U_{\Omega} )$ of an $\cL_{\infty}$-algebra $(L,U)$, and an \emph{analytic} homotopy retract context between $(L,U)$ and $(L_{\Omega}, U_{\Omega} )$.
	\end{rmk}

	We will see in the application to Chern-Simons theory, one can think of $\ref{omegamodel}$ as defining a way to transfer ``large gauge transformations" of a dgla model $L_{\Omega}$ to $L$. Large gauge transformations, in particular, are automorphisms of $L_{\Omega}$ that are not contained in the subgroup $\exp(L^{0}_{\Omega})$. This is a way in which the above data may be viewed as endowing an $\cL_{\infty}$-algebra with the data of a \emph{homotopy gauge group action}.\\
	
	Suppose that a group $\gG = \gG_{\Omega}$ acts on $L_{\Omega}$ by dgla automorphisms. Furthermore, suppose that $\mu \in \tMC(L,U)$ is such that $g \star \mu := \mbf{P} \big( g \cdot \mbf{I}(\mu) \big) \in U$. Then, we define the $L_{\infty}$ quasi-isomorphism:

	\begin{equation} \label{homotopygauge} 
	\Phi_{(g,\mu)}(-) :=   \mbf{P}^{g \cdot \mu }(  g \cdot \mbf{I}^{\mu}  ( -  ))  
	\end{equation}
	For sufficiently small open neighbourhoods $U_{\mu}$ and $U_{g \star \mu}$ of $\mu$ and $g \star \mu$ respectively, 

	 we may consider $\Phi_{(g,\mu)}$ as a morphism of $\cL_{\infty}$-algebras $\Phi_{(g,\mu)}: (L^{\mu}, U_{\mu}) \rightarrow (L^{ g \star \mu} , U_{g \star \mu})$.

	\begin{lem} \label{homotopyaction}
	For $g, g' \in \gG$, if $\Phi_{(g,\mu)}$ and $\Phi_{(g', g \star \mu)}$ are composable, then we have a homotopy between morphisms of $\cL_{\infty}$-algebras 
	
	\[
	\Phi_{(g', g \star \mu) } \circ \Phi_{(g,\mu)}  \sim \Phi_{ (g' \cdot g, \mu) } 
	\] 	
	\end{lem}

	\begin{defn}
	We define an $\cL_{\infty}^{\Omega}$-algebra to be the data of a dgla-modelled local $L_{\infty}$ algebra $L \mdl L_{\Omega}$, and a discrete extension $\gG$ of the infinitesimal gauge group $\gG_{0} = \exp(L_{\Omega}^{0})$. We write $\pi_{0}\gG = \gG/\gG_{0}$.
	\end{defn}
	
	If $L \rightsquigarrow L_{\Omega}$ is a dgla-modelled $\cL_{\infty}$-algebra, then one has that 
	$[\tMC(L,U)/\g] \cong [\tMC(L_{\Omega}, U_{\Omega} )/ \gG_{0}]$. In particular, if $[\tMC(L_{\Omega},U_{\Omega} )/ \gG_{0}]$ is representable by an analytic space, then so is $[\tMC(L,U)/\g]$. Furthermore, we also obtain a (strict) $\pi_{0} \gG$ action on $[\tMC(L,U)/\g]$, and one has that 
	
	\[
	\big[ [\tMC(L,U)/\g] / \pi_{0} \gG \big] \cong \big[ [\tMC(L_{\Omega}, U_{\Omega} )/ \gG_{0}]/\pi_{0} \gG \big]  \cong [\tMC(L_{\Omega}, U_{\Omega} )/\gG] =  [ \tMC(L_{\Omega}) /\gG ]_{ \resto_{U_{\Omega} } }
	\]

	\begin{defn} \label{extendedgaugequotient} 
	If $(L \mdl L_{\Omega}, U)$ is an $\cL^{\Omega}_{\infty}$-algebra, then we define the extended gauge group quotient to be 
	
	\[
	\tMC_{\gG}(L,U) = \big[ [\tMC(L,U)/\g] / \pi_{0} \gG \big]  \cong  [ \tMC(L_{\Omega}, U_{\Omega}  ) /\gG  ]
	\]
	\paragraph{{\bfseries Convention:}}
	
	For simplicity, we will always demand that our extended gauge group quotients are representable by analytic spaces. We may also write $\tMC_{\gG}(L,U) = \tMC(L,U) \git \gG $. \\
	We say that $\mu \sim_{\gG} \mu$ is an \emph{extended} gauge equivalence or a $\gG$-equivalence, if there exists a $g \in \gG$ such that $g \star \mu  = \mu'$. Note that if $\gG = \gG_{0}$, then $\tMC_{\gG}(L,U) = \tMC_{\g}(L,U)$. In this case, the $\gG$-equivalence relation agrees with the (infinitesimal) gauge equivalence relation. 
	\end{defn}

	Note that the definition of $\Phi_{(g,\mu)}$ (\ref{homotopygauge}) only defines a morphism between $\cL_{\infty}$-algebras if we choose appropriate convergence domains for $L_{\mu}$ and $L_{g \star \mu}$ respectively. Given a homotopy minimal atlas $\Phi$ for an $\cL_{\infty}$-algebra $(L,U)$, we can make choices of convergence domains subordinate to such an atlas.  That is, given a homotopy minimal atlas $\Phi$ for the $\cL_{\infty}$-algebra $(L,U)$, we can always shrink $U_{\mu}$ if necessary so that $\Phi_{(g,\mu)}: (L_{\mu}, U_{\mu})  \rightarrow  (L_{g \star \mu}, U_{g \star \mu} )$ is a well-defined morphism of $\cL_{\infty}$-algebras.

	\subsection{Gauge-fixed $\cL_{\infty}^{\Omega}$-algebras}
	\begin{notation} \label{ordering} Suppose that $(L \mdl L_{\Omega}, U, \Phi )$ is an $\cL_{\infty}^{\Omega}$-algebra equipped with a homotopy minimal atlas $\Phi$. 
	
	\begin{enumerate}
		\item{For $\mu, \ups \in \tMC(L,U)$, we write $\ups \preceq_{\Phi} \mu$ if there exists a $\ups_{\mu} \in \tMC(L,U)$ such that $\ups_{\mu}$ is contained in the image of $\tau^{0}I_{\mu}: \tMC(H_{\mu}, U_{\mu/\g}^{\min}) \hookrightarrow \tMC(L,U)$ and there exist a $\gG$-equivalence $\ups_{\mu} \sim_{\gG} \ups$. }
		\item{If in the above, $\ups \preceq_{\Phi} \mu$ are such that $g \cdot \ups = \ups_{\mu}$ with $g \in \gG_{0}$, then we write $\ups \sim_{\gG_{0}} \mu$ and $\ups \preceq_{\Phi_{0}}  \mu$. }
		\item{We write $\gG|_{U} := \{ g \in \gG \mid g \star \mu \in U \text{, for some } \mu \in \tMC(L,U) \}$. Furthermore, we define the notation: $\pi_{0}(\gG|_{U} ) = \gG|_{U} \Big/ \gG_{0} |_{U}$ }
		
		\item{We write $\gG^{\tMC}_{U} = \{ (\mu', g, \mu) \in \tMC(L,U) \times \gG \times \tMC(L,U) \mid g \star \mu = \mu' \}$, which can be identified with the restriction action groupoid to $\tMC(L_{\Omega},V)$.}
	\end{enumerate}
	
	\end{notation}

	If $\ups \preceq_{\Phi} \mu$, then there exists a $g \in \gG$ such that $g \star \ups = \ups_{\mu}$. Then, we obtain an $L_{\infty}$ isomorphism 
	
	\[
	\Phi_{(g,\mu)}^{\min} : \bH_{\ups}( U^{\min}_{\ups/\g}) \resto_{U_{\ups,\ups_{\mu}/\g}^{\min}  }  \longiso \bH_{\ups_{\mu}}( U_{\ups_{\mu}/\g}^{\min}  ) \resto_{U^{\min}_{\ups_{\mu},\ups/\g} } 
	\]
	For $U_{\ups,\ups_{\mu}/\g}^{\min} \subset U_{\ups/\g}^{\min}$ and $U_{\ups_{\mu},\ups/\g}^{\min} \subset U_{\ups_{\mu}/\g}^{\min}$ sufficiently small polydiscs around the origin. \\
	
	Furthermore,  where composable, the homotopy relation $\Phi_{(g', g \star \mu)} \circ \Phi_{(g,\mu)} \sim \Phi_{(g' \cdot g, \mu) }$ strictifies to equality upon passing to minimal models:
	
	\[
	\Phi_{(g', g \star \mu) }^{\min} \circ \Phi_{(g , \mu) }^{\min}  = \Phi_{(g' \cdot g, \mu) }^{\min}
	\]

	If there is no confusion (for example, if there is a unique $g \in \gG$ such that $g  \star \ups = \ups_{\mu}$), then we may write $\Phi_{\ups,\ups_{\mu}} = \Phi_{(g,\mu)}$ and $\Phi_{\ups,\ups_{\mu}}^{\min} = \Phi_{(g, \mu)}^{\min}$.

	Suppose that $\ups, \mu \in \tMC(L,U)$, such that $\ups \preceq_{\Phi} \mu$, then by definition we have an $L_{\infty}$ homotopy equivalence $\Phi_{(g, \ups)}: L_{\ups} \rightarrow L_{\ups_{\mu}}$, with $g \star \mu = \ups_{\mu} \in \tMC(L_{\mu}, U_{\mu})$, for some $g \in \gG$. Furthermore, one has the following commutative diagram:
	
	\begin{equation} \label{gaugefixing}
	\begin{tikzcd} [sep = huge]
		(L_{\ups}, U_{\ups,\ups_{\mu}}) \arrow[r, "\Phi_{(g,\ups)}"] & (L_{\ups_{\mu}}, U_{\ups_{\mu},\ups})  \arrow[d, "P_{\ups_{\mu}}"]   \\
		(H_{\ups},U_{\ups,\ups_{\mu}/\g}^{\min}) \arrow[u, "I_{\ups}", hookrightarrow] \arrow[r, "\Phi_{(g, \ups)}^{\min}" , "\sim" swap] &  (H_{\ups_{\mu}}, U_{\ups_{\mu},\ups/\g}^{\min}) 
	\end{tikzcd}
	\end{equation}

	Therefore, we may also define the following injective quasiisomorphisms:
	
	\begin{enumerate}
	\item{$\Phi_{\ups_{\mu}} : H_{\ups} \hookrightarrow L_{\ups_{\mu}}$, by $\Phi_{\ups_{\mu}} := I_{\ups_{\mu}} \circ \Phi_{(g, \ups)}^{\min}$. }
	\item{$I_{\ups/\mu}:  H_{\ups} \hookrightarrow H_{\mu}^{\ups_{\mu}}$, by $I_{\ups/\mu} := I_{\ups_{\mu}/\mu} \circ \Phi_{(g, \ups)}^{\min}$. }
	\end{enumerate}

	\begin{defn}
	For a $\cL_{\infty}^{\Omega}$-algebra $(L \mdl L_{\Omega}, U)$ with a homotopy minimal atlas $\Phi$, we write $\Phi = \Phi_{0}$ and define $\Phi = \Phi_{0} \cup \{ \Phi_{(g,\mu)}  \mid (g , \mu) \in \gG \}$ (such that $\Phi_{(g,\mu)}$ induces a commutative diagram \ref{gaugefixing}), and say that $(L \mdl L_{\Omega}, U, \Phi)$ defines a \emph{gauge-fixed} $\cL_{\infty}^{\Omega}$-algebra. We say that $\Phi$ is a $\gG$-adapted homotopy minimal atlas. \\

	\end{defn}

	The $\gG$-adapted homotopy minimal atlas $\Phi$ has the following compatibility property. Suppose that $\ups, \ups' \in X_{(e)}$ and $\mu, \nu \in X_{(d)}$ with $e < d$ and $\ups, \ups' \preceq_{\Phi} \mu, \nu$, then the following diagram commutes:
	\begin{equation} \label{compatibility}
	\begin{tikzcd} [sep = huge ]
		\bH_{\ups_{\mu}}(U_{\ups_{\mu}/\g}^{\min}) \arrow[r, "I_{\ups_{\mu}}"]  & \bL_{\ups_{\mu}}(U_{\ups_{\mu} })    \arrow[r, "\Phi_{\ups_{\mu},\ups'_{\nu}} \, \, = \, \,  \Phi_{\mu,\nu}^{\ups, \ups'}","\sim" {anchor=north}] &  \bL_{\ups'_{\nu}}(U_{\ups'_{\nu}}) \arrow[d,"\Phi_{g'{}^{-1}}" swap] \arrow[ddr, "\wtil{\Phi}_{\ups'_{\nu}}"]  \arrow[r, "P_{\ups'_{\nu} }" ]    &  \bH_{\ups'_{\nu} }( U_{\ups'_{\nu} /\g}^{\min}) \arrow[dd, "\Phi_{g'{}^{-1}}^{\min}",, "\sim" {anchor=south, rotate = 90}] \\
		& \bL_{\ups}(U_{\ups} )  \arrow[u,"\Phi_{g}" swap] \arrow[r, " \Phi_{\ups,\ups'}", "\sim" {anchor=north}]  & \bL_{\ups'}(U_{\ups'})   \arrow[dr, "P_{\ups'}" swap]  &  \\ 
		\bH_{\ups}(U_{\ups/\g}^{\min} ) \arrow[uu, "\Phi_{g}^{\min}", "\sim" {anchor=north, rotate=90}] \arrow[ur, "I_{\ups}" swap, hookrightarrow] \arrow[uur, "\Phi_{\ups_{\mu}}", hookrightarrow]  \arrow[rrr, "\Phi^{\min}_{\ups,\ups'}", "\sim" {anchor=north}]   & & &	\bH_{\ups'}(U_{\ups'/\g}^{\min} )  
	\end{tikzcd}
	\end{equation} 
	
	That is, we have that 
	
	\begin{equation}
	\Phi_{\ups,\ups'}^{\min} = \Phi                                                                                                                                                                                                                                                                                                                                                                                                                                                                                                                                                                                                                                                                                                                                                                                                                                                                                                                                                                                                                                                                                                                                                                                                                                                                                                                                                                                                                                                                                                                                                                                                                                                                                                                                                                                                                                                                                                                                                                                                                                                                                                                                                                                                                                                                                                                                                                                                                                                                                                                                                                                                                                                                                                                                                                                   _{\ups'_{\nu}} \circ \Phi_{\ups_{\mu}, \ups_{\nu}'} \circ \Phi_{\ups_{\mu}}   : \bH_{\ups}(U_{\ups/\g}^{\min}) \longiso \bH_{\ups'}(U_{\ups'/\g}^{\min})   
	\end{equation}
	
	At the very least, it is easy to see that $\Phi_{\ups,\ups'}^{\min}$ and $\Phi                                                                                                                                                                                                                                                                                                                                                                                                                                                                                                                                                                                                                                                                                                                                                                                                                                                                                                                                                                                                                                                                                                                                                                                                                                                                                                                                                                                                                                                                                                                                                                                                                                                                                                                                                                                                                                                                                                                                                                                                                                                                                                                                                                                                                                                                                                                                                                                                                                                                                                                                                                                                                                                                                                                                                                                  _{\ups'_{\nu}} \circ \Phi_{\ups_{\mu}, \ups_{\nu}'} \circ \Phi_{\ups_{\mu}}$ are $L_{\infty}$-homotopic morphisms. On the other hand, if $F, G : H \rightarrow H'$ are homotopic morphisms between \emph{minimal} $L_{\infty}$-algebras $H$ and $H'$, with $H^{0} = 0 = (H')^{0}$ then they must be equal.

	\section{Homotopy BV geometry} 
	
	In this section, we will introduce the notion of a $(-1)$-shifted symplectic structure for an $\cL_{\infty}$-algebra. We will see in the forthcoming sequel to this paper, $(-1)$-shifted symplectic structures and homotopy BV data are among the many things one can obtain for an $\cL_{\infty}^{\Omega}$-algebra from its dgla model.
	In this section, we will investigate some of the geometry that arises out of a $(-1)$-shifted symplectic structure on an $\cL_{\infty}$-algebra. In \cite{joyce2013classicalmodelderivedcritical}, Joyce shows that derived spaces with $(-1)$-shifted symplectic structures have classical loci which are locally presentable as critical loci.

	\subsection*{Kahler differentials on $\cL_{\infty}$-algebras}
	
	Suppose that $\bL(U)$ is an $\cL_{\infty}$-algebra, concentrated in non-negative degrees. In particular, there is a cdga of functions on $\bL(U)$ given by:

	\[
	\O^{\bullet}_{(L,U)} := (\wh{\Sym}_{\O_{U}}( \bL^{\bullet }[-1]^{\vee} ), q)
	\]
	
	Thus, we can consider the complex of Kahler differentials on $\O^{\bullet}_{(L,U)}$, given by:
	
	\[
	\Omega^{1}_{(L,U)} := \O^{\bullet}_{(L,U)} \wh{\otimes} (L^{\bullet})^{\vee}[1]  
	\]
	
	if $f \otimes x[1] \in \Omega^{1}_{(L,U)}$, we will suggestively write $f dx  := f \otimes x[1]$. Then, as usual, there is a canonical map (the exterior derivative)
	
	\[
	d_{dR}: \O^{\bullet}_{(L,U)}\rightarrow  \Omega^{1}_{(L,U)}
	\]
	defined by considering the shift map (of chain complexes) $\bL^{\vee} \rightarrow \bL^{\vee}[1]$, and then continuously extending to $\O^{\bullet}_{(L,U)}$ uniquely to act as a derivation.  By definition, we have that $d_{dR}$ and $q$ commute, so that $[d_{dR},q] = 0$. \\
	
	Once we have defined Kahler $1$-forms for $\bL(U)$, we can then define $p$-forms. They are defined as:
	\[
	\Omega^{p}_{(L,U)} := \bigwedge_{\O^{\bullet}_{(L,U)}}^{p} \Omega^{1}_{(L,U)}
	\]
	Note that $\Omega^{0}_{(L,U)} = \O^{\bullet}_{(L,U)}$. 
	
	\begin{defn}
		The deRham complex of $\bL(U)$ is defined to be the complex defined by:\\ 
		\[
		\Omega^{\bullet}_{(L,U)} := ( \bigoplus\limits_{p \geq 0} \Omega^{p}_{(L,U) } , d_{dR} + q )
		\]
	\end{defn}

	Let us now consider $2$-forms on $\bL(U)$. By definition, a $2$-form is given by an element $ \omega \in \bigwedge^{2} \Omega^{1}_{\bL(U)} = \O^{\bullet}_{\bL(U)} \otimes \big( d_{dR} L^{\vee} \bigwedge  d_{dR} L^{\vee} \big)$.
	Therefore, for any $x \in U$ we obtain a pairing $\omega_{x}: T_{x} \bL(U) \wedge T_{x} \bL(U) \rightarrow \CC$, and that these pairings are analytic with respect to $x \in U$.  Furthermore, note that if $\omega$ is $q$-closed, then we get that 
	\[
	\omega_{x} ( d_{x}u , v ) =  (-1)^{|u||x|}\omega_{x} ( u, d_{x}v) 
	\]
	for all $u,v \in T_{x} \bL(U)$. Therefore, $\omega$ defines a morphism of chain complexes  $T_{\mu} \bL(U) \rightarrow T_{\mu} \bL(U)^{\vee}$ for all $\mu \in \tMC( L,U )$.
	\\

	\begin{defn}
		We say that an $\cL_{\infty}$-algebra $\bL(U)$ is $(-1)$-shifted symplectic, if there exists a $d_{dR} +q$ closed element $\omega \in \bigwedge^{2} \Omega^{1}_{\bL(U)}$ of homological degree $-1$ such that 
		for each $\mu \in \tMC(\bL(U) )$ the induced morphism of chain complexes $\omega_{\mu}: L_{\mu} \rightarrow L_{\mu}^{\vee}[-1]$ is a quasi-isomorphism. Recall that $L_{\mu} = T_{\mu} \bL(U)$.\\
	\end{defn}
	
	\subsection{BV $\cL_{\infty}$-algebras}
	
	Suppose that a $(-1)$-shifted symplectic form $\omega$ for an $\cL_{\infty}$-algebra $\bL(U)$ is exact in the deRham complex $\Omega^{\bullet}_{(L,U)}$. Then, this would mean that there is a total degree $0$ element $\theta \in  \Omega^{\bullet}_{(L,U)}$ such that $(d_{dR} + q)\theta =\omega$. As $\O_{(L,U)}^{\bullet}$ is concentrated in $\leq 0$ degrees, such a $\theta$ must be given by a pair $(S, \sigma)$ where $S \in \O^{0}_{(L,U)}$ is a degree $0$ function and $\sigma \in (\Omega^{1}_{(L,U)})^{-1}$ is a degree $-1$ $1$-form. The equation $(d_{dR} + q) \theta =\omega $ can then be written as:
	\begin{align} \label{BVeqns}
		& d_{dR} \sigma = \omega \\
		& \nonumber q \sigma = d_{dR} S  \\
		& \nonumber qS = 0 
	\end{align}
	In general, the equation $q\sigma = d_{dR} S$ implies that $d_{dR}  S$ vanishes on the classical locus of $\bL(U)$. 
	Note that if $L$ is in strictly positive degrees, then $qS = 0$ is trivially true for degree reasons.

	\subsubsection*{Bundle maps and $1$-forms }
	
	Recall that we assume that our $\cL_{\infty}$-algebras are quasi-smooth. \\

	One can identify a morphism of bundles 
	
	\[
	\sigma: \bL^{2} \rightarrow (\bL^{1})^{\vee} = \Omega^{1}_{U}
	\]
	
	with a degree $-1$ $1$-form $\sigma \in \Omega^{1}_{U} \otimes (L)^{\vee}  \hookrightarrow \Omega^{1,-1}_{(L,U)}$.  Conversely, any degree $-1$ $1$-form contained in $\Omega^{1}_{U} \otimes (L^{2})^{\vee}$ can be written as a  morphism of bundles $\bL^{2}_{U} \rightarrow \Omega^{1}_{U}$.

	\begin{rmk}
		
	Suppose that $\bL(U)$ is a strictly quasismooth $\cL_{\infty}$-algebra -- i.e., where $L$ is in degrees $[1,2]$.
	  
	Then, we have splittings of the deRham complex of $\bL(U)$, and the space of $2$-forms of degree $-1$ is given by

	\[
	\Omega_{(L,U)}^{2, -1}  = \Big( d_{dR}(L^{2})^{\vee} \otimes \Omega_{U}  \Big)[1] \, \,  \oplus \, \, \Big( (L^{2})^{\vee} \otimes \Omega^{2}_{U} \Big)[1] 
	\]
	
	and the space of $1$-forms of degree $-1$ is given by 
	
	\[
	\Omega_{(L,U)}^{1, -1}  = \Big( d_{dR} (L^{2})^{\vee}  \Big)[1]  \, \, \oplus \, \, \Big( (L^{2})^{\vee}  \otimes \Omega^{1}_{U} \Big)[1]
	\]

	Furthermore, the deRham differential splits into components, so that so that $d_{dR} = \sD + \iota_{d} + D$ in the below diagram:
	
	\[
	\begin{tikzcd} [row sep= 8em]
		d(L^{2})^{\vee} \otimes \Omega_{U}^{1}    & \oplus & (L^{2})^{\vee} \otimes \Omega^{2}_{U} \\ 
		d_{dR}(L^{2})^{\vee}	\arrow[u, "\mbf{s}D"]			&    \oplus 		&   (L^{2})^{\vee}  \otimes \Omega^{1}_{U} \arrow[u, "D"] \arrow[ull, "\iota_{d}" swap] 
	\end{tikzcd}
	\]
	The diagonal map $\iota_{d}: (L^{2})^{\vee} \otimes \Omega^{1}_{U} \rightarrow d_{dR}(L^{2})^{\vee} \otimes \Omega_{U}^{1}$ is defined by $x \otimes \alpha \mapsto x[1] \otimes \alpha = d_{dR} x \otimes \alpha$. Clearly $\iota_{d}$ is an isomorphism, acts as a degree shift for functions in degree $-1$. In particular, for any $\omega \in d_{dR}(L^{2})^{\vee} \otimes \Omega_{U}^{1}$, there exists a $\wtil{\sigma} \in (L^{2})^{\vee} \otimes \Omega_{U}^{1}$ such that $\iota_{d}( \wtil{\sigma} ) = \omega$. Note that we do not necessarily have that $d_{dR} \wtil{\sigma} = \omega$, unless $D \wtil{\sigma} = 0$. \\
	\end{rmk}

	In light of \ref{BVeqns}, if a quasi-smooth $(-1)$-shifted symplectic $\cL_{\infty}$-algebra $(\bL(U), \omega)$ had a choice of potential $(\sigma, S)$ so that $(d_{\dR} + q )( S + \sigma) = \omega$ with $\sigma \in (L^{2})^{\vee} \otimes \Omega_{U}^{1}$, then we have a commutative diagram:
	
	\[
	\begin{tikzcd}
		\bL^{2}_{U} \arrow[r, "\sigma"] & \Omega_{U} \\
		U \arrow[u, "\lambda"] \arrow[ur, "d_{dR}S", swap] 
	\end{tikzcd}
	\text{ \big( where } \Omega_{U} = (\bL^{1}_{U})^{\vee} \text{ \big)}
	\]
	
	Furthermore, the third equation in \ref{BVeqns} tells us that $S$ is invariant under the infinitesimal gauge action of $\g = \bL^{0}_{U}$ on $U$. The (homological) non-degeneracy of $d_{\dR} \sigma = \omega$ implies that the induced morphism (extended by $0$ in the other degrees)

	\[
	\sigma_{\mu}: L_{\mu}^{\bullet} \rightarrow (L_{\mu}^{\bullet})^{\vee}[-1] 
	\]
	
	is a quasi-isomorphism. In other words,\\

	\begin{equation} \label{inducedHCLdiagram2} 
		\begin{tikzcd}
			\cdots L^{0}_{\mu} \arrow[r,"d_{\mu}"] \arrow[d, "0"] & L^{1}_{\mu} \arrow[d, "\sigma_{\mu}^{\vee}"] \arrow[r, "d_{\mu}"] &  L^{2}_{\mu} \arrow[r, "d_{\mu}"] \arrow[d, "\sigma_{\mu}"] &  L^{3}_{\mu}  \arrow[d, "0"] \cdots  \\
			\cdots (L^{3}_{\mu})^{\vee} \arrow[r,"d^{\vee}_{\mu}"] & (L^{2}_{\mu})^{\vee} \arrow[r, "d^{\vee}_{\mu}"] &  (L^{1}_{\mu})^{\vee} \arrow[r, "d_{\mu}^{\vee}"] &  (L^{0}_{\mu})^{\vee} \cdots  \\
		\end{tikzcd}
	\end{equation}
	
	is a quasi-isomorphism of the rows.\\

	We encapsulate these properties by making the definition: 
	
	\begin{defn} (BV $\cL_{\infty}$-algebra)\\
		Suppose that $\bL(U)$ is a quasi-smooth $\cL_{\infty}$-algebra. We say that $(L,U, S ,\sigma)$ defines a BV $\cL_{\infty}$-algebra if: $S: U \rightarrow \CC$ is an analytic function, and $\sigma: \bL^{2}_{U} \rightarrow \Omega^{1}_{U}$ is a morphism of bundles such that

		\[
		\begin{tikzcd}
			\bL^{2}_{U} \arrow[r, "\sigma"] & \Omega_{U}^{1} \\
			U \arrow[u, "\lambda_{MC}"]  \arrow[ur, "d_{dR}S", swap] & \\
		\end{tikzcd}
		\]
		commutes, and  for all $\mu \in X = \tMC(L,U)$:
		
		\begin{enumerate}
			\item{$q_{\mu} \left( d_{dR}S \contract( - )  \right) = d_{dR}S \circ d_{\mu}(-) = 0 $, in other words $S$ is gauge invariant. }
			\item{ The induced commutative diagram: 
				
				\begin{equation} \label{inducedHCLdiagram2} 
					\begin{tikzcd}
						\cdots L^{0}_{\mu} \arrow[r,"d_{\mu}"] \arrow[d, "0"] & L^{1}_{\mu} \arrow[d, "\sigma_{\mu}^{\vee}"] \arrow[r, "d_{\mu}"] &  L^{2}_{\mu} \arrow[r, "d_{\mu}"] \arrow[d, "\sigma_{\mu}"] &  L^{3}_{\mu}  \arrow[d, "0"] \cdots  \\
						\cdots (L^{3}_{\mu})^{\vee} \arrow[r,"d^{\vee}_{\mu}"] & (L^{2}_{\mu})^{\vee} \arrow[r, "d^{\vee}_{\mu}"] &  (L^{1}_{\mu})^{\vee} \arrow[r, "d_{\mu}^{\vee}"] &  (L^{0}_{\mu})^{\vee} \cdots  \\
					\end{tikzcd}
				\end{equation}
				
				is a quasi-isomorphism of the rows. }
		\end{enumerate}
	\end{defn}

	Now, an important property is that homotopy BV data pulls back along quasi-isomorphisms between $\cL_{\infty}$-algebras. That is,
	
	\begin{defn}
		Suppose that $\big( L,U , S , \sigma \big)$ is a BV $\cL_{\infty}$-algebra. Furthermore, suppose that $\wtil{\bL}(V)$ is an $\cL_{\infty}$-algebra with a quasi-isomorphism $\Phi: \wtil{\bL}(V) \rightarrow \bL(U)$.\\
		
		Then, we define $\Phi^{*}\sigma$ via the diagram:
		
		\[
		\begin{tikzcd}
			\bL^{2}_{U} \arrow[r,"\sigma"] & (\bL^{1}_{U})^{\vee} \\
			\wtil{\bL}_{ V }^{2} \arrow[u," \phi^{\#}"] \arrow[r, "\Phi^{*}\sigma" swap, dotted] & (\wtil{\bL}_{V}^{1})^{\vee} \arrow[u, "\phi^{*}" swap] 
		\end{tikzcd}
		\]
	\end{defn}
	
	\begin{lem}
		Suppose that $\big( L,U , S , \sigma \big)$ is a BV $\cL_{\infty}$-algebra. Furthermore, suppose that $\wtil{\bL}(V)$ is an $\cL_{\infty}$-algebra with a quasi-isomorphism $\Phi: \wtil{\bL}(V) \rightarrow \bL(U)$. Then, the pair $\Phi^{*}(S,\sigma) = ( \phi^{*}S , \phi^{*} \sigma) $ defines homotopy BV data for the $\cL_{\infty}$-algebra $\wtil{\bL}(V)$.
	\end{lem}
	
	\begin{proof}
		The point is to check that for any $\ups \in \tMC(\wtil{L},V)$, the diagram
		
		\[
		\begin{tikzcd}
			0 \arrow[r] & \wtil{L}_{\ups}^{1} \arrow[r, "d_{\ups}"] \arrow[d, "\phi^{*} \sigma^{\vee}_{\ups}" swap] &  \wtil{L}_{\ups}^{2}  \arrow[d, "\phi^{*} \sigma_{\ups}" ]  \arrow[r]  & 0  \\ 
			0 \arrow[r] & (\wtil{L}_{\ups}^{2} )^{\vee}          \arrow[r, "(d_{\ups})^{\vee}" , swap] 	&  (\wtil{L}_{\ups}^{1})^{\vee} \arrow[r] &  0
		\end{tikzcd}
		\]
		is a quasi-isomorphism of the rows. \\

		On the other hand, we have the diagram:
		
		\[
		\begin{tikzcd} [sep = huge]
			0 \arrow[r] & 	\wtil{L}_{\ups}^{1} \arrow[r, "d_{\ups}"] \arrow[ddd, bend right = 60, "\phi^{*} \sigma^{\vee}_{\ups}" swap] \arrow[d, "\phi^{1}_{\ups}" swap] &  \wtil{L}_{\ups}^{2}  \arrow[d, "\phi^{1}_{\ups}" ] \arrow[ddd, bend left = 60, "\phi^{*} \sigma_{\ups}" ]  \arrow[r] & 0  \\
			0 \arrow[r]	&	L_{\phi(\ups)}  \arrow[r, "d_{\phi(\ups) }" swap ]  \arrow[d, "\sigma^{\vee}_{\phi(\ups)}" swap ]& L^{2}_{\phi(\ups)} \arrow[d,"\sigma_{\phi(\ups)}"] \arrow[r] & 0  \\
			0 \arrow[r] & (L^{2}_{\phi(\ups)})^{\vee} \arrow[r, "d_{\phi(\ups)}^{\vee}"] \arrow[d,"(\phi_{\ups}^{1})^{\vee}"]  & (L^{1}_{\phi(\ups)})^{\vee}  \arrow[d, "(\phi_{\ups}^{1})^{\vee}" swap]  \arrow[r] & 0 \\
			0 \arrow[r] &	(\wtil{L}_{\ups}^{2})^{\vee} \arrow[r, " d_{\ups}^{\vee}" swap]  & 
			(\wtil{L}_{\ups}^{1})^{\vee}  \arrow[r] & 0
		\end{tikzcd}
		\]
		
		Then, as $\Phi$ is a quasi-isomorphism, and $(S, \sigma)$ is a BV pair for $\bL(U)$, the above diagram depicts quasi-isomorphisms between consecutive rows. Therefore, the two rows are also quasi-isomorphic via the map $(\phi^{*}\sigma^{\vee}_{\ups}, \phi^{*}\sigma_{\ups})$. 
	\end{proof}

	Although it may not be explicitly required for us to ever cite, the following proposition will be of interest.
	
	\begin{prop}
		Suppose that $(L, U, \omega)$ is a $(-1)$-shifted symplectic $\cL_{\infty}$-algebra. Then, there exists a $(d_{dR} + q)$-lift $(S, \sigma)$ of $\omega$ such that $(d_{dR} + q ) ( S + \sigma ) = \omega$ and $\sigma \in  (\bL_{U}^{2})^{\vee} \otimes \Omega^{1}_{U}$
	\end{prop}

	\subsubsection{BV $\cL_{\infty}$-algebras as ``homotopy critical loci"}
	
	Note that if $(L,U, S, \sigma)$ is a BV $\cL_{\infty}$-algebra, then the equation $\sigma \circ \lambda = dS$ implies that $\tMC(L,U) \subset \Crit(S)$. If $\sigma: \bL^{1}_{U} \rightarrow (\bL^{2}_{U})^{\vee}$ is an \emph{isomorphism} of bundles, then $\bL(U) \cong \dCrit(S)$ as dg analytic spaces -- however, in general, the equality $\tMC(L,U) \subsetneq \Crit(S)$ may be a strict inclusion.  \\

	As an illustrative example of what may happen, consider the following:
	\begin{ex}
		Suppose that $U \subset \CC^{n}$ is a polydisc centered at the origin, and $S:U \rightarrow \CC$ is an analytic function, with a critical point at the origin.\\
		Then, we define the \qs space $\dCrit(S) \times \CC^{m}$ through ``extending $\dCrit(S)$ by a contractible factor" -- that is, $\dCrit(S) \times \CC^{m} = (U \times \CC^{m}, \Omega_{U} \times \ul{\CC}^{m}, dS \oplus \Delta  )$, where $\Delta( x ) = (x,x)$ is the diagonal section of $\ul{\CC}^{m} = \CC^{m} \times \CC^{m}$.

		\begin{enumerate}
			\item{The pair $( \pi^{*}_{U}S, (i_{U} \pi_{U})^{*})$ defines homotopy BV data for $\dCrit(S) \times \CC^{m}$}
			\item{ We have a strict inclusion $\tau^{0} \big( \dCrit(S) \times \CC^{m} \big) \hookrightarrow \Crit( \pi^{*}_{U} S )$, so that $\tau^{0}(\dCrit(S) \times \CC^{m}) \neq \Crit( \pi^{*}_{U} S )$}
		\end{enumerate}
		
	\end{ex}

	\subsubsection*{BV data and minimal models}
	
	If $(L, U,S, \sigma)$ is a BV $\cL_{\infty}$-algebra in degrees $[1,2]$, then for any $\mu \in \tMC(L,U)$ the morphism $(\sigma_{\mu}^{\vee}, \sigma_{\mu})$ is a quasi-isomorphism in the diagram:
	
	\begin{equation} \label{inducedHCLdiagram1} 
		\begin{tikzcd}
			L^{1}_{\mu}  \arrow[r, "d_{\mu}"] \arrow[d, "\sigma_{\mu}^{\vee}"] & L^{2}_{\mu} \arrow[d, "\sigma_{\mu}" ]  \\ 
			(L^{2}_{\mu})^{\vee}           \arrow[r, "d_{\mu}^{\vee}" , swap] 	&     (L_{\mu}^{1})^{\vee} 
		\end{tikzcd}
	\end{equation}

	Therefore, we see that if $\bL(U)$ is minimal at $\mu$, so that $d_{\mu} = 0$, then the exact square \ref{inducedHCLdiagram1} implies that $\sigma_{\mu}$ is an isomorphism. This implies that $\sigma: \bL^{2}_{U} \rightarrow (\bL_{U}^{1})^{\vee}$ is an isomorphism when restricted to an open neighbourhood of $\mu$. \\

	We will see that in the situation relevant to us for Chern-Simons theory, $X$ may not be a critical locus -- \textbf{even locally}. It is possible that $X$ is locally a critical locus only after "dividing out" by gauge equivalences.\\
	
	If $L$ has a non-trivial $L^{0}$ term, then the condition $qS = 0$ is no longer vaccuous. This condition says that 
	
	\[
	qS = \delta_{0}^{\vee}(d_{dR}S) = 0 
	\]
	in other words, $d_{dR} S$ is annihilated by contraction with the anchor map $\delta_{0}: \bL_{U}^{0} \hookrightarrow \bL_{U}^{1}$ -- or equivalently, $S$ is gauge invariant.

	Suppose that $\big(L,U , S, \sigma, \Phi \big)$ is a BV $ \cL_{\infty}$-algebra, with a homotopy minimal atlas $\Phi$ for the $\cL_{\infty}$-algebra $\mbf{L}(U)$. Then, recall that this gives the following:

	\begin{enumerate}
		\item{For each $\mu \in X$, a local quasi-smooth minimal model $I_{\mu/\g}: \mbf{H}_{\mu}(U_{\mu/\g}^{\min})  \hookrightarrow  \mbf{L}_{\mu}(U_{\mu})$, so that $\tau^{0}(  \mbf{H}_{\mu}(U_{\mu/\g}^{\min})   ) = X_{\mu/\g}$ is a local slice for the homotopy $\gG$-action on $X = \tMC(L,U)$. }
		\item{A trivial polydisc bundle $\pi_{\mu/\g}: U_{\mu} \rightarrow U_{\mu/\g}^{\min}$ over the body of the minimal model $U_{\mu/\g}^{\min}$}
	\end{enumerate}
	
	Recall that we assume throughout this paper that our $\cL_{\infty}$-algebras are all quasi-smooth. Then, a minimal model $I_{\mu/\g}: \mbf{H}_{\mu}(U_{\mu/\g}^{\min}) \hookrightarrow \bL(U)$ for some $\mu \in \tMC(L,U)$ is so that $\mbf{H}_{\mu}(U_{\mu/\g}^{\min})$ is a \emph{strictly} quasi-smooth $\cL_{\infty}$-algebra in degrees $[1,2]$. Therefore,  considering the pullback BV data $I_{\mu/\g}^{*}(S, \sigma) =: (S_{\mu/\g}^{\min}, \sigma_{\mu/\g}^{\min} )$, we get that $\sigma_{\mu/\g}^{\min}: \mbf{H}_{U_{\mu/\g}^{\min}}^{2} \rightarrow  (\mbf{H}_{U_{\mu/\g}^{\min}}^{1})^{\vee}$ is an isomorphism after possibly shrinking $U_{\mu/\g}^{\min}$ further around the basepoint (given by the origin).

	To summarize the above discussion, we have:
	
	\begin{prop}
		Suppose that $\big( \mbf{L}(U), S, \sigma, \Phi \big)$ is a quasi-smooth BV $\cL_{\infty}$-algebra, with a homotopy minimal atlas $\Phi$. Then, for any $\mu \in X = \tMC(L,U)$, there exists an open neighbourhood $X_{\mu}$ of $X$, containing a slice $X_{\mu/\g}$ at $\mu$ of the gauge action on $X$, and a closed embedding $\iota_{\mu/\g}: X_{\mu/\g} \hookrightarrow U_{\mu/\g}^{\min}$, where $U_{\mu/\g}^{\min}$ is a polydisc with an analytic function $S_{\mu/\g}^{\min}: U_{\mu/\g}^{\min} \rightarrow \CC$, such that $\iota_{\mu/\g}( X_{\mu/\g} ) = \Crit(S_{\mu/\g}^{\min})$. 	 
	\end{prop}

	For $\cL_{\infty}^{\Omega}$-algebras with BV data, we can extend the notion of gauge invariance of $(S,\sigma)$ to invariance under the extended gauge action. 
	
	\begin{defn}
		Suppose that $\big( L \mdl L_{\Omega} , U \big)$ is an $\cL_{\infty}^{\Omega}$-algebra. We say that $\big( L \mdl L_{\Omega} , U, S, \sigma \big)$ defines a BV $\cL_{\infty}^{\Omega}$-algebra if:

		\begin{enumerate}
			\item{$\big( L \mdl L_{\Omega} , U \big)$ is an $\cL_{\infty}^{\Omega}$-algebra, and $(L, U, S , \sigma )$ defines a BV $\cL_{\infty}$-algebra }
			\item{There exists a BV pair $(S_{\Omega}, \sigma_{\Omega})$ for $(L_{\Omega}, U_{\Omega})$ such that $\mbf{I}^{*}(S_{\Omega}, \sigma_{\Omega}) = (S,\sigma)$, and both $\exp(S_{\Omega})$ and $\sigma_{\Omega}$ are invariant under the (extended) gauge group $\gG = \gG_{\Omega}$.}
		\end{enumerate}
 
	\end{defn}

	In particular, for a quasi-smooth BV $\cL_{\infty}^{\Omega}$-algebra $(L \mdl L_{\Omega}, U , S, \sigma, \Phi)$, any $\mu \in \tMC(L,U)$ and $g \in \gG$, we have that the following diagram commutes:
	
	\[
	\begin{tikzcd}[sep = huge]
		H_{g \star \mu}^{2} \arrow[r, "\sigma^{\min}_{g \star \mu}"] & (H_{g \star \mu}^{1})^{\vee} \arrow[d, "\calD \Phi_{(g,\mu)}^{\min}"] \\
		H_{\mu}^{2}  \arrow[u, "( \calD \Phi_{(g,\mu)}^{\min})^{*}" ] \arrow[r, "\sigma^{\min}_{\mu} "] & (H_{\mu}^{1})^{\vee}  
	\end{tikzcd}
	\]

	\subsubsection{BV data and homotopy minimal models}
	Now, we will be interested in the following question: 
	
	\paragraph*{(*)} For a BV $\cL^{\Omega}_{\infty}$-algebra $\big( L \mdl L_{\Omega}, U, S, \sigma, \Phi \big)$ equipped with a homotopy minimal atlas $\Phi$, if $X = \tMC(L,U)$ can be covered by slices $X_{\mu/\g}$ of the gauge action on $X$, such that there exist various closed embeddings $\iota_{\mu/\g}(X_{\mu/\g}) = S_{\mu/\g}^{\min}$ with $\iota_{\mu/\g}(X_{\mu/\g}) = \Crit(S_{\mu/\g}^{\min})$ for some analytic $S_{\mu/\g}^{\min}: U_{\mu/\g}^{\min} \rightarrow \CC$, how are the minimal critical actions $S_{\mu/\g}^{\min}$ related to each other for different $\mu$?\\

	In the forthcoming sequel to this paper, we will show that if a BV $\cL^{\Omega}_{\infty}$-algebra is \emph{oriented}, then there is a certain perverse sheaf that glues together the locally defined perverse sheaves of vanishing cycles for each locally defined minimal critical action $S_{\mu/\g}^{\min}$.

	Suppose that $\big( L \mdl L_{\Omega}, U, S, \sigma, \Phi \big)$ is a BV $\cL^{\Omega}_{\infty}$-algebra, with a homotopy minimal atlas $\Phi$. 
	Then, for any choice of Morse-Thom splitting  $U_{\mu} \cong U_{\mu}^{(2)} \times U_{\mu}^{(\geq 3)}$ for $\mu \in \tMC(L,U)$, so that 
	\[
	S_{\mu} \sim S_{\mu}^{(2)} \boxplus S_{\mu}^{(\geq 3)}
	\] 
	Then, one must have that 
	
	\begin{enumerate}
		\item{ $U_{\mu/\g}^{\min} \subset U_{\mu}^{(\geq 3)}$}
		\item{ $U_{\mu}^{(2)} \subset N_{\mu}$   }
	\end{enumerate}

	Now, suppose that $\ups \preceq_{\Phi} \mu$ and $d_{\ups} \lneq d_{\mu}$, where recall that $d_{\mu} = \ed_{\mu} X^{\red} $. We define $d_{\ups/\mu} = d_{\mu} - d_{\ups}$. As $\ups \preceq_{\Phi} \mu$, we have an injective quasi-isomorphism of $L_{\infty}$-algebras, where $\ups_{\mu} = \Phi_{\ups, \mu}(\0)$:
	
	\[
	I_{\ups/\mu} : \bH_{\ups}(U_{\ups/\g}^{\min}) \hookrightarrow \bH_{\mu}^{\ups_{\mu}}(U_{\mu/\g}^{\ups}) 
	\]
	
	Then, we have that $d_{\ups/\mu} = \rk( d^{2}_{\ups_{\mu}} S_{\mu/\g}^{\min} )$, where $d^{2}_{\ups_{\mu}} S_{\mu/\g}^{\min} $ is the Hessian of $S_{\mu/\g}^{\min}$ at $\ups_{\mu}$. Let us write $S^{(2)}_{\ups/\mu} := d^{2}_{\ups_{\mu}} S_{\mu/\g}^{\min}$, and the $1$-form  $d_{dR} \qQ_{\ups/\mu} : N_{\ups/\mu} \iso N_{\ups/\mu}^{\vee}$ defines a linear contractible $L_{\infty}$-algebra $(N_{\ups/\mu} \oplus N_{\ups/\mu}^{\vee}, d_{dR} \qQ_{\ups/\mu} , 0 ,\cdots ) = \dCrit(\qQ_{\ups/\mu}) = \nN_{\ups/\mu}$. By \ref{Linftylift}, we obtain an $L_{\infty}$-isomorphism
	
	\begin{equation} \label{MorseThomSplit}
		\Phi_{\ups/\mu}: \bH_{\ups}(U_{\ups/\g}^{\min}) \times \dCrit(S^{(2)}_{\ups/\mu})  \longiso \bH_{\mu}^{\ups_{\mu}}(U_{\mu/\g}^{\ups})
	\end{equation}
	
	which is a minimal model decomposition $H_{\ups} \times \dCrit(S^{(2)}_{\ups/\mu}) \longiso H_{\mu}^{\ups_{\mu}}$. \\

	\begin{defn}
		Suppose that $\big( L \mdl L_{\Omega}, U, S, \sigma, \Phi \big)$ is a BV $\cL^{\Omega}_{\infty}$-algebra, with a homotopy minimal atlas $\Phi$. We say that $\Phi$ is $S$-adapted if for each $\mu \in \tMC(L,U)$, we have a choice of Morse-Thom splitting so that for each $\ups \preceq_{\Phi} \mu$, we have an isomorphism  $\Phi_{\ups/\mu}: \bH_{\ups}(U_{\ups/\g}^{\min}) \times \dCrit(S^{(2)}_{\ups/\mu})  \longiso \bH_{\mu}^{\ups_{\mu}}(U_{\mu/\g}^{\ups})$ as in \ref{MorseThomSplit}.
	\end{defn}

			Note that $U_{\mu} \cong U_{\mu}^{(2)} \times U_{\mu}^{(\geq 3)}$ is a linear splitting of polydiscs, so that $U_{\mu}^{(2)}$ comes with a standard choice of basis and a coordinate volume form $\dVol_{U_{\mu}^{(2)}}$. Furthermore, as $S_{\mu}^{(2)}: U_{\mu}^{(2)} \rightarrow \CC$ is non-degenerate, we have that $\Crit(S_{\mu/\g}^{\min}) \cong \Crit(S_{\mu}^{(2)} \boxplus S_{\mu/\g}^{\min})$. 
			
			\begin{defn}
				If one has a BV $\cL_{\infty}$-algebra  $(\bL(U), S, \sigma, \Phi)$ with an $S$-adapted homotopy minimal atlas $\Phi$,  we define:
				
				\[
				S_{\mu}^{(2)} \boxplus S_{\mu/\g}^{\min} = S_{\mu}^{\eff} :  U_{\mu}^{(2)} \times U_{\mu/\g}^{\min} \longrightarrow \CC
				\]
			\end{defn}
			
			By gauge-invariance of $S$ modulo $\ZZ$, if $\ups \preceq_{\Phi} \mu$, then we have an induced commutative diagram:
			
			\begin{equation} \label{gaugeeffective}
				\begin{tikzcd} [sep=large]
					U_{\ups}^{(2)} \times U_{\ups/\g}^{\min} \arrow[d] \arrow[r, "\Phi_{\ups,\ups_{\mu}}^{\eff}","\sim" swap, dotted] & 	U_{\ups_{\mu}}^{(2)} \times U_{\ups_{\mu}/\g}^{\min} \arrow[d]\\
					U_{\ups/\g}^{\min} \arrow[r, "\Phi_{\ups,\ups_{\mu} }^{\min}", "\sim" swap  ] & U_{\ups_{\mu}/\g}^{\min} 
				\end{tikzcd} 
			\end{equation}
			such that 
			
			\[
			d_{dR} (\Phi_{\ups,\ups_{\mu}}^{\eff})^{*}(S_{\ups_{\mu}}^{\eff} ) = d_{dR} S_{\ups}^{\eff}
			\]

			\section{Orientations and metric structures}

			In this section, we consider the notation of orientation for BV $\cL_{\infty}^{\Omega}$-algebras. 
			
			\begin{defn}
				Suppose that $(L \mdl L_{\Omega}, U, S, \sigma, \Phi)$ is a quasi-smooth BV $\cL_{\infty}^{\Omega}$-algebra with a gauge-fixing atlas $\Phi$. Then a metric structure on $(L \mdl L_{\Omega}, U, S, \sigma)$ is the data of:
				
				\begin{enumerate}
					\item{A $\gG$-invariant non-degenerate quadratic form $\qQ_{L}$ on $L^{1}$, such that $\qQ_{L} \resto_{ \tMC(L,U) } = 0$}
					\item{For each $\mu \in \tMC(L,U)$, $\qQ_{\mu} := \ii_{\mu}^{*} \qQ_{L} = \qQ_{L} \resto_{ H^{1}_{\mu}}$ is a non-degenerate quadratic form on $H^{1}_{\mu}$}
					\item{For each $\ups \preceq_{\Phi} \mu$, one has that 
						\[
						J_{\phi_{\ups/\mu}}^{*}( \qQ_{\mu} ) = \qQ_{\ups/\mu} \boxplus \qQ_{\ups}
						\]
						for some non-degenerate quadratic form $\qQ_{\ups/\mu}$ on $U_{\ups/\mu}^{(2)}$, where
						$\phi_{\ups/\mu}: U_{\ups/\mu}^{(2)} \times U_{\ups/\g}^{\min} \longiso    U_{\mu/\g}^{\ups}$  
					}
				\end{enumerate}
				
				Note that by $1.$, we have that for $\mu, \mu' \in \tMC(L,U)$ with $\mu \sim_{\gG} \mu'$ , then we have that the compatibility relation
				\[
				(J_{\phi_{\mu,\mu'}^{\min}})^{*}( \qQ_{\mu'} ) = \qQ_{\mu} 
				\]
			\end{defn}
	
			We say that the metric structure $\qQ_{L}$ is \emph{oriented} if we fix an orientation of $\det(\qQ_{L})$. Recall that the determinant of a non-degenerate quadratic form is only well-defined up to a square, so fixing an orientation is to fix a choice of square-root. As one can always choose $+ \sqrt{\det(\qQ_{L})}$  or $-\sqrt{\det(\qQ_{L})}$, which is equivalent to choosing a sign. \\
			Note that if we fix an orientation of $\det(\qQ_{L})$, then there are well-defined choices of square roots of $\det(\qQ_{\mu})$ and $\det(\qQ_{\ups/\mu})$ for any $\mu$ and pairs $\ups \preceq_{\Phi} \mu$ in $\tMC(L,U)$.

			\begin{defn}
				Suppose that  $(L \mdl L_{\Omega}, U, S, \sigma, \Phi)$ is a gauge-fixed BV $\cL_{\infty}^{\Omega}$-algebra with oriented metric structure $\qQ_{L}$. Then, for each minimal chart $U_{\mu/\g}^{\min} \hookrightarrow U_{\mu}$, we define the \emph{metric volume form} on $U_{\mu/\g}^{\min}$ to be
				
				\[
				\varOmega_{\mu}  := \sqrt{ \det( \qQ_{\mu} ) } \cdot \dVol_{U_{\mu/\g}^{\min}}
				\]	
			\end{defn}
			
			For each $\ups \preceq_{\Phi} \mu$, we define 
			
			\[
			\varOmega_{\ups/\mu} := \sqrt{  \det(\qQ_{\ups/\mu} ) } \cdot \dVol_{U_{\ups/\mu}^{(2)}} 
			\]

			\begin{lem}
				For each $\ups \preceq_{\Phi} \mu$, one has that: 
				\[
				\phi_{\ups/\mu}^{*}( \varOmega_{\mu} ) = \varOmega_{\ups/\mu} \wedge \varOmega_{\ups} 
				\]
			\end{lem}
			
			\begin{proof}
				Indeed, by the property that $J^{*}_{\phi_{\ups/\mu} } \qQ_{\mu} = \qQ_{\ups/\mu} \boxplus \qQ_{\ups}$, we have that $\det(J^{*}_{\phi_{\ups/\mu} }\qQ_{\mu} ) = \det(J_{\phi_{\ups/\mu} } )^{2} \det( \qQ_{\mu} ) =  \det( \qQ_{\ups/\mu}  ) \cdot \det( \qQ_{\ups} )$. From this, we see that
				
				\[
				\det(J_{\phi_{\ups/\mu}} ) =  \frac{ \sqrt{  \det(\qQ_{\ups/\mu} ) } \cdot \sqrt{ \det( \qQ_{\ups } ) }  } { \sqrt{\det( \qQ_{\mu} )}  }
				\]
				
				Therefore, 
				
				\begin{align*}
					& \phi_{\ups/\mu}^{*}( \varOmega_{\mu} )  = \phi_{\ups/\mu}^{*}\big(   \sqrt{ \det( \qQ_{\mu} ) } \cdot \dVol_{U_{\mu/\g}^{\min}}    \big) \\  
					& =  \sqrt{ \det( \qQ_{\mu} ) } \cdot \frac{ \sqrt{  \det(\qQ_{\ups/\mu} ) } \cdot \sqrt{ \det( \qQ_{\ups } ) }  } { \sqrt{\det( \qQ_{\mu} )}  }   \dVol_{U_{\ups/\mu}^{(2)} } \wedge \dVol_{U_{\ups/\g}^{\min} } \\
					& = \varOmega_{\ups/\mu} \wedge \varOmega_{\ups} 
				\end{align*}
			\end{proof}

			\begin{defn} (Quantum BV $\cL_{\infty}^{\Omega}$-algebras)\\
				We say that $(L \mdl L_{\Omega}, U, S, \sigma, \Phi, \qQ_{L})$  is an oriented metric BV $\cL_{\infty}^{\Omega}$-algebra, or a {\bfseries quantum BV $\cL_{\infty}^{\Omega}$-algebra}, if $(L \mdl L_{\Omega}, U, S, \sigma, \Phi)$ is a gauge-fixed BV $\cL_{\infty}^{\Omega}$-algebra, with a choice of oriented metric structure $\qQ_{L}$. As we will only use this stronger form of orientation (i.e. with an oriented metric structure) in this paper, we may simply say that an oriented metric BV $\cL_{\infty}^{\Omega}$-algebra is an \emph{oriented} BV $\cL_{\infty}^{\Omega}$-algebra. \\
				
				We will refer to the pair $(\Phi, \qQ_{L})$ as the \emph{orientation data} associated to the oriented BV $\cL_{\infty}^{\Omega}$-algebra  $(L \mdl L_{\Omega}, U, S, \sigma, \Phi, \qQ_{L})$.
			\end{defn}

			\subsection{Relation to other forms of orientation}
			
			\begin{defn}
				Suppose that $(L,U)$ (resp. $(L \mdl L_{\Omega}, U)$) is an $\cL_{\infty}$-algebra (resp. $\cL_{\infty}^{\Omega}$-algebra), then we define the $\cL_{\infty}$-orientation line bundle over $X = \tMC(L,U)$  to be given by 
				
				\[
				\det(L,U) := \det(\bL_{U}^{\bullet}) \resto_{X} 
				\]
				So that for each $\mu \in \tMC(L,U)$, we have that $\det(L,U)_{\mu} = \det( [L^{0}_{\mu} \rightarrow L^{1}_{\mu} \rightarrow \cdots ]  ) = \bigotimes  \det(L^{i}_{\mu})^{(-1)^{i}}$. In the case that $(L \mdl L_{\Omega}, U)$ is a $\cL_{\infty}^{\Omega}$-algebra, $\det(L,U)$ is $\gG$-invariant and descends to a line bundle $\det_{\gG}(L,U)$ over $X \git \gG$. \\
				
				We say that $(L,U)$ is orientable (resp. $(L \mdl L_{\Omega}, U)$ is orientable), if there exists a global section of the line bundle $\det(L,U)$ (resp. $\det_{\gG}(L,U)$).
			\end{defn}
			
			\begin{defn} \label{BVorient}
				If $(L \mdl L_{\Omega}, U, S, \sigma)$ is a BV $\cL_{\infty}^{\Omega}$-algebra, then we say that $(L \mdl L_{\Omega}, U, S, \sigma)$ is BV orientable if there exists a square root line bundle $\det_{\gG}(L,U)^{1/2}$ over $X \git \gG$, of the orientation bundle $\det_{\gG}(L,U)$, and a global section of $\det_{\gG}(L,U)^{1/2}$. 
			\end{defn}
			
			\begin{lem}
				If  $(L \mdl L_{\Omega}, U, S, \sigma, \Phi, \qQ_{L})$  is an orientable metric BV $\cL_{\infty}^{\Omega}$-algebra, then it is BV orientable in the sense of \ref{BVorient}
			\end{lem}

\bibliographystyle{plain}
\bibliography{quantumBVLinfinity}

\end{document}